\documentclass[11pt,reqno,hidelinks]{amsart}
\usepackage{amsthm, amsmath,amsfonts,amssymb,euscript,hyperref,graphics,color}
\usepackage{graphicx}
\usepackage[square, comma, sort&compress, numbers]{natbib}
\usepackage{subfigure}
\usepackage{comment}
\usepackage{import}
\usepackage{tikz}
\usepackage{latexsym}

\usepackage{ulem}
\usepackage{todonotes}
\newcommand{\C}[1]{\textcolor{red}{#1}}

\makeatother
\newtheorem{theorem}{Theorem}[section]
\newtheorem{lemma}[theorem]{Lemma}

\newtheorem{proposition}[theorem]{Proposition}

\theoremstyle{definition}

\newtheorem{assumption}[theorem]{Assumption}

\theoremstyle{remark}
\newtheorem{remark}[theorem]{Remark}

\newcommand \bel {\be\label}
\newcommand \be {\begin{equation}}
\newcommand \ee {\end{equation}}
\newcommand \bes {\begin{equation*}}
\newcommand \ees {\end{equation*}}
\numberwithin{equation}{section}

\def\ub{\underline{u}}
\def\Lb{\underline{L}}
\def\Cb{\underline{C}}
\def\Eb{\underline{E}}
\def\Fb{\underline{F}}

\def\ub{\underline{u}}

\def\Lb{\underline{L}}
\def\Cb{\underline{C}}
\def\Eb{\underline{E}}
\def\Fb{\underline{F}}

\def\Lambdab{\underline{\Lambda}}

\def\fb{\underline{f}}

\def\TLb{\tilde{\underline{L}}}
\def\TL{\tilde{L}}

\newcommand{\lie}{\mathcal{L}}
\newcommand{\Lmdu}{\Lambda(u)}
\newcommand{\Lmdub}{\Lambdab(\ub)}
\newcommand{\deyu}{\Lambda^\frac{1}{2}(u)}
\newcommand{\deyub}{\Lambdab^\frac{1}{2}(\ub)}
\newcommand{\Psip}{\Psi^\prime}
\newcommand{\ceta}{\mathring\eta}
\newcommand{\cg}{\mathring{g}}
\newcommand{\nnb}{\nonumber}

\def\p{\partial}
\newcommand{\di}{\mathrm{d}} 
\newcommand{\D}{\mathcal{D}}
\newcommand{\T}{\mathcal{T}}

\newcommand{\als}[1]{
\begin{align*}
  #1
\end{align*}}

\newcommand \bei  {\begin{itemize}}
\newcommand \eei {\end{itemize}}
\allowdisplaybreaks[4]
\begin{document}
\title[Large perturbations of plane wave]{Global stability of the plane wave solutions to the relativistic string with non-small perturbations}

\vspace{10mm}
\author[J. Wang]{Jinhua Wang}
\address{School of Mathematical Sciences, Xiamen University, Xiamen 361005, China.}\email{ wangjinhua@xmu.edu.cn}

\author[C. Wei]{Changhua Wei}
\address{Corresponding author: Department of Mathematics, Zhejiang Sci-Tech University, Hangzhou, 310018, China.}\email{chwei@zstu.edu.cn}

\date{}

\begin{abstract}
This paper is concerned with the global stability of the plane wave solutions to the relativistic string equation with non-small perturbations. Under certain decay assumptions on the plane wave, we conclude that the perturbed system admits a globally smooth solution if the perturbation along the transversal direction is sufficiently small, while the travelling direction is allowed to be large. By choosing a gauge adapted to the plane wave solution, we deduce an equivalent Euler-Lagrangian equation for the perturbation whose quasilinear structure is reflected precisely in the induced geometry of the relativistic string. It then helps to proceed a geometrically adapted and weighted energy argument for which robust estimates suffice. Moreover, due to the non-trivial background solutions, the induced metric of the relativistic string involves linear perturbations with undetermined signs, and hence a key observation is needed to guarantee that the energies associated to the multipliers are positive up to lower order terms.  
\end{abstract}

\maketitle
\tableofcontents

{\sl Key words and phrases}: relativistic string; plane wave; null condition;  large energy; global existence.

\section{Introduction}\label{sec:1}

Let $(t,x_1,x_2,\cdots,x_n,y)$ be coordinates in a $\left(1+(n+1)\right)$-dimensional Minkowski spacetime and a graph takes the form of
\bes
y=\varphi(t,x_1,\cdots,x_n).
\ees
The embedding $(t,x_1,x_2,\cdots,x_n) \rightarrow (t,x_1,x_2,\cdots,x_n, \varphi)$ has the area functional
\be\label{I1}
\varphi \mapsto \int_{\mathbb R^{1+n}}\sqrt{1+|\nabla\varphi|^2-\varphi^2_t} \di t \di x_1\cdots \di x_n,
\ee
where $\nabla\varphi = (\p_{x_1} \varphi,\cdots,\p_{x_{n}}\varphi)$, $\varphi_t = \p_t \varphi$.
The corresponding Euler-Lagrange equation reads
\be\label{I2}
\left(\frac{\varphi_t}{\sqrt{1+|\nabla\varphi|^2-\varphi_t^2}}\right)_t-\sum_{i=1}^{n}\left(\frac{\varphi_{x_i}}{\sqrt{1+|\nabla\varphi|^2-\varphi_t^2}}\right)_{x_i}=0.
\ee
The equation \eqref{I2} is called the relativistic membrane equation when $n\geq 2$, and the relativistic string equation when $n=1$. We refer to \cite{Hoppe1} for a review on the related topics. Note that, the relativistic string equation can be transformed into a quasilinear hyperbolic system whose characteristics are totally linearly degenerate in the sense of P. D. Lax \cite{Lax} (referring to \cite{Wang-Wei1} for more details). On the other hand, \eqref{I2} itself is a quasilinear wave equation that is highly non-resonant, and its nonlinearity satisfies the ``double null condition'' of Klainerman \cite{Klainerman}.

The Cauchy problem of the relativistic string equation had been considered in \cite{Wang-Wei1}, where we constructed a globally smooth solution with non-small energy.
However, it is unable to tell whether this solution can be interpreted as a large perturbation of the plane wave solution or not. This serves as the main motivation of the current paper, where we intend to demonstrate the global existence of non-small perturbations of plane wave solutions, and consequently confirm that the solution obtained in \cite{Wang-Wei1} involves non-small perturbations of the plane wave solutions as well. Therefore, the key point in this paper is different from the ones on small perturbations of plane wave solutions which essentially boils down to a small data problem.

Before the statement of our main result, let us introduce some necessary works in regard to the relativistic membrane and string. Since a wealth of previous results are stated clearly in \cite{Wang-Wei1}, we here give a brief description for the self-containedness of this paper.
\subsection{Previous results}
Due to the properties of ``linearly degenerate characteristics'' and ``double null structure'' stated above, fruitful results on both of the global existence and blow-up of the solutions to the relativistic membrane and string equation have been established. For the small data Cauchy problem, Lindblad \cite{Lindblad} obtained globally smooth solutions and proved the decay rate $t^{-\frac{n-1}{2}}$, $n\geq 1$. The $n=1$ (relativsitic string) result in \cite{Lindblad} was later extended by Wong \cite{Wong}. Chae and Hun \cite{Chae-Hun} studied the more general case of Born-Infeld equations, and concluded the global existence as well. 
As is known to all, in general,  the one dimensional linear wave does not decay. Thus, we here mention an interesting result of Alejo and Mu$\tilde{n}$oz \cite{Alejo}, where they investigated the decay of small solutions to the Born-Infeld equation in $1+1$ dimension and showed that all globally small $H^{s+1}\times H^s$, $s\geq \frac{1}{2}$, solutions decay to the zero background state in space, inside a strictly subset of the light cone.
Alternatively, a decay mechanism due to the spatial separations of two families of wave packet (the right-travelling and left-travelling ones) after a sufficiently long time is responsible to the global existence of $(1+1)$-dimensional semilinear wave equations with null structure, see Luli-Yang-Yu \cite{Luli-Yang-Yu}.

Let us turn to the situation with general data. Initially, there were several works by Hoppe (see the review \cite{Hoppe1}) on the singularities of the relativistic membranes via the algebraic method. In addition, the type of self-similar singularities 
were constructed by Eggers-Hoppe \cite{Hoppe2}, and recently by Yan \cite{Yan}, where a stability result for these singularities were asserted as well. 
Moreover, Eggers, Hopppe, Hynek and Suramlishvilli \cite{Hoppe4} found that singularities of the relativistic membranes came from the fact that the tangent vectors become linearly dependent, i.e., near the singular time, the solutions are close to those of the eikonal equation $1-(\frac{\partial \varphi}{\p t})^2+|\nabla\varphi |^2=0$. Therefore, all global results in a large-data setting are built on a framework that makes the time-like condition $1-(\frac{\partial \varphi}{\p t})^2+|\nabla\varphi |^2>0$ hold true for all time, see for instance, \cite{Wang-Wei, Wang-Wei1}. Besides, Wong \cite{Wong3} found an axially symmetric singularity, which is due to the degeneration of the principal symbol of the evolution, and thus not shock-type.
When focusing on the $1+1$-dimensional case, by the approach of characteristics, Kong and Tusji \cite{Kong-Tsuji} proposed sufficient conditions and necessary conditions respectively for the globally classical solutions to the general quasilinear system with linearly degenerate characteristics (the relativistic string is included). An application of this work was given by Kong, Zhang and Zhou \cite{Kong-Zhang-Zhou} to address a sufficient and necessary condition for the global existence of the classical solution to the relativistic string. Meanwhile, the relativistic string with large data generates topological singularities that arise from the periodic initial data \cite{Nguyen-Tian, Kong-Zhang}, and ``cusp type'' singularities from smooth initial data  \cite{Kong-Wei-Zhang, Kong-Wei}. In fact, it is believed that the relativistic string equation prevents the formation of shock-type singularities. This is different from the hyperbolic equations with ``genuinely nonlinear characteristics'', for which, shocks usually form in finite time even for small and smooth data, see \cite{Lax1, John,Holzegel} and references therein.

In view of the results mentioned above, for the relativistic string or membrane equation, in principle we can not expect global solutions with general large initial data. 
The nature question arises: \textit{whether there exists a set of large data such that the relativistic string equation still has a globally smooth solution?} To the best of the authors' knowledge, the first choice is the stability of the non-trivial solution. Concretely speaking, it is not difficult to find a class of non-trivial solutions to the relativistic string equation or membrane equation, which has large energy. Then one may construct a global solution which is indeed a small perturbation to the non-trivial solution. The composition of the background solution and the perturbation provides a large solution. At the same time, the stability problem for the non-trivial solutions is itself interesting. In the following, we briefly review some related results in this aspect. When the background solution is stationary, such as catenoid,  Donninger, Krieger, Szeftel and Wong \cite{Krieger} had proved the codimension one stability in $\mathbb R^{1+3}$. They could show that there exists a function $g$ 
such that for any small perturbations $(\Phi_0,\Phi_1)$ of the catenoid, satisfying the same rotational and reflective symmetries, the vanishing mean curvature flow (or the membrane equation) with initial Cauchy data $(\Phi_0+\alpha g,\Phi_1)$ has a global solution and decays asymptotically to  the catenoid for some small real number $\alpha$.
 For the plane wave case,
Abberscia and Wong \cite{Abbrescia-Wong} proved that, under sufficiently small compactly supported perturbations, any simple planer travelling wave to the membrane equation with bounded spatial extent is globally non-linearly stable when the spatial dimension $n\geq 3$. They further showed that the perturbation converges to zero in $C^2$ norm. The $n=2$ case was  accomplished by Liu and Zhou \cite{Liu-Zhou} through adopting a gauge adapted to the trivial solution. In the $n=1$ case, Cha and Shao \cite{Cha-Shao} considered the $1+1$ dimensional nonlinear wave equations whose nonlinearity satisfies a certain asymptotic null condition, and got the stability of the travelling wave providing that the travelling wave decays sufficiently fast.

 Another attempt to study the large data problem for nonlinear wave equations is employing the short-pulse method, which was introduced by Christodoulou \cite{Christodoulou1} where he showed the formation of black holes for the vacuum Einstein equation. This result was extended and significantly simplified by Klainerman and Rodnianski \cite{Klainerman-Rodnianski}. Later, the ideas in \cite{Christodoulou1} and \cite{Klainerman-Rodnianski} had been adapted to the $(1+3)$-dimensional nonlinear wave equation with null quadratic forms \cite{Miao-Pei-Yu, Wang-Yu, Wang-Yu1} and the relativistic membrane equation with $n=2, 3$ \cite{Wang-Wei}. All of these results regarding large perturbations of the trivial background solution depend on the null structure of the corresponding equations. Motivated by these facts and the result \cite{Luli-Yang-Yu} mentioned previously, we \cite{Wang-Wei1} addressed the global existence for the relativistic string ($n=1$) within the setting that the ``right travelling wave'' was non-small while the ``left travelling wave'' had to be small enough. The solution obtained in \cite{Wang-Wei1} is a non-small perturbation of the trivial solution, moreover, this non-small solution exists in the whole domain of future development instead of a thin null strip in the context evolved by the short-pulse data.
Apart from those results, an interesting work inspired by the shock formation in fluid was conducted by Miao and Yu \cite{Miao-Yu}, where they proved for a class of variational wave equations with certain cubic nonlinearities that shock forms in finite time when the initial data set is of the short-pulse type. Actually, a different mechanism of shock formation from the ones in small data situation was detected precisely in \cite{Miao-Yu}, since the equation studied admits globally smooth solution with small data. However, when the cubic nonlinearities satisfy the double null condition, such as the relativistic membrane equation, global existence with short-pulse data is also expected, see \cite{Wang-Wei}.  

\subsection{Main result}
Before the statement of the main theorem, we briefly introduce the perturbed system of the plane wave solution of relativistic string and the associated geometry. The detailed derivations are left to Section \ref{sec-string}, for which we refer to \cite{Abbrescia-Wong} for the thorough idea.

Let $\Psi(t-x)$ be a right-travelling wave solution to \eqref{I2} with spatial dimension $n=1$.
Define the new coordinate system that is adapted to the plane wave solution,
\bes
u = t-x, \quad \ub = \frac{t+x}{2} - \frac{1}{2} \int_0^{t-x} (\Psip)^2 (\tau) \di \tau.
\ees
The equation for the perturbation $\phi :=\varphi-\Psi(u)$ is described by
\be\label{main-equation}
\Box_{g (\p \phi)} \phi= S_0 (\p^2 \phi, \p \phi),
\ee
with initial data
 \be\label{ID}
 (\phi, \, \p_t\phi)|_{t=0}=(F(x), \, G(x)),
 \ee
where the source term $S_0(\p^2 \phi, \p\phi)$ depending on the second order derivatives $\p^2 \phi$ is given in \eqref{m-e} and $\Box_{g (\p \phi)} = \frac{1}{\sqrt{g}}\p_{\mu}\left(\sqrt{g}g^{\mu\nu}\p_{\nu}\right)$ refers to the Laplace Beltrami operator of the metric $g_{\mu \nu}$. In terms of the $\{u, \ub\}$ coordinate system, the metric $g_{\mu \nu}$ takes the form of
\als{
g_{\mu \nu} = {}& \cg_{\mu \nu} + \di \phi \otimes \di \phi,
 }
where $\cg_{\mu \nu}$ is the linear truncation
 \[\cg_{\mu \nu} := - \di u \otimes \di \ub - \di \ub \otimes \di u + \di \Psi \otimes \di \phi + \di \phi \otimes \di \Psi. \]
Direct calculation leads to
 \begin{equation}\label{metric}
g^{\mu \nu}= \cg^{\mu \nu} - \frac{\cg}{g} \p^\mu_{\cg} \phi \otimes \p^\nu_{\cg} \phi,
 \end{equation}
 where $\p^\mu_{\cg} \phi :=\cg^{\mu\nu}\p_{\nu}\phi$.
The $\cg$ and $g$ are the absolute values of the determinants of $\cg_{\mu \nu}, \, g_{\mu \nu}$ in the coordinate system $\{u, \ub\}$ respectively.

For any fixed $0<\gamma<1$, let $$a(x) :=\langle x \rangle^{2+2\gamma}, \quad \langle x \rangle := \sqrt{1+x^2},$$ then the functions $\Lambdab(\ub), \, \Lambda (u)$ are defined by
\be\label{def-Lambda-Lambdab}
\Lambdab(\ub) := a(\ub), \quad \Lambda (u) := a(u).
\ee

Throughout the paper, we use $A\sim B$ to denote $C_1A\leq B\leq C_2 A$ and use $A\lesssim B$ to denote $A\leq C_3 B$ for some universal positive constants $C_1$, $C_2$ and $C_3$ which may change from line to line.
We always denote $f^{(k)} (x) := \frac{\di^{k}}{\di x^{k}} f(x)$ for any smooth function $f(x)$. 

We further assume that the travelling wave solution $\Psi(u)$ satisfies the following assumptions.

 \begin{assumption}[Decay assumption]\label{Ass-2}
Given a small positive parameter $\epsilon$, we assume
 \begin{align}\label{Ass3}
&|\Psi^{(i+1)} (u) | \lesssim \Lambda^{-\frac{1}{2}} (u) \langle u \rangle^{-\frac{1+\epsilon}{2}}, \quad i \geq 0.
\end{align}
\end{assumption}

\begin{theorem}\label{main-theorem1}
Fix $0<\gamma<1$ and let $\Psi (t-x)$ satisfy the decay assumption \ref{Ass-2}.
Suppose there exist smooth functions $f(x), \, \fb(x) \in C^\infty(\mathbb R)$ satisfying
\be\label{data-ass-0}
\int_{\mathbb R} \Lambda (x) \left( |f^{(k)}(x)|^2 + |\fb^{(k)}(x)|^2 \right) \di x \leq I, \quad \text{for all} \,\, k \in \mathbb{Z}_{\geq 0},
\ee
for some constant $I$ that can be arbitrarily large, such that the data $(F(x), G(x))$ obey
\be\label{data-ass-1}
F^\prime(x) + G(x) = \delta f(x), \quad  G(x) - F^\prime(x) = 2\fb(x) - \delta (\Psip (-x))^2 f(x).
\ee
Then if $\delta$ is small enough (depending on the data $I$), the Cauchy problem \eqref{main-equation}-\eqref{ID} admits a unique and globally smooth solution $\phi$ in $C^{\infty}(\mathbb R^{+}\times \mathbb R)$. Moreover, the perturbation $\phi$ admits the estimates that for all $i,\,j \geq 0$,
\[| \p_u \p^i_u \p^j_{\ub} \phi | \lesssim \Lambda^{-\frac{1}{2}} (u), \quad |\p_{\ub} \p^i_u \p^j_{\ub} \phi| \lesssim \delta \Lambdab^{-\frac{1}{2}} (\ub). \]
\end{theorem}

\begin{remark}
The perturbations in the above theorem will be always non-small in any finite spatial region if initially they are non-small. Specifically, we will see in the proof that if initially it holds that \[  (\p_u \phi)^2 (t, x) |_{t=0} >   (\psi (-x))^2, \] in any finite spatial region, where $\psi(x)$ is an arbitrarily smooth function, then in a future region, there always holds (see Lemma \ref{lem-inital}) \[ (\p_u \phi)^2 > (\psi (u))^2. \]
\end{remark}

This result should be compared with our earlier one \cite{Wang-Wei1} which stated a global stability of the trivial solution with non-small perturbations. In fact, \cite{Wang-Wei1} was focusing on the original relativistic equation \eqref{I2} (not the perturbed one) for $\varphi$.  If we impose therein $\varphi = \phi + \Psi (t-x)$, where $\Psi(t-x)$ denotes a plane wave solution, 
then $\phi$ becomes the perturbation and the main result in \cite{Wang-Wei1} can be restated as follows.

\begin{theorem}\label{main-theorem0}[Main result in \cite{Wang-Wei1}]
Let $\Psi(t-x)$ denote a plane wave solution. Consider the Cauchy problem for  \eqref{I2} with $\varphi = \phi + \Psi (t-x)$ and data $(\phi, \p_t \phi)|_{t=0} = (F(x), G(x))$. Fix $0<\gamma<1$ and 
suppose there exist smooth functions $f(x), \, \fb(x) \in C^\infty(\mathbb R)$ satisfying
\[
\int_{\mathbb R} \Lambda (x) \left( |f^{(k)}(x)|^2 + |\fb^{(k)}(x)|^2 \right) \di x \leq I, \quad \text{for all} \,\, k \in \mathbb{Z}_{\geq 0},
\]
for some constant $I$ that can be arbitrarily large, such that the data $(F(x), G(x))$ obey
\bes
F^\prime(x) + G(x) = \delta f(x), \quad  G(x) - F^\prime(x) + 2 \Psip (-x)= \fb(x),
\ees
then there is a unique and globally smooth solution $\phi$ in $C^{\infty}(\mathbb R^{+}\times \mathbb R)$,  if $\delta$ is small enough (depending on the data $I$). Moreover, the perturbation $\phi$ admits the estimates that for all $i,\,j \geq 0$,
\[| \p_v \p^i_v \p^j_{\underline{v}} (\phi + \Psi (v)) | \lesssim \Lambda^{-\frac{1}{2}} (v), \quad |\p_{\underline{v}} \p^i_v \p^j_{\underline{v}} \phi| \lesssim \delta \Lambdab^{-\frac{1}{2}} (\underline{v}), \]
where $v=t-x$, $\underline{v}=t+x$.
\end{theorem}
It is clear that from the above theorem \ref{main-theorem0}, one is unable to know whether the perturbation $\phi$ is large or not. Nevertheless, based on the main theorem \ref{main-theorem1} of present paper, and the uniqueness of the Cauchy problem of hyperbolic equations, Theorem \ref{main-theorem0} indeed involves non-small perturbations if we make certain decay assumptions for $\Psi(t-x)$.

\subsection{Outline of the proof} In this paper, we develop the method in our previous work \cite{Wang-Wei1} which is inspired by \cite{Luli-Yang-Yu, Wang-Wei}, and now built on a framework where the induced metric of the relativistic string involves linear perturbations (rather than barely quadratic perturbations \cite{Wang-Wei, Wang-Wei1}).

Within the $\{u, \ub \}$ coordinate system, the perturbed system \eqref{main-equation} takes the alternative form
\begin{equation}\label{E-L-good}
g^{\mu\nu} (\p \phi) \p_{\mu}\p_{\nu}\phi = R_0 (\p \phi),
\end{equation}
where the semi-linear term 
\als{
R_0(\p \phi) \sim \Psi^{\prime \prime} (u) (\p_{\ub} \phi)^2.
}
The formula of \eqref{E-L-good}, which is equivalent to the Euler-Lagrangian equation for the perturbation $\phi$, benefits from the gauge choice and its proof involves huge amounts of cancellations, see \ref{Appendix-A1}. In view of the expression of $g^{\mu \nu} (\p \phi)$ \eqref{metric}, \eqref{E-L-good} is a quasilinear wave equation with the quasilinear term $g^{\mu \nu} (\p \phi) \p_\mu \p_\nu \phi - \eta^{\mu \nu} \p_\mu \p_\nu \phi$ satisfying the ``null condition'' ($\eta_{\mu \nu}$ is the flat Lorentzian metric). Namely, the derivative $\p_u$ is always pairwise coupled with the derivative $\p_{\ub}$.
At the same time, due to the decay assumption \ref{Ass-2}, the   derivatives of the  background plane wave solution $\Psi (u)$ share  analogous bounds with $\p_u \phi$. Thus, we can virtually replace $\Psip (u)$ and $\Psi^{\prime \prime} (u)$ by $\p_u \phi$, then we can image $R_0(\p \phi)$ as $\p_u \phi (\p_{\ub} \phi)^2$ which satisfies the ``null condition''.

On the other hand, in view of \eqref{E-L-good}, we are considering in a non-small setting the quasilinear wave equation with the quasilinear structure reflected completely in the geometry of the relativistic string $g^{\mu \nu} (\p \phi)$. Our proof takes advantage of the weighted energy estimates \cite{Luli-Yang-Yu}, the hierarchy of energy estimates \cite{Wang-Wei, Wang-Yu, Wang-Yu1}, and advances the method of geometrically adapted multipliers \cite{Wang-Wei, Wang-Wei1}. In particular, the multiplier vector fields $\underline{\Lambda}(\underline{v}) \p_{\underline{v}}$ and $\Lambda(v) \p_v$ used in \cite{Luli-Yang-Yu} are replaced now by the following dynamic ones
\be\label{multiplier-intro}
\TLb = -\Lmdu D \ub, \quad \TL = -\Lmdub Du,
\ee
where $D \ub := g^{\mu \nu} \p_\nu \ub \p_\mu$, $D u := g^{\mu \nu} \p_\nu u \p_\mu$ denote the gradients of $\ub$ and $u$ with respect to $g_{\mu \nu}$ respectively. By making use of the geometrical multipliers and the energy scheme for the geometric wave equation \eqref{main-equation} (and its higher order versions), the quasilinear structure which is determined precisely by $g^{\mu \nu} (\p \phi)$ naturally enters into the energies associated to the multipliers. To proceed this energy argument, it is crucial to require that the energies are positive, which will be usually ensured by the usage of non-spacelike (with respect to $g_{\mu \nu} (\p \phi)$) multipliers. However, the vector field $\TLb$ \eqref{multiplier-intro} fails to be non-spacelike with respect to $g_{\mu \nu} (\p \phi)$, since we know that the linear truncated metric $\cg^{\ub \ub} \cg = -2 \Psip (u) \p_{u} \phi$ has an undetermined sign. This gives rise to huge differences between the current situation and the ones in \cite{Wang-Wei, Wang-Wei1}, where the induced metric of membrane or string exhibits quadratic perturbations (with favorable signs) of the trivial background solution. In particular, we should note that the multipliers $\Lmdu \left(\p_u + |\p_u \phi|^2 \p_{\ub} \right)$, $\Lmdub \left(\p_{\ub} + |\p_{\ub} \phi|^2 \p_{u} \right)$ used in \cite{Wang-Wei1} are no longer valid, for their associated energies fail to generate positive dominating terms.
 Nevertheless, a key observation that the energies associated to $\TL$ and $\TLb$ are non-negative up to some lower order terms comes into play. For instance, in the energies associated to $\TL$,  there is, up to some lower order terms,
\als{
\T[\psi](-Dt, \TL)  \sim &  \frac{1+(\Psip(u))^2}{2} \Lambdab(\ub) ( \p_{\ub} \psi )^2  + \frac{\Lambdab(\ub)}{2}  \left(  |\p^{u}_{\cg} \phi|^2 (\p_{u} \psi)^2 - g^{\ub \ub} (\p_{\ub} \psi)^2  \right),
}
where $\T[\psi]_{\alpha \beta}$ defined in Section \ref{sec-energy-sch} is the energy momentum tensor associated to the geometric wave operator $\Box_{g (\p \phi)} \psi$, and the difficult term $ - \frac{\Lambdab(\ub)}{2} g^{\ub \ub} (\p_{\ub} \psi)^2$ is neither a small perturbation nor manifesting a favorable sign (note that, in \cite{Wang-Wei, Wang-Wei1}, $-g^{\ub \ub}$ is non-negative, and hence admits a favorable sign). Fortunately, there is the additional positive term $\frac{ (\Psip(u))^2}{2} \Lambdab(\ub) ( \p_{\ub} \psi )^2$, so that, up to lower order terms,
\als{
& \frac{ (\Psip(u))^2}{2} \Lambdab(\ub) ( \p_{\ub} \psi )^2 - \frac{\Lambdab(\ub)}{2} g^{\ub \ub} (\p_{\ub} \psi)^2 \\
\sim {} &  \frac{\Lambdab(\ub)}{2} \left((\Psip(u))^2 + (\p_{u} \phi)^2 + 2 \Psip (u) \p_{\ub} \phi \right) (\p_{\ub} \psi)^2 \\
= {}& \frac{\Lambdab(\ub)}{2} \left(\Psip(u) + \p_{u} \phi \right)^2 (\p_{\ub} \psi)^2
}
is non-negative and hence admits a favorable sign, see Lemma \ref{lemma-energy-formula}.

Afterwards, we can make use of the ``null structure'' and a hierarchy of energy estimates corresponding to $\TLb$ and $\TL$ to close the energy argument. We point out that, it suffices to conduct robust estimates in the energy argument, since we have paid much effort on deriving a concisely equivalent Euler-Lagrangian equation \eqref{E-L-good}.

\subsection{Arrangement of the paper}
We arrange our paper as follows: In Section \ref{sec:2}, we give some preparations for the proof of the main theorem including the derivation of the perturbed system and its associated geometry, energy scheme, higher order equations, and present the geometrical multipliers. Section \ref{sec:3} is the main body of the paper, in which we prove the main Theorem \ref{main-theorem1} by the weighted energy method. In the end, we collect some detailed calculations in the appendix.

\section{Preliminaries}\label{sec:2}
In this section, we give some preparations such as the geometry of the relativistic string, the equivalent  Euler-Lagrangian equation and its higher order equation, the energy scheme and the multipliers.

\subsection{Formulation of the perturbed relativistic string equation}\label{sec-string}
The derivations of the perturbed relativistic string equation in this section are inspired by \cite{Abbrescia-Wong}.

\subsubsection{Gauge choices and embedding}
Let $\eta_{\mu \nu} =-\di t^2 + \di x^2$ be the standard flat Lorentzian metric. We consider the graphic embedding $\mathbb{R}^{1+1} \hookrightarrow \mathbb{R}^{1+2}$ taking the form of
\be\label{emb-varphi}
(t, x) \mapsto (t, x, \varphi (t, x)).
\ee
The induced metric is given by \[g_{\mu \nu} = \eta_{\mu \nu} + \di \varphi \otimes \di \varphi. \]
A critical point of the area functional $\varphi \rightarrow \int \text{dvol}_g$ is called \emph{a relativistic string}. The associated Euler-Lagrange equation is the relativistic string equation.

Given any function $\Psi: \mathbb{R}^{1+1} \rightarrow \mathbb{R}$, the embedding given by
\be\label{emb-Psi}
(t, x) \mapsto (t, \, x, \, \Psi(t-x) )
\ee is a critical point of the induced volume. That is, $\Psi(t-x)$ is a plane wave solution to the relativistic string equation.  
The induced metric of the embedding \eqref{emb-Psi} reads \[\ceta_{\mu \nu} = \eta_{\mu \nu} + \di \Psi \otimes \di \Psi, \] 
which is in fact flat \cite{Abbrescia-Wong}, for if we define the $\{u, \, \ub\}$ coordinate
\begin{equation}\label{coordinate}
u = t-x, \quad \ub = \frac{t+x}{2} - \frac{1}{2} \int_0^{t-x} (\Psip)^2 (\tau) \di \tau,
\end{equation}
then the line element of $\ceta_{\mu \nu}$ can be written in a double null form  \[ \ceta_{\mu \nu} = -2 \di u \di \ub. \] 
For later computations, we note that under the change of coordinate \eqref{coordinate}, there are $$\p_{\ub} = \p_t + \p_x, \quad \p_u =  \frac{(\Psip (u))^2 +1}{2} \p_t +  \frac{ (\Psip (u))^2 - 1}{2}\p_x,$$ and $$\di u = \di t -\di x,\quad  \di \ub = \frac{1- (\Psip (u))^2 }{2}\di t + \frac{1+(\Psip (u))^2}{2} \di x,$$ and hence the determinant of the Jacobian matrix $|\frac{\p(u, \ub)}{\p(t, x)}| = 1$.

Consider the small perturbations of the embedding \eqref{emb-Psi}.
We take $\varphi = \phi + \Psi(t-x)$ in \eqref{emb-varphi}.
Alternatively, the induced metric takes the form \[g_{\mu \nu} =
\ceta_{\mu \nu} +  \Psip (u) \di \phi  \otimes \di u  +  \Psip (u) \di u \otimes \di \phi + \di \phi \otimes \di \phi.  \]
The determinant  in the coordinate system $\{u, \ub\}$ is then given by \[g = 1+ \ceta (\p \phi, \p \phi) -2 \Psip(u) \p_{\ub} \phi + (\Psi^\prime(u) )^2 (\p_{\ub} \phi)^2. \]
Let us record the perturbations truncated to the linear terms
\[\cg_{\mu \nu} := \ceta_{\mu \nu} +  \Psip (u) \di \phi  \otimes \di u  +  \Psip (u) \di u \otimes \di \phi, \] and then its reciprocal is
\als{
\cg^{\mu \nu} = {}& \ceta^{\mu \nu} - \frac{\Psip (u) \p_{\ub} \phi - (\Psi^\prime(u) )^2 (\p_{\ub} \phi)^2}{\cg} \p_{\ub} \otimes \p_{u} \\
 &- \frac{\Psip (u) \p_{\ub} \phi - (\Psi^\prime(u) )^2 (\p_{\ub} \phi)^2}{\cg} \p_{u} \otimes \p_{\ub}   - \frac{2\Psip (u) \p_u \phi}{\cg} \p_{\ub} \otimes \p_{\ub},
}
where the determinant is denoted by $\cg=(1-\Psi^\prime(u) \p_{\ub} \phi)^2$. We notice that \[\cg^{\ub \ub} = \frac{-2 \Psip (u) \p_u \phi}{\cg}, \quad \cg^{u \ub} = \frac{-1+\Psip (u) \p_{\ub} \phi}{\cg}= -1 - \frac{\Psip (u) \p_{\ub} \phi - (\Psi^\prime(u) )^2 (\p_{\ub} \phi)^2}{\cg}.\] Implied by the above formulae \[ g = \cg (1+ \cg^{-1} (\p \phi, \p \phi) ), \] which
together with the formula for the induced metric $g_{\mu \nu} = \cg_{\mu \nu} + \di \phi \otimes \di \phi$, yields
\[ g^{\mu \nu}= \cg^{\mu \nu} - \frac{\cg}{g} \p^\mu_{\cg} \phi \otimes \p^\nu_{\cg} \phi, \]
 where $\p^\mu_{\cg} \phi :=\cg^{\mu\nu}\p_{\nu}\phi$.

\begin{remark}
We note that $\cg^{u u} =0$ and
\als{
\cg^{u\ub} \cg =& -1+ \Psip (u) \p_{\ub} \phi, & \cg^{\ub \ub} \cg =& -2 \Psip (u) \p_{u} \phi, \\
\cg \p^u_{\cg} \phi =& - \p_{\ub} \phi + \Psip (u) (\p_{\ub} \phi)^2, & \cg \p^{\ub}_{\cg} \phi =&  - \p_u \phi - \Psip (u) \p_u \phi \p_{\ub} \phi.
}
In our strategy, $|\p_{\ub} \phi| \lesssim \delta$ will be small, while $\p_u \phi$ would be non-small. Then it holds that
\als{
\cg |g^{\ub \ub}| ={}& |- 2\Psip (u) \p_u \phi - g^{-1}  \cg^2 \p^{\ub}_{\cg} \phi \p^{\ub}_{\cg} \phi| \\
={}& | (2 \Psip (u) + \p_u \phi)  \p_u \phi| \cdot | 1 - \Psip (u) \p_{\ub} \phi (2-  \Psip (u) \p_{\ub} \phi) | \\
\sim {} & | (2 \Psip (u) + \p_u \phi)  \p_u \phi|, \\
 - \cg g^{uu} = {}& g^{-1}  \cg^2 \p^{\ub}_{\cg} \phi \p^{\ub}_{\cg} \phi \sim  (\p_{\ub} \phi)^2, \\
- \cg g^{u \ub} = {}& - \cg \cg^{u \ub} + g^{-1} \cg^2 \p^{\ub}_{\cg} \phi \p^{u}_{\cg} \phi \sim 1- \Psip (u) \p_{\ub} \phi +  \p_{\ub} \phi \p_u \phi.
 }
 \end{remark}

 \subsubsection{Euler-Lagrangian equation for the perturbed system}

In the null coordinates $(u,\ub)$, the density of the Lagrangian function is
$$\lie = \sqrt{|g|} = \sqrt{1 -2 \p_u \phi \p_{\ub} \phi -2 \Psip(u) \p_{\ub} \phi + (\Psi^\prime(u) )^2 (\p_{\ub} \phi)^2}.$$ The system of the perturbation $\phi$ can be derived as the Euler--Lagrangian equation of $\lie$.
As we have ($\phi_u := \p_u \phi, \, \phi_{\ub} := \p_{\ub} \phi$)
\als{
\frac{\p \lie}{\p \phi_u} = {}& -  \lie^{-1} \p_{\ub} \phi, \\
\frac{\p \lie}{\p \phi_{\ub}} 
={}&- \lie^{-1} \p_{u} \phi -  \lie^{-1} \Psip (u)  +   \lie^{-1} (\Psip (u))^2 \p_{\ub} \phi,
}
the Euler-Lagrangian equation
 \begin{equation}\label{E-L-1}
 \p_u \left( \frac{\p \lie}{\p \phi_u} \right) + \p_{\ub} \left( \frac{\p \lie}{\p \phi_{\ub}} \right) = 0,
  \end{equation}
becomes a quasilinear wave equation on $\phi$
\begin{align}
  & \p_u \left(-  \p_{\ub} \phi (\sqrt{g})^{-1} \right) + \p_{\ub} \left(-  \p_{u} \phi (\sqrt{g})^{-1} \right)\nnb \\
   &- \p_{\ub} \left( (\sqrt{g})^{-1} \Psip (u) \right) +  \p_{\ub} \left( (\sqrt{g})^{-1} (\Psip (u))^2 \p_{\ub} \phi \right) =0. \label{eq-EL-0}
\end{align}

To reveal the geometry, we provide an alternative formulation of the Euler-Lagrangian equation.
\begin{proposition}\label{prop-EL-ini}
We reformulate the Euler-Lagrangian equation \eqref{eq-EL-0} as
\begin{align}
\Box_{g(\p\phi)} \phi = {}& S_0 (\p^2 \phi, \p \phi), \nnb \\
  S_0  (\p^2 \phi, \p \phi) ={}& \frac{\cg}{g^2}  \left( 1-g \right) (\Psip (u))^2 \p^2_{\ub} \phi + \frac{1}{g^2}  \left( \cg -g \right)  \Psip (u)  \p_u \phi \p^2_{\ub} \phi \nonumber\\
 &- \frac{\cg}{g^2}  (\Psip (u))^3 \p_{\ub} \phi \p_{\ub}^2 \phi + \frac{\cg}{g^2}  \Psip (u)  \p_{\ub} \phi \p_{\ub} \p_u \phi +  \frac{(\cg -1)}{g} 2 \p_u  \p_{\ub} \phi \nonumber\\
&+  \frac{1}{g}  \Psip (u)  \p_{\ub} \phi \p_u  \p_{\ub} \phi  +  \frac{1}{g}  \Psi^{\prime \prime} (u) (\p_{\ub} \phi )^2- \frac{\cg}{\sqrt{g}} (\Psip (u))^2 \p_{\ub} \phi  \p_{\ub} (\sqrt{g})^{-1}  \nonumber \\
&   +  \frac{(\cg -1)}{\sqrt{g}}  \p_{\ub} \phi  \p_u \left( (\sqrt{g})^{-1} \right) + \frac{(\cg -1)}{\sqrt{g}} \p_{u} \phi  \p_{\ub} \left(  (\sqrt{g})^{-1} \right) \nonumber\\
 &+  \frac{1}{\sqrt{g}} \Psip (u)  (\p_{\ub} \phi)^2 \p_u (\sqrt{g})^{-1} -  \frac{1}{\sqrt{g}}  \Psip (u) \p_{\ub} \phi  \p_{u} \phi \p_{\ub} (\sqrt{g})^{-1}.   \label{m-e}
\end{align}
\end{proposition}
\begin{remark}
When expanding the $g-\cg$, $g-1$, $\cg -1$ in \eqref{m-e}, we see that each term in $S_0(\p^2 \phi, \p\phi)$ has one of the factors $\Psip (u)  \p_{\ub} \phi$, $\Psi^{\prime \prime} (u) \p_{\ub} \phi$ or $ \p_u \phi \p_{\ub} \phi$. Note that, $\Psip (u)$ and $\Psi^{\prime \prime} (u)$ decay (in terms of $u$) similarly as $\p_u \phi$. Therefore, this equation \eqref{m-e} mimics a geometrically quasilinear wave equation with null structure.
\end{remark}

\begin{proof}
Following \cite{Abbrescia-Wong}, we resort to $\p_\mu (\cg^{\mu \nu} \p_\nu \phi \cg (\sqrt{g})^{-1} )$ for this formula. Combining
\begin{align}
\p_\mu (\cg^{\mu \nu} \p_\nu \phi \cg (\sqrt{g})^{-1} ) = {}& \p_u \left(\cg^{u u} \p_u \phi \cg (\sqrt{g})^{-1} + \cg^{u \ub} \p_{\ub} \phi \cg (\sqrt{g})^{-1} \right) \nnb \\
&+ \p_{\ub} \left( \cg^{\ub u} \p_u \phi \cg (\sqrt{g})^{-1} + \cg^{\ub \ub} \p_{\ub} \phi \cg (\sqrt{g})^{-1} \right) \nnb \\
={}& \p_u \left(-  \p_{\ub} \phi \cg (\sqrt{g})^{-1} \right) + \p_{\ub} \left(-  \p_{u} \phi \cg (\sqrt{g})^{-1} \right) \nnb \\
& - \p_u \left( ( \Psip (u) \p_{\ub} \phi - (\Psi^\prime(u) )^2 (\p_{\ub} \phi)^2 ) \p_{\ub} \phi (\sqrt{g})^{-1} \right) \nnb \\
& - \p_{\ub} \left( (\Psip (u) \p_{\ub} \phi - (\Psi^\prime(u) )^2 (\p_{\ub} \phi)^2)  \p_{u} \phi (\sqrt{g})^{-1} \right) \nnb \\
& -2 \p_{\ub} \left((\Psip (u) \p_{u} \phi) \p_{\ub} \phi (\sqrt{g})^{-1} \right), \label{E-C}
\end{align}
with the Euler-Lagrangian equation \eqref{eq-EL-0}, leads to the following derived formula
 \begin{align}
 & \p_\mu (\cg^{\mu \nu} \p_\nu \phi \cg (\sqrt{g})^{-1} )  \nnb \\
 ={}& \cg \p_{\ub} ((\sqrt{g})^{-1} \Psip (u) ) - \cg \p_{\ub} ((\sqrt{g})^{-1} (\Psip (u))^2 \p_{\ub} \phi) \nnb \\
 &- \p_u \cg \left( \p_{\ub} \phi  (\sqrt{g})^{-1} \right) - \p_{\ub} \cg \left( \p_{u} \phi  (\sqrt{g})^{-1} \right) \nnb \\
& - \p_u \left( ( \Psip (u) \p_{\ub} \phi - (\Psi^\prime(u) )^2 (\p_{\ub} \phi)^2 ) \p_{\ub} \phi (\sqrt{g})^{-1} \right) \nnb  \\
& - \p_{\ub} \left( (3\Psip (u) \p_{\ub} \phi - (\Psi^\prime(u) )^2 (\p_{\ub} \phi)^2)  \p_{u} \phi (\sqrt{g})^{-1} \right). \label{E-L1}
\end{align}
Notice that $g = \cg (1+ \cg^{-1} (\p \phi, \p \phi) )$ and thus there is
\als{
 g \cdot g^{\mu \nu} \p_\nu \phi ={}& g \cdot \cg^{\mu \nu} \p_\nu \phi - \cg \p^\mu_{\cg} \phi \p^\nu_{\cg} \phi \p_\nu \phi \\
={}& g \cdot  \p^\mu_{\cg} \phi - \cg \cg^{-1} (\p \phi, \p \phi)  \p^\mu_{\cg} \phi \\
={}& \cg \p^\mu_{\cg} \phi.
}
It then follows that the equation \eqref{E-L1} is identical to
\begin{align}
\frac{1}{\sqrt{g}} \p_\mu ( g^{\mu \nu} \p_\nu \phi \sqrt{g}  )  = {}& S_0 (\p^2 \phi, \p \phi), \nnb \\
 S_0 (\p^2 \phi, \p \phi) ={}& \frac{\cg}{\sqrt{g}} \p_{\ub} ((\sqrt{g})^{-1} \Psip (u) ) - \frac{\cg}{\sqrt{g}}  \p_{\ub} ((\sqrt{g})^{-1} (\Psip (u))^2 \p_{\ub} \phi)  \nnb \\
& - \frac{1}{\sqrt{g}}\p_u \left( ( \Psip (u) \p_{\ub} \phi - (\Psi^\prime(u) )^2 (\p_{\ub} \phi)^2 ) \p_{\ub} \phi (\sqrt{g})^{-1} \right) \nnb  \\
& - \frac{1}{\sqrt{g}} \p_{\ub} \left( (3 \Psip (u) \p_{\ub} \phi - (\Psi^\prime(u) )^2 (\p_{\ub} \phi)^2)  \p_{u} \phi (\sqrt{g})^{-1} \right) \nnb \\
 &- \frac{1}{\sqrt{g}}\p_u \cg \left( \p_{\ub} \phi  (\sqrt{g})^{-1} \right) - \frac{1}{\sqrt{g}}\p_{\ub} \cg \left( \p_{u} \phi  (\sqrt{g})^{-1} \right).  \label{E-L2}
\end{align}
A key observation for $S_0 (\p^2 \phi, \p \phi)$ lies in that there are numerous cancellations hidden inside the formula \eqref{E-L2}. In particular, terms such as $\Psip (u) \p_u \phi \p^2_{\ub} \phi$\footnote{This term will be an issue in the energy argument if it were not absent in $S_0 (\p^2 \phi, \p \phi)$.} are absent.

The rest of the proof which is devoted to technically finding out these crucial cancellations in \eqref{E-L2} is collected in \ref{Appendix-A1}.
\end{proof}

Despite that the variant of Euler-Lagrangian equation \eqref{m-e} satisfies the ``null structure'', it gives us an illusion that apart from the quasilnear structure reflected in the metric $g^{\mu \nu} (\p \phi)$, there are additional quasilinear terms in $S_0 (\p^2 \phi, \p \phi)$. Nevertheless, thanks to the gauge choice, the equation \eqref{m-e} has the following simplified version\footnote{In \ref{Appendix-A2}, we also provide a different proof which results in an alternative formulation of the Euler-Lagrangian equation that is equivalent to \eqref{eq-expanding} but contains much more quasilinear terms. In fact,  if we substitute the Euler-Lagrangian equation \eqref{eq-EL-0} into this alternative formula, those additionally quasilinear terms vanish.}.
\begin{proposition}\label{prop-EL}
In the coordinate system $\{u, \ub\}$, the Euler-Lagrangian equation \eqref{eq-EL-0} implies,
\begin{equation}\label{eq-expanding}
g^{\mu\nu} (\p \phi) \p_{\mu}\p_{\nu}\phi = R_0 (\p \phi),
\end{equation}
where $R_0 (\p \phi)$ denotes the semi-linear term
\als{
R_0 (\p \phi) ={}
&  (\cg g)^{-1} \Psi^{\prime \prime} (u) (\p_{\ub} \phi)^2 \left(  1 - 3\Psip (u) \p_{\ub} \phi +  (\Psip (u) \p_{\ub} \phi)^2 - (\Psip (u) \p_{\ub} \phi)^3 \right).
}
\end{proposition}
\begin{proof}
We here give an overview of the proof. By virtue of \eqref{E-L1}-\eqref{E-L2},
 \als{
 S_0(\p^2 \phi, \p \phi) = {}& \frac{1}{\sqrt{g}} \p_{\mu} \left( \cg^{\mu\nu}\p_{\nu} \phi \cg (\sqrt{g})^{-1}\right)  \\
= {} & \cg g^{-1} g^{\mu\nu}\p_{\mu}\p_{\nu}\phi  + \Re(\p^2 \phi, \p \phi),
}
where
\als{ &\quad  \Re(\p^2 \phi, \p \phi)  \\
 = {} & \left( \frac{\cg^2}{g^2} \p_{\cg}^{\mu}\phi\p_{\cg}^{\nu}\phi\p_{\mu}\p_{\nu}\phi \right) + \frac{1}{\sqrt{g}} \p_{\nu}\phi \left(
\cg^{\mu\nu} \cg \p_{\mu}(\sqrt{g})^{-1} + \p_{\mu} (\cg^{\mu\nu} \cg) (\sqrt{g})^{-1} \right).
}
Then we find that, in $S_0(\p^2 \phi, \p \phi) - \Re(\p^2 \phi, \p \phi)$, all of the terms containing the second order derivatives $\p^2 \phi$ cancel exactly. The detailed proof is carried out in \ref{Appendix-A1}.
\end{proof}

The equivalent Euler-Lagrangian equation \eqref{eq-expanding} clarifies that the quasilinear structure in the relativistic string equation is completely reflected in the metric $g^{\mu\nu} (\p \phi)$. It also helps to derive the high order equation immediately.
Let \[\phi_k := \p^k \phi, \quad \p^k := \p^i_u \p^j_{\ub} \phi, \quad i+j=k, \, k \geq 1, \] denote all possible derivatives of $k$th-order.
\begin{proposition}
Based on the equivalent Euler-Lagrangian equation \eqref{eq-expanding}, we derive the high order geometrical wave equation,
\als{
\Box_{g}\phi_k ={}& 
S_k,
}
where $S_k (k \geq 1)$ containing lower order terms is given by
\als{
S_k :={} & 
\p^{k} \left( (\cg g)^{-1} \Psi^{\prime \prime} (u) (\p_{\ub} \phi)^2 \left(  1 - 3\Psip (u) \p_{\ub} \phi +  (\Psip (u) \p_{\ub} \phi)^2 - (\Psip (u) \p_{\ub} \phi)^3 \right) \right) \\
&- \sum_{ i+j\leq k,\, j<k}\p^{i}g^{\mu\nu}\p_{\mu}\p_{\nu}\phi_{j}+\frac{1}{\sqrt{g}}\p_{\nu}\phi_k\p_{\mu}\left(g^{\mu\nu}\sqrt{g}\right).
}
\end{proposition}

\begin{remark}

Keeping in mind that the $\p_{\ub}$ derivative has always smallness, we infer from the above formula that $S_k (k \geq 1)$ has the following leading terms,
\bes
\p_{u}\phi_i\p_{\ub}\phi_j\p_{u}\p_{\ub}\phi_l,\,\Psi^{(i+1)}(u)\p_{\ub}\phi_j\p_u\p_{\ub}\phi_l,\,\p_{\ub}\phi_i\p_{\ub}\phi_j\p_{u}\p_{u}\phi_{l},
\ees
\bes
\Psi^{(i+1)}(u)\p_{u}\phi_j\p_{\ub}\p_{\ub}\phi_l,\,\p_{u}\phi_i\p_u\phi_j\p_{\ub}\p_{\ub}\phi_l, \, \Psi^{(i+2)} \p_{\ub} \phi_j \p_{\ub} \phi_l,
\ees
\bes
 \Psip (u) \p_{\ub}\p_{\ub}\phi\p_u\phi_k,\, \p_u\phi\p_{\ub}\p_{\ub}\phi\p_u\phi_k,\,\p_{\ub}\phi\p_{u}\p_{\ub}\phi\p_{u}\phi_k,
\ees
\bes
\p_{u}\p_{u}\phi\p_{\ub}\phi\p_{\ub}\phi_k,\,\p_{u}\phi\p_{u}\p_{\ub}\phi\p_{\ub}\phi_k,\,\Psip(u)\p_{u}\p_{\ub}\phi\p_{\ub}\phi_k,
\ees
where $i+j+l\leq k,\,l<k$. In principle, each of these terms manifests the null structure.
\end{remark}

\begin{proof}
We have
\be\label{h-order}
\Box_g\phi_k=\frac{1}{\sqrt{g}}\partial_{\mu}\left(g^{\mu\nu}\sqrt{g}\p_{\nu}\phi_k\right)=g^{\mu\nu}\p_{\mu}\p_{\nu}\phi_k
+\frac{1}{\sqrt{g}}\p_{\nu}\phi_k\p_{\mu}\left(g^{\mu\nu}\sqrt{g}\right).
\ee
Applying the $\partial^k$ to $g^{\mu\nu}\p_{\mu}\p_{\nu}\phi$, we derive
\be\label{commute}
g^{\mu\nu}\p_{\mu}\p_{\nu}\phi_k=\p^{k}(g^{\mu\nu}\p_{\mu}\p_{\nu}\phi)-
\sum_{ i+j\leq k,\, j<k }\p^{i}g^{\mu\nu}\p_{\mu}\p_{\nu}\phi_{j}.
\ee

Thus, combining \eqref{eq-expanding}, \eqref{h-order} and \eqref{commute}, leads to the high order equation
\bes
\Box_{g}\phi_k =\p^k R_0 (\p \phi) - \sum_{ i+j\leq k,\, j<k}\p^{i}g^{\mu\nu}\p_{\mu}\p_{\nu}\phi_{j} +\frac{1}{\sqrt{g}}\p_{\nu}\phi_k\p_{\mu}\left(g^{\mu\nu}\sqrt{g}\right).
\ees

\end{proof}

\subsection{Energy scheme}\label{sec-energy-sch}

To make use of the null structure in the main equation, we follow \cite{Luli-Yang-Yu} to define the null foliations. The out-going null segment is defined as
$$C^\tau_{u_0} := \Big \{(u, \ub) \big| u = u_0, \, 0 \leq t (u, \ub) \leq \tau \Big\},$$
and the in-coming null segment is
$$\Cb^\tau_{\ub_0} := \Big\{(u, \ub) \big|  \ub = \ub_0, \, 0 \leq t (u, \ub) \leq \tau \Big\}.$$ Recall that according to the null coordinate $\{u, \, \ub\}$ \eqref{coordinate}, the rectangular coordinates $t$ and $x$ are smooth functions of $u, \ub$, and vice versa.
The spacetime region on the right of $C^\tau_{u_0}$ is,
$$\D^{+}_{\tau, u_0} := \Big\{(u, \ub) \big| 0 \leq t(u, \ub) \leq \tau, \,    u \leq u_0 \Big\},$$ and the spacetime region on the left of  $\Cb^\tau_{\ub_0}$ is 
$$\D^{-}_{\tau, \ub_0} := \Big\{(u, \ub) \big| 0 \leq t(u, \ub) \leq \tau, \,   \ub \leq \ub_0\Big \}.$$
We also define
\bes
\D_{\tau} := \{(u, \ub) \big| 0 \leq t (u, \ub) \leq \tau \}.
\ees
Then $\D_\tau$ can be foliated by $\{C^\tau_u\}_{u \in \mathbb{R}}$ and $\{ \Cb^\tau_{\ub}\}_{\ub \in \mathbb{R}}$, and $\D^{-}_{\tau, \ub} \subset \D_\tau$, $\D^{+}_{\tau, u} \subset \D_\tau$.
There are correspondingly spatial segments
\begin{align*}
\Sigma^{+}_{\tau, u_0} := \Big\{(t, x) \big| t =\tau, \,  u (t, x) \leq u_0\Big \}, \\
\Sigma^{-}_{\tau, \ub_0} := \Big\{(t, x) \big| t =\tau, \, \ub (t, x) \leq \ub_0\Big \}.
\end{align*}
And we denote $$\Sigma_\tau := \{(t, x) \big| t=\tau\}.$$ There is of course, $\Sigma^{+}_{\tau, u}  \subset \Sigma_\tau$, and  $\Sigma^{-}_{\tau, \ub}  \subset \Sigma_\tau$.

We can define the corresponding energy momentum tensor
\begin{equation}\label{2.13}
\T^{\alpha}_{\beta} [\psi] := g^{\alpha \mu}\partial_{\mu}\psi\partial_{\beta}\psi-\frac{1}{2} g^{\mu\nu}\partial_{\mu}\psi\partial_{\nu}\psi \delta^{\alpha}_{\beta},
\end{equation}
where $\delta^{\alpha}_{\beta}$ denotes the Kronecker tensor.

Given any vector field $\xi$, the corresponding current $P^{\alpha}$ is defined as
\begin{equation}\label{2.14}
P^{\alpha}= P^{\alpha}[\phi,\xi] := \T^{\alpha}_{\beta} [\phi] \cdot \xi^{\beta}.
\end{equation}
The divergence of $P^{\alpha}$ could be computed straightforwardly, see \cite{Wang-Wei}.
\begin{lemma}\label{lem:2.4}
The energy current $P^{\alpha}$ satisfies
\begin{equation}\label{eq-div-P}
\frac{1}{\sqrt{g}}\partial_{\alpha}(\sqrt{g}P^{\alpha})= \Box_{g(\p\phi)}\phi \cdot \xi \phi + \T^{\alpha}_{\beta} [\phi] \cdot \partial_{\alpha}\xi^{\beta}
-\frac{1}{2\sqrt{g}}\xi(\sqrt{g}g^{\gamma\rho})
\partial_{\gamma}\phi\partial_{\rho}\phi.
\end{equation}
\end{lemma}
As a remark, an alternative expression of \eqref{eq-div-P} is \[D_\alpha P^\alpha =  \Box_{g(\p\phi)}\phi \cdot \xi \phi + \T^{\alpha}_{\beta} [\phi] \cdot D_{\alpha}\xi^{\beta}, \] where $D$ is the covariant derivative associated to $g_{\mu \nu}$. In fact, \[ \T^{\alpha}_{\beta} [\phi] \cdot \partial_{\alpha}\xi^{\beta}
-\frac{1}{2\sqrt{g}}\xi(\sqrt{g}g^{\gamma\rho}) = \T^{\alpha}_{\beta} [\phi] \cdot D_{\alpha}\xi^{\beta}= \frac{1}{2} \T^{\alpha \beta} [\phi] \mathcal{L}_\xi g_{\alpha \beta}\] reflects the deformation tensor $\frac{1}{2}\mathcal{L}_\xi g_{\alpha \beta}$ of $\xi$. And here $\T^{\alpha \beta} [\phi]$ will be defined later.

The $(0,2)$-energy momentum tensor $\T_{\alpha \beta}$ is defined by $\T_{\alpha \beta} := g_{\alpha \sigma} \T^{\sigma}_{\beta}$, where $\T^{\sigma}_{\beta}$ is the $(1,1)$-energy momentum tensor defined in \eqref{2.13}, and $\T^{\alpha \beta} := g^{\alpha \sigma} \T^{\beta}_{\sigma}$. We
will omit the $[\phi]$ in $\T_{\alpha \beta} [\phi], \, \T_{\alpha}^{\beta} [\phi], \cdots$ if there is no confusion.
Using the $\T_{\alpha \beta}$ tensor, we have
\begin{equation}\label{def-energy-xi-u-ub}
-P^{u}[\phi,\xi] = \T(-Du, \xi), \quad -P^{\ub}[\phi,\xi] = \T(-D \ub, \xi),
\end{equation}
and
\begin{equation}\label{def-energy-xi-t}
-P^{t}[\phi,\xi] = \T(-Dt, \xi),
\end{equation}
where we recall that for any function $f$, $Df$ is the gradient of $f$ with respect to $g_{\alpha\beta}$.

By the divergence theorem (referring to \ref{subsection:2.5}), we have the following energy identities
\begin{equation}\label{energy-id-1}
\begin{split}
& \int_{\Sigma^{-}_{\tau, \ub}} \T[\psi](-Dt, \xi) \sqrt{g} \cdot J \di x+ \int_{\Cb^\tau_{\ub}} \T[\psi](-D \ub, \xi) \sqrt{g} \di u \\
&= \int_{\Sigma^{-}_{\tau_0, \ub}} \T[\psi](-Dt, \xi)\sqrt{g}  \cdot J \di x -\iint_{\D^{-}_{\tau, \ub}}\p_{\alpha}(\sqrt{g}P^{\alpha} [\psi, \xi]) \di u \di \ub,
\end{split}
\end{equation}
and
\begin{equation}\label{energy-id-2}
\begin{split}
& \int_{\Sigma^{+}_{\tau, u}} \T[\psi](-Dt, \xi) \sqrt{g} \cdot J \di x+ \int_{C^\tau_{u}} \T[\psi](-D u, \xi) \sqrt{g} \di \ub \\
&= \int_{\Sigma^{+}_{\tau_0, u}} \T(-Dt, \xi)\sqrt{g} \cdot J \di x -\iint_{\D^{+}_{\tau, u}}\p_{\alpha}(\sqrt{g}P^{\alpha} [\psi, \xi]) \di u \di \ub,
\end{split}
\end{equation}
where $J := |\frac{\p(u, \ub)}{\p(t, x)}| = 1$ is the absolute value of the determinant of the Jacobian matrix $\frac{\p(u, \ub)}{\p(t, x)}$.

\subsection{Multipliers}\label{subsection:2.4}
We employ the weighted, geometrical multipliers which take the following forms,
\bel{def-Z-Zb}
\TLb = -\Lmdu D \ub, \quad \TL = -\Lmdub Du.
\ee
\begin{remark}
We expand the two multipliers as below,
\als{
\TLb ={}& -\Lmdu g^{\ub u} \p_u -\Lmdu g^{\ub \ub} \p_{\ub} \\
={}&  - \cg^{-1} \Lmdu \left(-1+ \Psip (u) \p_{\ub} \phi - g^{-1}  \p_{\ub} \phi  \p_{u} \phi [1-  (\Psip (u) \p_{\ub} \phi)^2 ]  \right) \p_u -\Lmdu g^{\ub \ub} \p_{\ub} \\
={}& (\Lmdu + \text{l.o.t.}) \p_u  -\Lmdu g^{\ub \ub} \p_{\ub} ,
}
and
\als{
\TL ={}& -\Lmdub g^{u u} \p_u -\Lmdub g^{u \ub} \p_{\ub} \\
={} 
& -\Lmdub g^{u u} \p_u - \cg^{-1} \Lmdub \left(-1+ \Psip (u) \p_{\ub} \phi - g^{-1}  \p_{\ub} \phi  \p_{u} \phi [1-  (\Psip (u) \p_{\ub} \phi)^2 ]  \right) \p_{\ub} \\
={}& (\Lmdub + \text{l.o.t.})  \p_{\ub}  -\Lmdub g^{u u} \p_u .
}
The notation l.o.t. always refers to terms that are of lower order in both of the decay rates and the $\delta$ power.

As a remark, in \cite{Wang-Wei1}, we use the principal part of \eqref{def-Z-Zb} as multipliers.
\end{remark}

The deformation tensors of $\TL, \, \TLb$ are calculated in Lemma \ref{lem-deform}.

\section{Proof of Theorem \ref{main-theorem1}}\label{sec:3}
In this section, we will prove the main Theorem \ref{main-theorem1} by energy estimates. 

Firstly, for conceptual convenience, we denote the null frame by \[L= \p_{\ub}, \quad \Lb = \p_u.\]

\subsection{Bootstrap arguments}
We let
\bel{def-phi-k}
\phi_K \doteq \p^{k_1}_u \p_{\ub}^{k_2} \phi, \quad (k_1, k_2) = K.
\ee
Also for notational convenience, we will allow ourself abuse the notation, and use $\phi_k, \, k =|K|$ to denote any possible $k$th-order derivatives of $\phi$. By this notation, $\phi_0 = \phi$.

Define
\begin{subequations}
\begin{align}
F^2_{(k+1)}(u,\tau) & = \sum_{|K| = k} \int_{C^\tau_{u}} \Lambdab(\ub) (|L \phi_{K}|^2+|L\phi|^4|\Lb\phi_K|^2)  \sqrt{g}  \di \ub, \label{2.17-1} \\
\Fb^2_{(k+1)}( \ub, \tau) & =  \sum_{|K| = k} \int_{\Cb^\tau_{\ub}} \Lambda (u) (|\Lb \phi_{K}|^2 +  (2 \Psip (u) + \Lb \phi)^2  |\Lb \phi|^2 |L\phi_K|^2) \sqrt{g} \di u,  \label{2.17} \\
E^2_{(k+1)}(u,\tau) & =  \sum_{|K| = k} \int_{\Sigma^{+}_{\tau, u}} \Lambdab(\ub) \left( [1+ (\Psip (u) + \Lb\phi )^2] |L\phi_{K}|^2+|L\phi|^2|\Lb \phi_{K}|^2 \right) \sqrt{g} \di x, \label{2.18} \\
\Eb^2_{(k+1)}( \ub,\tau) & =  \sum_{|K| = k} \int_{\Sigma^{-}_{\tau, \ub}} \Lambda (u) ( |\Lb \phi_{K}|^2 + \left( [1+(\Psip(u))^2 ] |\Lb \phi|^2 + |g^{\ub \ub}|^2 \right)  |L \phi_{K}|^2) \sqrt{g} \di x. \label{2.19}
\end{align}
\end{subequations}
And we let
\begin{align}
E^2_{(k+1)}(\tau) ={}& \sup_{u \in \mathbb{R}} E^2_{(k+1)}(u, \tau), \label{def-E-t} \\
\Eb^2_{(k+1)}(\tau) = {}& \sup_{\ub \in \mathbb{R}} \Eb^2_{(k+1)}(\ub, \tau). \label{def-Eb-t}
\end{align}
The energies defined above are related to the left hand sides of the energy identities \eqref{energy-id-1} with $\xi=\TLb$, \eqref{energy-id-2} with $\xi = \TL$, respectively. This will be confirmed in Lemma \ref{lemma-energy-formula}.
 We also drop the bracket in the subscript to denote the inhomogeneous energy
\begin{align}
E^2_{N+1}(t) & =\sum_{k \leq N} E^2_{(k+1)}(t) \label{def-E-IH}.
\end{align}
Analogous notations apply to $\Eb^2_{N+1}(t), \, F^2_{N+1}(u,t), \, \Fb^2_{N+1} (\ub, t)$ as well.

\emph{Bootstrap assumptions}:
Given any $N \in \mathbb{N}$, $N \geq 5$, we assume
\begin{align}
& \Eb^2_{N+1}(t )+ \sup_{\ub \in \mathbb{R}} \Fb^2_{N+1}( \ub,t) \leq M^2, \label{bt-Eb-Fb}\\
& E^2_{N+1}(t) + \sup_{u \in \mathbb{R}} F^2_{N+1}(u,t) \leq \delta^2 M^2, \label{bt-E-F}
\end{align}
for all $t \geq 0$, where $M$ is a large constant to be determined later. In fact, it will be clear from the proof that the bootstrap assumptions \eqref{bt-Eb-Fb}-\eqref{bt-E-F} can be reduced as
\als{
& \sum_{|K| \leq N} \int_{\Cb^t_{\ub}} \Lambda (u) |\Lb \phi_{K}|^2 \sqrt{g} \di u  +\int_{\Sigma_t }  \Lambda (u)  |\Lb \phi_{K}|^2  \sqrt{g}  \di x \leq  M^2,\\
& \sum_{|K| \leq N} \int_{C^t_{u}} \Lambdab(\ub) |L \phi_{K}|^2 \sqrt{g}  \di \ub + \int_{\Sigma_t } \Lambdab(\ub)  |L \phi_{K}|^2  \sqrt{g}  \di x \leq  \delta^2 M^2.
}
\begin{remark}
The non-small setting we proposed in this paper is similar to that in \cite{Wang-Wei1}. Therefore, a hierarchy of energy estimates between the $E, \, F$ and $\Eb, \, \Fb$ is needed. Namely, the energy bound for $\Eb, \, \Fb$ energies will be improved at first, and hence we can make use of the updated $\Eb, \, \Fb$ energies to further improve the $E, \, F$ energies.
\end{remark}

Let us now state the bootstrap argument. We will prove the following proposition.
\begin{proposition}\label{pro}
Let $N\geq 5$, under the bootstrap assumptions \eqref{bt-Eb-Fb} and \eqref{bt-E-F}, there exists some positive constant $C_1$ such that
\begin{subequations}
\begin{align}
\Eb^2_{N+1} (\ub,t)+ \Fb^2_{N+1}(\ub, t) & \leq I_{N+1}^2 + C_1\delta M^4,\label{energy1}\\
E^2_{N+1}(u,t) + F^2_{N+1}(u, t) & \leq  \delta^2 I^2_{N+1} +C_1 \delta^3 M^6,\label{energy2}
\end{align}
\end{subequations}
for all $t \geq 0$ and $u, \, \ub \in \mathbb{R}$. Here $I_{N+1}$ is a constant depending only on the initial data (up to $N+1$ order derivatives).
\end{proposition}

Once Proposition \ref{pro} is proved, we can improve the bootstrap assumptions \eqref{bt-Eb-Fb}-\eqref{bt-E-F} as
\begin{align}
& \Eb^2_{N+1}(t )+ \sup_{\ub \in \mathbb{R}} \Fb^2_{N+1}(\ub, t) \leq \frac{M^2}{2}, \label{bt-Eb-Fb-improve}\\
& E^2_{N+1}(t) + \sup_{u \in \mathbb{R}} F^2_{N+1}(u, t) \leq \frac{\delta^2 M^2}{2}, \label{bt-E-F-improve}
\end{align}
and close the bootstrap arguments as follow. Concretely speaking,
let $[0, T^\ast]$ be the largest time interval so that  the bootstrap assumptions \eqref{bt-Eb-Fb} and \eqref{bt-E-F} hold true. By the local well poseness for wave equations, we know that $T^\ast >0$. The improvements \eqref{bt-Eb-Fb-improve} and \eqref{bt-E-F-improve} indicate that the estimates \eqref{bt-Eb-Fb}, \eqref{bt-E-F} can be extended to a larger time interval, thus contradicting the maximality of $T^\ast$. And hence we must have $T^\ast = +\infty$.

To achieve the improvement, we choose the constant $M$ large enough so that $I_{N+1}^2 \leq \frac{M^2}{4}$ (note that, $M$ depends only on the initial data, in particular, it is independent of $\delta$), and $\delta$ small enough so that $C_1\delta M^4\leq \frac{M^2}{4}$. Thus, \eqref{bt-Eb-Fb-improve} and \eqref{bt-E-F-improve} hold true.

\subsection{Preliminary estimates}
As a consequence of the bootstrap assumptions and the decay assumption\footnote{Throughout this paper, the  decay assumption \ref{Ass-2} will be repeatedly used without mention.} \ref{Ass-2} for $\Psi^{(i)} (u), \, i \geq 1$, we have the following $L^\infty$ bounds.
\begin{lemma}\label{lem-sobolev}
 With the bootstrap assumptions \eqref{bt-Eb-Fb} and \eqref{bt-E-F}, we have,
\bel{Sobolev-L-infty}
|L\phi_k| \lesssim \frac{\delta M}{\Lambdab^{\frac{1}{2}}(\ub)}, \quad |\Lb \phi_k| \lesssim \frac{M}{\Lambda^{\frac{1}{2}}(u)},
\quad k \leq N-1.
\ee
\end{lemma}
\begin{proof}
We note that $\|f\|_{L^\infty (\Sigma_t)} \lesssim \|f\|_{L^2 (\Sigma_t)} + \|\p_x f\|_{L^2 (\Sigma_t)}$ and \[\p_x = -\p_u + \frac{1+(\Psip(u))^2}{2} \p_{\ub}. \]
Then the proof follows, see \cite{Luli-Yang-Yu, Wang-Wei1}.
\end{proof}

Another application of the Sobolev inequality is the following weighted Sobolev inequalities, which will be useful in the proof of energy estimates.
\begin{lemma}\label{lem-sobolev1}
There are  the weighted Sobolev inequalities,
\begin{align}
\Big\| \frac{\Lambdab^{\alpha}(\ub)}{\Lambda^{\beta} (u)} L \phi_k \Big\|_{L^\infty(\Sigma^+_{t, u_0})} \lesssim \Big\| \frac{\Lambdab^{\alpha}(\ub)}{\Lambda^{\beta} (u)} L \phi_k\Big\|_{L^2(\Sigma^+_{t, u_0})} + \Big\| \frac{\Lambdab^{\alpha}(\ub)}{\Lambda^{\beta} (u)} L \phi_{k+1} \Big\|_{L^2(\Sigma^+_{t, u_0})}, \label{Sobolev-L-infty-weight-L} \\
\Big\| \frac{\Lambda^{\alpha}(u)}{\Lambdab^{\beta} (\ub)}  \Lb \phi_k \Big\|_{L^\infty(\Sigma^-_{t, \ub_0})} \lesssim \Big\| \frac{\Lambda^{\alpha}(u)}{\Lambdab^{\beta} (\ub)}  \Lb \phi_k \Big\|_{L^2(\Sigma^-_{t, \ub_0})} + \Big\| \frac{\Lambda^{\alpha}(u)}{\Lambdab^{\beta} (\ub)}  \Lb \phi_{k+1} \Big\|_{L^2(\Sigma^-_{t, \ub_0})}, \label{Sobolev-L-infty-weight-Lb}
\end{align}
for all $u_0, \ub_0 \in \mathbb{R}$ and $\alpha, \, \beta \in \mathbb{R}^{+}$, $k \leq N-1$.
\end{lemma}
The proof of this lemma is straightforward and can be found in \cite{Luli-Yang-Yu, Wang-Wei1}.


As a consequence of the $L^\infty$ bound \eqref{Sobolev-L-infty}, the following equivalence of energies are derived. These energies all come from the left hand sides of the energy identities \eqref{energy-id-1}-\eqref{energy-id-2}, where we will take $\xi = \TLb$ and $\xi = \TL$ respectively, and let $\psi = \phi_k$.
\begin{lemma}\label{lemma-energy-formula}
Under the bootstrap assumptions \eqref{bt-Eb-Fb} and \eqref{bt-E-F}, we have $g\sim1$ and
\begin{align}
\int_{C_u^{t}}\T(-Du, \tilde{L})\sqrt{g}\di \ub&\sim\int_{C_u^{t}}\Lambdab(\ub)\left(|L\phi_k|^2+|L\phi|^4|\Lb\phi_k|^2\right)\sqrt{g} \di \ub,  \label{cal-F} \\
\int_{\Cb_{\ub}^t}\T(-D\ub,\tilde{\Lb})\sqrt{g}\di u&\sim \int_{\Cb_{\ub}^t}\Lambda(u)\left(|\Lb\phi_k|^2+  (2 \Psip (u) + \Lb \phi)^2  |\Lb \phi|^2|L\phi_k|^2\right)\sqrt{g} \di u, \label{cal-Fb}
\end{align}
and
\als{
\int_{\Sigma^{-}_{t,\ub}}\T(-Dt,\tilde{\Lb})\sqrt{g} \di x& = \int_{\Sigma^{-}_{t,\ub}}\Lambda(u)\left( |\Lb\phi_k|^2+  \left(  \frac{1+(\Psip(u))^2}{4}  |\Lb\phi|^2 + \frac{1}{2} |g^{\ub \ub}|^2 \right) |L\phi_k|^2\right) \sqrt{g} \di x\\
& + \underline{e}_{(k+1)} (\ub, t),\\
\int_{\Sigma^{+}_{t,u}}\T(-Dt,\tilde{L})\sqrt{g}\di x&=
\int_{\Sigma^{+}_{t,u}}  \frac{1}{2} \Lambdab(\ub)\left( [1+ (\Psip (u) + \Lb\phi )^2] |L\phi_k|^2+|L\phi|^2|\Lb\phi_k|^2\right)\sqrt{g}\di x \\
& + e_{(k+1) } (u, t),
}
where
 \[ |e_{(k+1)} (u, t) | \lesssim \delta^3 M^6, \quad  | \underline{e}_{(k+1)} (\ub, t) |  \lesssim  \delta M^4. \]

\end{lemma}

\begin{proof}
We find that by direct calculations
\als{
& \T[\psi](D \ub, D\ub) = \T_{\mu \nu} [\psi] g^{\ub \mu} g^{\ub \nu} \\
={}& (g^{\ub \ub} \p_{\ub} \psi + g^{\ub u} \p_u \psi )^2 - \frac{1}{2} g^{\ub \ub} (g^{uu} (\p_u \psi)^2 + 2g^{u \ub} \p_{\ub} \psi \p_{u} \psi + g^{\ub \ub} (\p_{\ub} \psi)^2 ) \\
={}& (g^{\ub u} \p_{u} \psi )^2 + \frac{1}{2} (g^{\ub \ub}\p_{\ub} \psi )^2 + g^{\ub \ub} g^{\ub u} \p_{\ub} \psi \p_u \psi - \frac{1}{2} g^{\ub \ub} g^{u u} (\p_u \psi)^2.
}
With the help of the Cauchy-Schwarz inequality \[g^{\ub \ub} g^{\ub u} \p_{\ub} \psi \p_u \psi \geq - \frac{1}{2} \left( \frac{3}{2} (g^{\ub u} \p_{u} \psi)^2 + \frac{2}{3} (g^{\ub \ub} \p_{\ub} \psi)^2 \right), \] there is
\[ \T [\psi] (D \ub, D\ub)  \geq  \frac{1}{4} (g^{\ub u} \p_{u} \psi )^2 + \frac{1}{6} (g^{\ub \ub} \p_{\ub} \psi )^2 - \frac{1}{2} g^{\ub \ub} g^{u u} (\p_u \psi)^2. \]
Noting that $g^{u u} \sim (\p_{\ub} \phi)^2$  is small, we have
\[ \frac{1}{6} (g^{\ub u} \p_{u} \psi )^2 + \frac{1}{6} (g^{\ub \ub} \p_{\ub} \psi )^2 \leq \T[\psi](D \ub, D \ub). \]

In the same way, we have
\als{
 \T[\psi](D u, D u)
={}& (g^{u \ub} \p_{\ub} \psi )^2 + \frac{1}{2} (g^{u u} \p_{u} \psi )^2 + g^{u u} g^{u \ub} \p_{u} \psi \p_{\ub} \psi - \frac{1}{2} g^{u u} g^{\ub \ub} (\p_{\ub} \psi)^2,
}
and obtain
\[ \frac{1}{6} (g^{u \ub} \p_{\ub} \psi )^2 + \frac{1}{6} (g^{u u} \p_{u} \psi )^2 \leq \T[\psi](D u, D u ). \]

Noting that,
\als{
\cg |g^{\ub \ub}|
={}& | (2 \Psip (u) + \p_u \phi)  \p_u \phi| \cdot | 1 - \Psip (u) \p_{\ub} \phi (2-  \Psip (u) \p_{\ub} \phi) | \\
\sim {} & | (2 \Psip (u) + \p_u \phi)  \p_u \phi|,
}
we conclude \eqref{cal-F}-\eqref{cal-Fb}.

We also calculate $\T[\psi](D u, D\ub)$ for later use,
\als{
& \T[\psi](D u, D\ub) = \T_{\mu \nu} [\psi] g^{u \mu} g^{\ub \nu} \\
={}& \frac{1}{2} g^{u \ub} g^{u u} (\p_{u} \psi)^2 + \frac{1}{2} g^{u \ub} g^{\ub \ub} (\p_{\ub} \psi)^2 + g^{u u} g^{\ub \ub} \p_u \psi \p_{\ub} \psi.  
}

To prove the other two identities, we note that
$$
Dt=\frac{\partial t}{\p u}Du+\frac{\p t}{\p \ub}D\ub=\frac{1+(\Psip(u))^2}{2} Du+D\ub,
$$
which combining with the preceding computations yields
\als{
 \T[\phi_k](-Dt,\tilde{L}) = & \frac{1+(\Psip(u))^2}{2} \Lambdab(\ub) \T[\phi_k](Du, Du) + \Lambdab(\ub) \T[\phi_k](D \ub, Du) \\
 = &  \frac{1+(\Psip(u))^2}{2} \Lambdab(\ub) \left( (g^{u \ub} \p_{\ub} \phi_k )^2 + \frac{1}{2} (g^{u u} \p_{u} \phi_k )^2   \right) \\
 & +  \frac{1+(\Psip(u))^2}{2} \Lambdab(\ub) \left(  g^{u u} g^{u \ub} \p_{u} \phi_k \p_{\ub} \phi_k - \frac{1}{2} g^{u u} g^{\ub \ub} (\p_{\ub} \phi_k)^2 \right) \\
 & + \Lambdab(\ub)  \left( \frac{1}{2} g^{u \ub} g^{u u} (\p_{u} \phi_k)^2 + \frac{1}{2} g^{u \ub} g^{\ub \ub} (\p_{\ub} \phi_k)^2 + g^{u u} g^{\ub \ub} \p_u \phi_k \p_{\ub} \phi_k \right)  \\
 ={}&   \frac{\Lambdab(\ub)  }{2} (g^{u \ub} \p_{\ub} \phi_k )^2 + \frac{1+(\Psip(u))^2}{4} \Lambdab(\ub) (g^{u u} \p_{u} \phi_k )^2 -  \frac{ \Lambdab(\ub)}{2} \frac{ \cg }{g} g^{u \ub} (\p^u_{\cg} \phi)^2 (\p_{u} \phi_k)^2  \\
 & +  \Lambdab(\ub)  \left(  \frac{ (\Psip(u))^2}{2}(g^{u \ub})^2
 -  \frac{\cg}{2 g} g^{u \ub} (\p^{\ub}_{\cg} \phi)^2 - \frac{1}{\cg} g^{u \ub} \Psip (u) \p_u \phi  \right) ( \p_{\ub} \phi_k )^2 \\
 & +  \frac{1+(\Psip(u))^2}{2} \Lambdab(\ub) \left(  g^{u u} g^{u \ub} \p_{u} \phi_k \p_{\ub} \phi_k  - \frac{1}{2} g^{u u} g^{\ub \ub} (\p_{\ub} \phi_k)^2 \right)\\
  &+ \Lambdab(\ub)  \left(   g^{u u} g^{\ub \ub} \p_u \phi_k \p_{\ub} \phi_k \right).
 }
 Notice that, the first line of the last equality above contains non-negative terms, while in the last line are terms with indeterminate (or negative) signs, and crucially, in the second line
 \als{
 &\left(  \frac{ (\Psip(u))^2}{2}(g^{u \ub})^2
 -  \frac{\cg}{2 g} g^{u \ub} (\p^{\ub}_{\cg} \phi)^2 - \frac{1}{\cg} g^{u \ub} \Psip (u) \p_u \phi  \right) \\
= & \left(  \frac{ (\Psip(u))^2}{2}
+ \frac{1}{2} (\p_{u} \phi)^2 +  \Psip (u) \p_u \phi  \right) \pm \delta M^4 \\
=&  \frac{1}{2} \left(  \Psip(u) +\p_{u} \phi \right)^2 \pm \delta M^4,
 }
 where the leading term is non-negative and $\pm \delta M^4$ denotes terms that can be bounded by $\delta M^4$. It then holds that
 \als{
\int_{\Sigma^{+}_{t,u}} \T[\phi_k ](-Dt,\tilde{L})\sqrt{g} \di x =
& \int_{\Sigma^{+}_{t,u}} \frac{1 }{2} \Lambdab(\ub) \left( ( \p_{\ub} \phi_k )^2 + (\p_{\ub} \phi)^2 (\p_{u} \phi_k)^2 \right)  \sqrt{g} \di x  \\
& + \int_{\Sigma^{+}_{t,u}} \frac{1}{2} \Lambdab(\ub)  \left(  \Psip(u) +\p_{u} \phi \right)^2 ( \p_{\ub} \phi_k )^2 \sqrt{g} \di x  + e_{(k+1)} (u, t), \\
e_{(k+1)} (u, t) = & \int_{\Sigma^{+}_{t,u}}  \frac{1+(\Psip(u))^2}{2} \Lambdab(\ub)   g^{u u} \left( g^{u \ub} \p_{u} \phi_k \p_{\ub} \phi_k  - \frac{1}{2} g^{\ub \ub} (\p_{\ub} \phi_k)^2 \right)\sqrt{g} \di x  \\
&+\int_{\Sigma^{+}_{t,u}}  \Lambdab(\ub)  \left(   g^{u u} g^{\ub \ub} \p_u \phi_k \p_{\ub} \phi_k \right)\sqrt{g} \di x  \pm \delta^3 M^6,
}
and resorting to the bootstrap assumptions, \[ |e_{(k+1)} (u, t) | \lesssim \delta^2 M^4 \left( E_{(k+1)} (t) \Eb_{(k+1)} (t) + E^2_{(k+1)} (t) \right) + \delta^3 M^6 \lesssim \delta^3 M^6. \]

Next, we conduct an analogous calculation for $ \T[\phi_k](-Dt,\TLb)$,
  \als{
 \T[\phi_k](-Dt,\TLb)
  =&  \Lambda(u) (g^{\ub u} \p_{u} \phi_k )^2 +   \frac{1+(\Psip(u))^2}{4}  \Lambda(u)  g^{u \ub} g^{u u} (\p_{u} \phi_k)^2 \\
  &+  \frac{1}{2} \Lambda(u)   (g^{\ub \ub}\p_{\ub} \phi_k )^2 -  \frac{(1+(\Psip(u))^2 ) \cg}{4 g}  \Lambda(u) g^{u \ub} (\p^{\ub}_{ \cg} \phi)^2 (\p_{\ub} \phi_k)^2  \\
  & -   \frac{1+(\Psip(u))^2}{2 \cg}  \Lambda(u) g^{u \ub} \Psip (u) \p_u \phi (\p_{\ub} \phi_k)^2\\
  &+   \Lambda(u) \left(   \frac{1+(\Psip(u))^2}{2} g^{u u} g^{\ub \ub} \p_u \phi_k \p_{\ub} \phi_k  +  g^{\ub \ub} g^{\ub u} \p_{\ub} \phi_k \p_u \phi_k - \frac{1}{2} g^{\ub \ub} g^{u u} (\p_u \phi_k)^2   \right),
  }
  where the first two lines have only non-negative terms and in the last two lines are terms with indeterminate signs.

As a result,
 \als{
& \int_{\Sigma^{-}_{t,\ub}} \T[\phi_k](-Dt,\tilde{\Lb}) \sqrt{g} \di x
=  \int_{\Sigma^{-}_{t,\ub}}  \Lambda(u) \left( ( \p_{u} \phi_k )^2 +   \frac{1+(\Psip(u))^2}{4}  (\p_{\ub} \phi)^2 (\p_{u} \phi_k)^2 \right)  \sqrt{g} \di x \\
  &\qquad + \int_{\Sigma^{-}_{t,\ub}}  \frac{\Lambda(u)  }{2} \left( (g^{\ub \ub}\p_{\ub} \phi_k )^2 + \frac{1+(\Psip(u))^2 }{2 }    (\p_{u} \phi)^2 (\p_{\ub} \phi_k)^2 \right)  \sqrt{g} \di x + \underline{e}_{(k+1)} (\ub, t),
  }
  with
  \als{
   & \underline{e}_{(k+1)}  (\ub, t)=  \pm \delta^2 M^4 - \int_{\Sigma^{-}_{t,\ub}}  \frac{1+(\Psip(u))^2}{2 \cg}  \Lambda(u) g^{u \ub} \Psip (u) \p_u \phi (\p_{\ub} \phi_k)^2 \sqrt{g} \di x\\
  &\quad + \int_{\Sigma^{-}_{t,\ub}}  \Lambda(u) g^{\ub \ub} \left(   \frac{1+(\Psip(u))^2}{2} g^{u u} \p_u \phi_k \p_{\ub} \phi_k  + g^{\ub u} \p_{\ub} \phi_k \p_u \phi_k - \frac{1}{2} g^{u u} (\p_u \phi_k)^2   \right) \sqrt{g} \di x,
}
and \als{
  | \underline{e}_{(k+1)}  (\ub, t) |  \lesssim {}& \delta^2 M^4 +  M E^2_{k+1} (t) + M^2 ( E_{k+1} (t) \Eb_{k+1} (t) + \delta^2 M^2 \Eb^2_{k+1} (t) ) \\ \lesssim {}&\delta M^4.
  }
 We complete the proof.
\end{proof}

In view of the formula \eqref{eq-div-P} and the energy inequalities \eqref{eq-energy-TLb}, \eqref{eq-energy-TL} in the later subsections, we will need to estimate the deformation tensors of $\TL$ and $\TLb$ respectively.

\begin{lemma}\label{lem-deform-esti}
We have the following estimates,
\als{
& |\T^{\alpha}_{\beta} [\phi_k] \p_\alpha \TL^{\beta}| + \frac{1}{2\sqrt{g}} |\TL(\sqrt{g}g^{\gamma\rho}) \p_{\gamma}\phi_k \p_{\rho}\phi_k| \\
 \lesssim {} & \Lambdab (\ub) [  (\Psip (u))^2 + (\p_{u} \phi)^2 + (\p^2_{u} \phi)^2 + (\p_u \p_{\ub} \phi)^2 + |\Psi^{\prime \prime} (u) \p_{\ub} \phi| +  |\p_u^2 \phi \p_{\ub} \phi| ] (\p_{\ub} \phi_k)^2 \\
 & + \Lambdab (\ub) [ (\p_{\ub} \phi)^2+ (\p^2_{\ub} \phi)^2 +  ( \p_{\ub} \p_u \phi)^2 ]  (\p_u \phi_k)^2,
}
and
\als{
& |\T^{\alpha}_{\beta} [\phi_k] \cdot \p_{\alpha}\TLb^{\beta}| + \frac{1}{2\sqrt{g}} |\TLb(\sqrt{g}g^{\gamma\rho}) \p_{\gamma}\phi_k \p_{\rho}\phi_k|\\
  \lesssim {}& \Lambda (u) [ (\Psip (u))^2 + (\p_{u} \phi)^2 +  (\p_{u}^2 \phi)^2 + (\Psi^{\prime \prime} (u))^2 ]  (\p_{\ub} \phi_k)^2  \\
&+ \Lambda (u) [ (\p_{\ub} \phi)^2 + (\p_{\ub} \p_u \phi)^2 +  (\p^2_{\ub} \phi)^2 +  |\p_{\ub}^2 \phi \p_{u} \phi| + |\p_{\ub}^2 \phi \Psip (u)| ] (\p_u \phi_k)^2,
 }
if the decay assumption \ref{Ass-2} and the bootstrap assumptions \eqref{bt-Eb-Fb}-\eqref{bt-E-F} hold true.

\end{lemma}
\begin{remark}
Lemma \ref{lem-deform-esti} shows that terms concerning the deformation tensors obey essentially the ``null condition''.
\end{remark}

\begin{proof}
Based on the computations in Lemma \ref{lem-deform}, we write down  the leading terms in $ \T^{\alpha}_{\beta} [\phi_k]  \p_\alpha \TL^{\beta} $,
\als{
 \T^{\alpha}_{\beta} [\phi_k]  \p_\alpha \TL^{\beta} & \sim  \Lambdab^\prime (\ub) [ (\p_{\ub} \phi)^2 (\p_u \phi_k)^2 +   (\p_{u} \phi)^2 (\p_{\ub} \phi_k)^2 +  |\Psip (u) \p_u \phi| (\p_{\ub} \phi_k)^2 ] \\
 &+ \Lmdub [ (\p_{\ub} \phi_k)^2 \left(  \p_u^2 \phi \p_{\ub} \phi + \p_u \phi \p_u \p_{\ub} \phi + \cg^{-1} g^{u \ub} \p_{u} ( \Psip (u) \p_{\ub} \phi) \right) ] \\
 &+ \Lmdub [ (\p_u \phi_k)^2 \p_{\ub}^2 \phi \p_{\ub} \phi + \cg^{-1} \partial_{u}\phi_k\partial_{\ub}\phi_k g^{u u} \p_{u} ( \Psip (u) \p_{\ub} \phi) ],
}
and in $\T^{\alpha}_{\beta} [\phi_k]  \p_\alpha \TLb^\beta$,
 \als{
\T^{\alpha}_{\beta} [\phi_k]  \p_\alpha \TLb^\beta \sim {}& \Lambda^\prime (u) [ (\p_{\ub} \phi)^2 (\p_u \phi_k)^2 + (\p_{u} \phi)^2 (\p_{\ub} \phi_k)^2 + |\Psip (u) \p_u \phi| (\p_{\ub} \phi_k)^2 ] \\
&+  \Lambda (u) \partial_{u}\phi_k\partial_{\ub}\phi_k  \p_{u} \phi ( \Psip (u))^2 \p^2_{\ub} \phi  +  \Lambda (u) ( \partial_{u} \phi_k )^2 \Psip (u) \p^2_{\ub} \phi \\
& +  \Lambda (u) (\p_{\ub} \phi_k)^2 \left( (\p_{u}^2 \phi)^2 +  (\p_{u} \phi)^2 \right) \\
 &+ \Lambda (u) (\p_{u} \phi_k)^2 \left(  \p_{\ub}^2 \phi \p_{u} \phi + \p_{\ub} \phi \p_u \p_{\ub} \phi \right)  \\
 &+ 2 \Lambda (u) \cg^{-1}  (g^{u u} \partial_{u}\phi_k \partial_{\ub}\phi_k ) \p_{u} ( \Psip (u) \p_{u} \phi) \\
 &+ 2 \Lambda (u) \cg^{-1}  (g^{u \ub}\partial_{\ub} \phi_k \partial_{\ub} \phi_k ) \p_{u} ( \Psip (u) \p_{u} \phi).
}
Then there are the upper bounds,
 \als{
 |\T^{\alpha}_{\beta} [\phi_k]  \p_\alpha \TL^{\beta}|
 \lesssim {} & \Lambdab (\ub) [ (\Psip (u))^2 + (\p_{u} \phi)^2 + (\p^2_{u} \phi)^2 + (\p_u \p_{\ub} \phi)^2 + |\Psi^{\prime \prime} (u) \p_{\ub} \phi| ] (\p_{\ub} \phi_k)^2 \\
 &+ \Lmdub  |\p_u^2 \phi \p_{\ub} \phi|  (\p_{\ub} \phi_k)^2+ \Lambdab (\ub) [ (\p_{\ub} \phi)^2+ (\p^2_{\ub} \phi)^2 +  ( \p_{\ub} \p_u \phi)^2 ]  (\p_u \phi_k)^2, \\
 |\T^{\alpha}_{\beta} [\phi_k] \cdot \p_{\alpha}\TLb^{\beta}|
  \lesssim {}&  \Lambda (u) [ (\p_{\ub} \phi)^2 + (\p_{\ub} \p_u \phi)^2 +  (\p^2_{\ub} \phi)^2  + |\p_{\ub}^2 \phi \p_{u} \phi| + |\p_{\ub}^2 \phi \Psip (u)| ] (\p_{u} \phi_k)^2 \\
&  + \Lambda (u) [  (\Psip (u))^2 + (\p_{u} \phi)^2 +  (\p_{u}^2 \phi)^2 + (\Psi^{\prime \prime} (u))^2 ]  (\p_{\ub} \phi_k)^2.
}

And the leading terms in $\frac{1}{\sqrt{g}}\TL(\sqrt{g}g^{\gamma\rho}) \p_{\gamma}\phi_k \p_{\rho}\phi_k$ and $\frac{1}{\sqrt{g}}\TLb(\sqrt{g}g^{\gamma\rho}) \p_{\gamma}\phi_k \p_{\rho}\phi_k$ are given as below,
\begin{align*}
& \frac{1}{\sqrt{g}}\TL(\sqrt{g}g^{\gamma\rho}) \p_{\gamma}\phi_k \p_{\rho}\phi_k \\
\sim {}& \Lmdub \p_{\ub} (2 \Psip (u) \p_u \phi + (\p_u \phi)^2) (\p_{\ub} \phi_k)^2 \pm  \Lmdub \Psip (u) \p^2_{\ub} \phi \p_{\ub} \phi_k \p_u \phi_k\\
 & \pm  \Lmdub \p_{\ub} (\p_u \phi \p_{\ub} \phi) \p_{\ub} \phi_k \p_u \phi_k +  \Lmdub \p_{\ub} \p_{\ub} \phi \p_{\ub} \phi (\p_{u} \phi_k)^2 +  \Lmdub \p_{\ub} \p_{u} \phi \p_{u} \phi (\p_{\ub} \phi_k)^2, \\
 & \frac{1}{\sqrt{g}}\TLb(\sqrt{g}g^{\gamma\rho}) \p_{\gamma}\phi_k \p_{\rho}\phi_k \\
\sim {}& \Lmdu \p_{u} (2 \Psip (u) \p_u \phi + (\p_u \phi)^2) (\p_{\ub} \phi_k)^2 -\Lmdu  \p_{u} (\Psip (u) \p_{\ub} \phi) \p_{\ub} \phi_k \p_u \phi_k    \\
 & -\Lmdu  \p_{u} (\p_u \phi \p_{\ub} \phi) \p_{\ub} \phi_k \p_u \phi_k+  \Lmdu \p_{\ub} \p_{u} \phi \p_{\ub} \phi (\p_{u} \phi_k)^2 +  \Lmdu \p_{u} \p_{u} \phi \p_{u} \phi (\p_{\ub} \phi_k)^2,
\end{align*}
then
\begin{align*}
& \frac{1}{\sqrt{g}} |\TL(\sqrt{g}g^{\gamma\rho}) \p_{\gamma}\phi_k \p_{\rho}\phi_k| \\
 \lesssim {} & \Lambdab (\ub) [ (\Psip (u))^2 + (\p_{u} \phi)^2 + (\p^2_{u} \phi)^2 + (\p_u \p_{\ub} \phi)^2  ] (\p_{\ub} \phi_k)^2 \\
 &+ \Lambdab (\ub) [ (\p_{\ub} \phi)^2+ (\p^2_{\ub} \phi)^2 +  ( \p_{\ub} \p_u \phi)^2 ]  (\p_u \phi_k)^2, \\
& \frac{1}{\sqrt{g}} |\TLb(\sqrt{g}g^{\gamma\rho}) \p_{\gamma}\phi_k \p_{\rho}\phi_k| \\
 \lesssim {}& \Lambda (u) [ (\Psip (u))^2 + (\Psi^{\prime \prime} (u))^2 + (\p_{u} \phi)^2 +  (\p_{u}^2 \phi)^2 ]  (\p_{\ub} \phi_k)^2 \\
 & + \Lambda (u) ((\p_{\ub} \phi)^2 + ( \p_{\ub} \p_u \phi)^2 ) (\p_{u} \phi_k)^2.
\end{align*}

Putting all these estimates together, this lemma is concluded.
\end{proof}

\subsection{Energy estimates}
We are now in a position to prove Proposition \ref{pro}.

As what we have explained before, we need to control the energies $\Eb$ and $\Fb$ at first.
\subsubsection{Estimates of the energies $\Eb_{N+1} $ and $\Fb_{N+1} $.}

Letting $\xi = \TLb$, $\psi = \phi_k$ in \eqref{energy-id-1}, we have the energy inequality
\be\label{eq-energy-TLb}
\Eb^2_{(k+1)}(\ub,t)+ \Fb^2_{(k+1)}( \ub,t) \lesssim \Eb^2_{(k+1)}(\ub, 0) + \iint_{\D^{-}_{ \ub,t}} |\p_{\alpha}(\sqrt{g}P^{\alpha} [\phi_k, \TLb])| \di u \di \ub +  \delta M^4.
\ee
Noting \eqref{eq-div-P} and hence
\begin{align*}
\frac{1}{\sqrt{g}}\p_{\alpha}(\sqrt{g}P^{\alpha}[\phi_k, \TLb])= \Box_{g(\p\phi)}\phi_k \cdot \TLb\phi_k + \T^{\alpha}_{\beta} [\phi_k] \cdot \p_{\alpha}\TLb^{\beta} - \frac{1}{2\sqrt{g}}\TLb(\sqrt{g}g^{\gamma\rho}) \p_{\gamma}\phi_k \p_{\rho}\phi_k,
\end{align*}
we will bound the nonlinear error terms on the right hand side in what follows.

\emph{The deformation tensor}: \[\T^{\alpha}_{\beta} [\phi_k] \cdot \p_{\alpha}\TLb^{\beta} - \frac{1}{2\sqrt{g}}\TLb(\sqrt{g}g^{\gamma\rho}) \p_{\gamma}\phi_k \p_{\rho}\phi_k.\]
It  has been concluded in Lemma \ref{lem-deform-esti} that
 \begin{equation}\label{T-Xb}
 \begin{split}
& |\T^{\alpha}_{\beta} [\phi_k] \cdot \p_{\alpha}\TLb^{\beta}| + \frac{1}{2\sqrt{g}} |\TLb(\sqrt{g}g^{\gamma\rho}) \p_{\gamma}\phi_k \p_{\rho}\phi_k|\\
  \lesssim {}& \Lambda (u) [ (\Lb \phi)^2 +  (\Lb^2 \phi)^2 + (\Psip (u))^2 + (\Psi^{\prime \prime} (u))^2 ]  (L \phi_k)^2  \\
&+ \Lambda (u) [ (L \phi)^2 + (L \Lb\phi)^2 +  (L^2 \phi)^2 +  |L^2 \phi \Lb \phi|  + |L^2 \phi \Psip (u)|] (\Lb \phi_k)^2.
 \end{split}
 \end{equation}

For the first line on the right hand side of \eqref{T-Xb},
\begin{align}\label{deform-Lb-1}
& \iint_{D^{-}_{t,\ub}} \Lambda (u) |\Lb \phi|^2  |L \phi_k|^2 \sqrt{g} \di u \di \ub = \iint_{D^{-}_{t,\ub}} \Lambda (u)  |\Lb \phi|^2  |L \phi_k|^2 \sqrt{g} \di \tau \di x \nnb\\
\lesssim {} & \sup_\tau \int_{\Sigma^{-}_{\tau, \ub}} \Lambdab (\ub) |L \phi_k|^2\sqrt{g} \di x \int_0^t \Big\|\frac{\Lambda^{\frac{1}{2}}(u)}{\Lambdab^{\frac{1}{2}}(\ub)} \Lb \phi \Big\|^2_{L^\infty (\Sigma^{-}_{\tau, \ub})} \di \tau \nnb\\
\lesssim {}& \sup_\tau \int_{\Sigma_{\tau}} \Lambdab (\ub) |L \phi_k|^2 \sqrt{g} \di x \int_0^t \sum_{i\leq 1} \Big\|\frac{\Lambda^{\frac{1}{2}}(u)}{\Lambdab^{\frac{1}{2}}(\ub)} \Lb \phi_i \Big\|^2_{L^2 (\Sigma^{-}_{\tau, \ub})}  \di \tau\nnb \\
\lesssim {}& \sup_\tau E^2_{k+1}(\tau) \iint_{\D^{-}_{t, \ub}} \frac{\Lambda(u)}{\Lambdab(\ub)} (|\Lb \phi|^2 + |\Lb \phi_1|^2)\sqrt{g} \di \tau \di x \nnb \\
\lesssim {}& \delta^2 M^2 \int^{\ub}_{-\infty} \frac{1}{\Lambdab(\ub^\prime)} \di \ub^\prime \int_{\Cb^t_{\ub^\prime}} \Lambda (u) (|\Lb \phi|^2 + |\Lb \phi_1|^2)\sqrt{g} \di u \nnb\\
\lesssim {}& \delta^2 M^4.
\end{align}
In the same way, $$\iint_{D^{-}_{t,\ub}} \Lambda (u)  |\Lb^2 \phi|^2 |L \phi_k|^2  \sqrt{g} \di u \di \ub \lesssim \delta^2 M^4$$ follows straightforwardly.
Next, since $|\Psip (u)| + |\Psi^{\prime \prime} (u) | \leq \Lambda^{-\frac{1}{2}} (u) \langle u \rangle^{-\frac{1+\epsilon}{2}}$,
\begin{align}\label{extra-deform}
&\iint_{D_{t,\ub}^{-}} \Lambda (u) \left( |\Psip (u)|^2 + |\Psi^{\prime \prime} (u)|^2 \right)  (L\phi_k)^2   \sqrt{g} \di u \di \ub \nnb  \\
\lesssim  &\iint_{D_{t}} \langle u \rangle^{-(1+\epsilon)} (L\phi_k)^2  \sqrt{g} \di u \di \ub \nnb \\
 \lesssim & \int^{+\infty}_{-\infty} \frac{1}{\langle u\rangle^{1+\epsilon}} \di u \int_{C^t_{u}}  |L \phi_k|^2 \sqrt{g} \di \ub \lesssim \delta^2 M^2.
\end{align}
And for the second line on the right hand side of \eqref{T-Xb},
\begin{align}\label{deform-Lb-2}
& \iint_{D_{t,\ub}^{-}} \Lambda (u) |\Lb \phi_k|^2 ( |L \phi|^2 + |L \Lb \phi|^2  + |L^2 \phi|^2 +  |L^2 \phi \Lb \phi| + |L^2 \phi \Psip (u)|) \sqrt{g} \di u \di \ub \nnb \\
\lesssim &{} \int_{-\infty}^{\ub} \left( \frac{\delta^2 M^2 }{\Lambdab(\ub^{\prime}) } + \frac{\delta M^2 }{\Lambdab^{\frac{1}{2}}(\ub^{\prime}) } \right) \di \ub^{\prime}\int_{\Cb_{\ub^{\prime}}} \Lambda (u) |\Lb \phi_k|^2\sqrt{g} \di u \nnb \\
\lesssim {}& \delta^2 M^4 +\delta M^4.
\end{align}

\emph{The source term}:
$$\Box_{g(\p\phi)}\phi_k \cdot \TLb\phi_k =
\Box_{g(\p\phi)}\phi_k \cdot \Lambda(u) \left( (1 + \text{l.o.t.})\Lb \phi_k - g^{\ub \ub} L \phi_k \right).$$

We begin with $\Box_{g(\p\phi)}\phi_k \Lambda(u) \Lb \phi_k$ which constitutes the following types,
\als{
\underline{T}_1 :=
&\sum_{\substack{l+i+j\leq k, \\ \{i,j\}\leq l<k}}  \Lambda(u) \Lb \phi_i L \phi_j \Lb L \phi_l \Lb \phi_k, & & \sum_{\substack{l+i+j\leq k, \\ j\leq i, \,l< i \leq k}}  \Lambda(u) \Lb \phi_i L \phi_j \Lb L \phi_l \Lb \phi_k,  \\
&\sum_{\substack{l+i+j\leq k, \\ j\leq l<k}} \Lambda(u)  \Psi^{(i+1)}(u) L \phi_j \Lb L \phi_l\Lb \phi_k, & & \sum_{\substack{l+i+j\leq k, \\ \{i,j\}\leq l<k}} \Lambda(u) L \phi_i L \phi_j \Lb^2 \phi_{l} \Lb \phi_k,\\
&\sum_{\substack{l+i+j\leq k, \\ l < j \leq k}}\Lambda(u)  \Psi^{(i+1)}(u) \Lb \phi_j L^2 \phi_l \Lb \phi_k, & & \sum_{\substack{l+i+j\leq k, \\ i \leq j, \, l< j \leq k}} \Lambda(u) \Lb \phi_i \Lb \phi_j L^2 \phi_l \Lb \phi_k, \\
\underline{T}_2 :=
& \sum_{\substack{l+i+j\leq k, \\ i\leq j, \, l< j \leq k}}  \Lambda(u) \Lb \phi_i L \phi_j \Lb L \phi_l \Lb \phi_k, & & \sum_{\substack{l+i+j\leq k, \\ l <j \leq k}} \Lambda(u)  \Psi^{(i+1)}(u) L \phi_j \Lb L \phi_l \Lb \phi_k,\\
&\sum_{\substack{l+i+j\leq k, \\ j \leq l<k}}\Lambda(u)  \Psi^{(i+1)}(u) \Lb \phi_j L^2 \phi_l \Lb \phi_k, & & \sum_{\substack{l+i+j\leq k, \\ \{i,j\}\leq l<k}} \Lambda(u) \Lb \phi_i \Lb \phi_j L^2 \phi_l \Lb \phi_k,
}
and
\als{
\underline{T}_3 :={}& \Lambda(u) L \phi_i L \phi_j \Lb^2 \phi_{l} \Lb \phi_k, \quad l+i+j\leq k, \, i \leq j, \, l< j\leq k, \\
\underline{T}_4 :={}& \Lambda(u) \Psi^{(i+2)} (u) L \phi_j L \phi_l \Lb \phi_k, \quad i+j+l\leq k,\,l<k, \\
\underline{T}_5: ={}& \Lambda(u) ( |\Psip (u) L^2 \phi| + |\Lb \phi L^2\phi| + |L \phi \Lb L \phi| ) (\Lb \phi_k)^2, \\
\underline{T}_6 :={}& \Lambda(u) (|\Lb^2 \phi L \phi| + |\Lb \phi L \Lb\phi| + |\Psip(u) L \Lb \phi|)  |L\phi_k \Lb \phi_k|.
}

For $\underline{T}_1$, we can always perform $L^\infty$ bounds on the lower order derivatives, which taking the forms of $\Lb \phi_i L \phi_j, \,L\phi_j \Lb L \phi_l, \,\Psi^{(i+1)} (u) L \phi_j, \, L \phi_i L \phi_j, \, \Psi^{(i+1)} (u) L^2 \phi_l, \, \Lb \phi_i L^2 \phi_l$ will afford the smallness $\delta$ and the factor $\Lambdab^{-\frac{1}{2}}(\ub)$, and perform $L^2$ estimates on the two higher order derivatives both of which present $\Lb$ derivative, such as $\Lb \phi_j \Lb \phi_k$. Consequently, there is
\als{
 \iint_{D_{t,\ub}^{-}} |\underline{T}_1| \sqrt{g} \di u \di \ub \lesssim & \int_{-\infty}^{\ub}\frac{ \delta M+ \delta M^2 + \delta^2 M^2}{ \Lambdab^{\frac{1}{2}}(\ub^{\prime})}\di \ub^{\prime}\int_{\Cb_{\ub^{\prime}}^{t}}\sum_{j \leq k}\Lambda(u) |\Lb\phi_{j}|^2 \sqrt{g} \di u \\
 \lesssim {} &\delta M^4.
}

For $\underline{T}_2$, after applying $L^\infty$ bounds to lower order derivatives $\Lb \phi_i \Lb L\phi_l$, $\Psi^{(i+1)} (u) \Lb L \phi_l $, $\Psi^{(i+1)} (u) \Lb \phi_j$, $\Lb\phi_i\Lb\phi_j$, which are bounded by $M^2 \Lambda^{-1} (u)$, noting that $|\Psi^{(i+1)} (u)| \lesssim \Lambda^{-\frac{1}{2}} (u)$, we obtain
\begin{align}\label{esti-Tb2}
& \iint_{D_{t,\ub}^{-}} |\underline{T}_2| \sqrt{g} \di u \di \ub \lesssim
 M^2  \sum_{j \leq k} \iint_{D_{t,\ub}^{-}}  |L\phi_j||\Lb\phi_k|\sqrt{g} \di u \di \ub \nnb \\
\lesssim & \int_{-\infty}^{\ub}\frac{\delta M^2}{\Lambdab (\ub^{\prime})} \di \ub^{\prime}\int_{\Cb^{t}_{\ub^{\prime}}}\Lambda(u) |\Lb\phi_k|^2 \sqrt{g} \di u
+\int_{-\infty}^{+\infty}\frac{\delta^{-1} M^2}{\Lambda(u^{\prime})} \di u^{\prime}\int_{C^{t}_{u^{\prime}}}\Lambdab(\ub)|L\phi_j|^2\sqrt{g} \di \ub \nnb \\
\lesssim {} &\delta M^4.
\end{align}

For $\underline{T}_3$,
\als{
& \sum_{\substack{l+i+j\leq k,\\ i \leq j , \, l<j \leq k}}\iint_{D_{t,\ub}^{-}} |\Lb^2\phi_l L\phi_i L\phi_j \Lb\phi_k| \Lambda(u)\sqrt{g} \di u \di \ub \\
\lesssim &\iint_{D_{t,\ub}^{-}}\Lambda(u)|L\phi_i|^2|\Lb\phi_k|^2\sqrt{g} \di u \di \ub +
\iint_{D_{t}}\Lambda(u)|\Lb^2\phi_l|^2|L\phi_j|^2\sqrt{g} \di u \di \ub \\
\lesssim {} &\delta^2 M^4, \quad \text{by} \,\, \eqref{deform-Lb-1} \,\, \text{and} \,\,\eqref{deform-Lb-2}.
}

For $\underline{T}_4$, we need the decay estimate $|\Psi^{(i+2)} (u) | \leq \Lambda^{-\frac{1}{2}} (u) \langle u \rangle^{-\frac{1+\epsilon}{2}}$,
\als{
&\sum_{\substack{i+j+l\leq k,\\l<k}} \iint_{D_{t,\ub}^{-}} |\Psi^{(i+2)}(u)L\phi_j L\phi_l \Lb\phi_k| \Lambda(u)\sqrt{g} \di u \di\ub \\
\lesssim&\sum_{\substack{i+j+l\leq k,\\ l<k}} \iint_{D_{t,\ub}^{-}} \Lambda^{-\frac{1}{2}}(u) \langle u\rangle^{-\frac{1+\epsilon}{2}} |L\phi_j L\phi_l\Lb\phi_k| \Lambda(u)\sqrt{g} \di u \di \ub \\
\lesssim &
\sum_{\substack{i+j+l\leq k,\\l <  k}}\iint_{D_t}\frac{\delta M}{\Lambda^{\frac{1}{2}}(u)\Lambdab^{\frac{1}{2}}(\ub)\langle u\rangle^{\frac{1+\epsilon}{2}}}|L\phi_{\max\{j,l\}}||\Lb\phi_k|\Lambda(u)\sqrt{g} \di u \di \ub\\
\lesssim &\int_{-\infty}^{+\infty}\frac{\delta M}{\langle u^\prime \rangle^{1+\epsilon}} \di u^{\prime}\int_{C_{u^{\prime}}^{t}} \sum_{i \leq k} \Lambdab(\ub)|L\phi_{i}|^2\sqrt{g} \di \ub
+\int_{-\infty}^{+\infty}\frac{\delta M}{\Lambdab^2(\ub^{\prime})}d\ub^{\prime}\int_{\Cb_{\ub^{\prime}}^{t}}\Lambda(u)|\Lb\phi_k|^2\sqrt{g} \di u\\
\lesssim {}&\delta M^3.
}

And for $\underline{T}_5$,
\als{
\iint_{D_{t,\ub}^{-}}| \underline{T}_5| \sqrt{g} \di u \di \ub \lesssim &
\iint_{D_{t,\ub}^{-}}\frac{\delta M^2}{\Lambda^{\frac{1}{2}}(u)\Lambdab^{\frac{1}{2}}(\ub)}|\Lb\phi_k|^2\Lambda(u)\sqrt{g}  \di u \di \ub \\
\lesssim &\int_{-\infty}^{\ub}\frac{\delta M^2}{ \Lambdab^{\frac{1}{2}}(\ub^\prime)}\di \ub^{\prime}\int_{\Cb_{\ub^{\prime}}^{t}}\Lambda(u)|\Lb\phi_k|^2\sqrt{g} \di u\\
\lesssim {} &\delta M^4.
}

And finally for $\underline{T}_6$,
 \als{
\underline{T}_6 
\lesssim {} & (|L\phi|^2 + |L \Lb \phi|^2 ) \Lambda(u) |\Lb \phi_k |^2 +  (|\Lb \phi|^2 + |\Psip (u)|^2 + |\Lb^2 \phi|^2) \Lambda(u) |L \phi_k |^2,
}
has been treated in \eqref{deform-Lb-1}-\eqref{deform-Lb-2}.

We next proceed to $$|\Box_{g(\p\phi)}\phi_k \Lambda(u) g^{\ub \ub} L\phi_k| \lesssim M^2 |\Box_{g(\p\phi)}\phi_k L\phi_k|,$$ and $\Box_{g(\p\phi)}\phi_k L\phi_k$ consists of the following types (note that, these terms have no $\Lambda(u)$ factor):
\als{
\underline{N}_1 :=
& \sum_{\substack{l+i+j\leq k, \\ \{i,j\}\leq l<k}}   \Lb \phi_i L \phi_j L \Lb \phi_l L \phi_k, && \sum_{\substack{l+i+j\leq k, \\ j\leq l<k}}   \Psi^{(i+1)}(u) L \phi_j L \Lb \phi_l L \phi_k, \\
& \sum_{\substack{l+i+j\leq k, \\ \{i,l\}\leq j\leq k}}   \Lb \phi_i L \phi_j L \Lb \phi_l L \phi_k, & & \sum_{\substack{l+i+j\leq k, \\ l <j \leq k}}   \Psi^{(i+1)}(u) L \phi_j L \Lb \phi_l L \phi_k,\\
&\sum_{\substack{l+i+j\leq k, \\ j \leq l<k}}  \Psi^{(i+1)}(u) \Lb \phi_j L^2 \phi_l L \phi_k, & & \sum_{\substack{l+i+j\leq k, \\ \{i,j\}\leq l<k}}  \Lb \phi_i \Lb \phi_j L^2 \phi_l L \phi_k, \\
&\sum_{\substack{l+i+j\leq k, \\ i \leq j,\, l< j\leq k}} L \phi_i L \phi_j \Lb^2 \phi_{l} L \phi_k, & & \sum_{\substack{l+i+j\leq k, \\ l<k}} \Psi^{(i+2)} (u) L \phi_j L \phi_l L \phi_k,\\
\underline{N}_2 :=
&\sum_{\substack{l+i+j\leq k, \\ j \leq i, \, l< i\leq k}}   \Lb \phi_i L \phi_j L \Lb \phi_l L \phi_k,  & & \sum_{\substack{l+i+j\leq k, \\ \{i,j\}\leq l<k}}  L \phi_i L \phi_j \Lb^2 \phi_{l} L \phi_k,\\
&\sum_{\substack{l+i+j\leq k, \\ l< j \leq k}}  \Psi^{(i+1)}(u) \Lb \phi_j L^2 \phi_l L \phi_k, & & \sum_{\substack{l+i+j\leq k, \\ i\leq j, \, l < j\leq k}}  \Lb \phi_i \Lb \phi_j L^2 \phi_l L \phi_k.
}
and
\als{
\underline{N}_3: ={}&  ( |\Psip (u) L^2 \phi| + |\Lb \phi L^2\phi| + |L \phi L \Lb \phi| ) |\Lb \phi_k L \phi_k|, \\
\underline{N}_4 :={}&  (|\Lb^2 \phi L \phi| + |\Lb \phi L \Lb\phi| + |\Psip(u) L \Lb \phi|)  |L\phi_k L \phi_k|.
}

For $\underline{N}_1$ and $\underline{N}_4$, we can always perform $L^\infty$ bounds on the lower order derivatives, that take in order $\Lb \phi_i L \phi_j$, $\Psi^{(i+1)} (u) L \phi_j$, $\Lb \phi_i L \Lb \phi_l$, $\Psi^{(i+1)} (u) L \Lb \phi_l$, $\Psi^{(i+1)} (u) \Lb \phi_j$, $\Lb \phi_i \Lb \phi_j$, $L \phi_i \Lb^2 \phi_l$, $\Psi^{(i+2)} (u) L \phi_j$, $|\Lb^2 \phi L \phi|$, $|\Lb \phi L \Lb\phi|$, $|\Psip(u) L \Lb \phi|$ and are bounded by $M^2 \Lambda^{-\frac{1}{2}}(u)$, and perform $L^2$ estimates on the two higher order derivatives both of which have $L$ derivative, such as $L \phi_j L \phi_k$. Then, it follows that
\als{
 \iint_{D_{t,\ub}^{-}} (|\underline{N}_1|+ |\underline{N}_4| ) \sqrt{g} \di u \di \ub \lesssim & \int_{-\infty}^{+\infty}\frac{ M^2}{ \Lambda^{\frac{1}{2}}(u^{\prime})}\di u^{\prime}\int_{C_{u^{\prime}}^{t}} \sum_{j \leq k}  |L\phi_{j}|^2 \sqrt{g} \di \ub \\
 \lesssim {} &\delta^2 M^4.
}

The structures of $\underline{N}_2$ and $\underline{N}_3$ are analogous to that of $\underline{T}_2$. We apply $L^\infty$ bound to lower order derivatives $L \phi_j L \Lb \phi_l$, $L\phi_i L\phi_j$, $\Psi^{(i+1)} (u) L^2 \phi_l $, $ \Lb \phi_i L^2 \phi_l$, $\Psip (u) L^2 \phi$, $\Lb \phi L^2\phi$, $L \phi L \Lb \phi$ which are bounded by $\delta M^2$, then in the same fashion as $\underline{T}_2$,
\als{
& \iint_{D_{t,\ub}^{-}} (|\underline{N}_2| + |\underline{N}_3| ) \sqrt{g} \di u \di \ub \lesssim
  \sum_{j \leq k} \iint_{D_{t,\ub}^{-}} \delta M^2  |L\phi_j||\Lb\phi_k|\sqrt{g} \di u \di \ub \\
\lesssim & \int_{-\infty}^{\ub}\frac{\delta^2 M^2}{\Lambdab (\ub^{\prime})} \di \ub^{\prime}\int_{\Cb^{t}_{\ub^{\prime}}}\Lambda(u) |\Lb\phi_k|^2 \sqrt{g} \di u
+\int_{-\infty}^{+\infty}\frac{ M^2}{\Lambda(u^{\prime})} \di u^{\prime}\int_{C^{t}_{u^{\prime}}}\Lambdab(\ub)|L\phi_j|^2\sqrt{g} \di \ub\\
\lesssim {} &\delta^2 M^4.
}

As a consequence, there is,
$$\iint_{D_{t,\ub}^{-}}  |\Box_{g(\p\phi)}\phi_k \Lambda(u) g^{\ub \ub} L\phi_k| \sqrt{g} \di u \di \ub \lesssim \delta^2 M^6.$$

In conclusion, we achieve
\be\label{Energy-1}
 \Eb^2_{N+1}(\ub, t)+ \Fb^2_{N+1}(\ub, t)
 \lesssim \Eb^2_{N+1}(\ub, 0)
 + \delta M^4  \lesssim I_{N+1}^2 + \delta M^4,
\ee
and then we take the upper bound over $\ub \in \mathbb{R}$.

Once this energy bound is completed, we can improve the $L^\infty$ estimates,
\bel{improved-L-infty-Lb}
|\Lb \phi_k| \lesssim \frac{I_{N+1} + \delta^{\frac{1}{2}} M^2}{\Lambda^{\frac{1}{2}} (u)}, \quad k \leq N-1.
\ee
Both of the improvements \eqref{Energy-1} and \eqref{improved-L-infty-Lb} will be useful in the next subsection.

\subsubsection{Estimates of $E_{N+1}$ and $F_{N+1}$.}
Letting $\xi = \TL$ and $\psi =\phi_k$ in \eqref{energy-id-2}, we have
\be\label{eq-energy-TL}
E^2_{(k+1)}(u,t) + F^2_{(k+1)}(u,t)  \lesssim E^2_{(k+1)}( u,0) + \iint_{\D^{+}_{ u,t}} |\p_{\alpha}(\sqrt{g}P^{\alpha} [\phi_k, \TL])| \di u \di \ub + \delta^3 M^6,
\ee
 and noting \eqref{eq-div-P}, there is
\begin{align*}
\frac{1}{\sqrt{g}}\p_{\alpha}(\sqrt{g}P^{\alpha}[\phi_k, \TL])= \Box_{g(\p\phi)}\phi_k \cdot \TL \phi_k + \T^{\alpha}_{\beta} [\phi_k] \cdot \p_{\alpha}\TL^{\beta} - \frac{1}{2\sqrt{g}}\TL (\sqrt{g}g^{\gamma \rho}) \p_{\gamma}\phi_k \p_{\rho}\phi_k.
\end{align*}

\emph{The deformation tensor}: \[\T^{\alpha}_{\beta} [\phi_k] \cdot \p_{\alpha}\TL^{\beta} - \frac{1}{2\sqrt{g}}\TL(\sqrt{g}g^{\gamma\rho}) \p_{\gamma}\phi_k \p_{\rho}\phi_k.\]
By Lemma \ref{lem-deform-esti}, this deformation tensor is bounded as
\begin{equation}\label{T-X}
\begin{split}
& |\T^{\alpha}_{\beta} [\phi_k]  \p_\alpha \TL^{\beta}| + \frac{1}{2\sqrt{g}} |\TL(\sqrt{g}g^{\gamma\rho}) \p_{\gamma}\phi_k \p_{\rho}\phi_k| \\
 \lesssim {} & \Lambdab (\ub) [ (\Psip (u))^2 + (\Lb \phi)^2 + (\Lb^2 \phi)^2 + (\Lb L \phi)^2 + |\Psi^{\prime \prime} (u) L\phi| +  |\Lb^2 \phi L \phi| ] (L \phi_k)^2 \\
 & + \Lambdab (\ub) [ (L \phi)^2+ (L^2 \phi)^2 +  ( L\Lb\phi)^2 ]  (\Lb \phi_k)^2.
 \end{split}
 \end{equation}

For the last line of \eqref{T-X}, we should use the improved energy estimate \eqref{Energy-1},
\begin{align}\label{deform-L-1}
& \iint_{\D^{+}_{t, u}}  |\Lambdab (\ub)| |L \phi|^2 |\Lb \phi_k|^2\sqrt{g}\di u \di \ub = \iint_{D^+_{t,u}} \Lambdab (\ub) |L \phi|^2 |\Lb \phi_k|^2\sqrt{g} \di \tau \di x\nnb \\
\lesssim &{} \sup_\tau \int_{\Sigma^{+}_{\tau, u}} \Lambda (u) |\Lb \phi_k|^2 \int_0^t \Big\|\frac{\Lambdab^{\frac{1}{2}}(\ub)}{\Lambda^{\frac{1}{2}}(u)} L \phi\Big\|^2_{L^\infty (\Sigma^{+}_{\tau, u})} \di \tau \nnb \\
\lesssim &{} \sup_\tau \int_{\Sigma_{\tau}} \Lambda (u) |\Lb \phi_k|^2 \int_0^t \sum_{ i \leq 1}  \Big\|\frac{\Lambdab^{\frac{1}{2}}(\ub)}{\Lambda^{\frac{1}{2}}(u)} L \phi_i \Big\|^2_{L^2 (\Sigma^{+}_{\tau, u})} \di \tau \nnb \\
\lesssim &{} \sup_\tau \Eb^2_{k+1}(\tau) \iint_{\D^{+}_{t, u}} \frac{\Lambdab(\ub)}{\Lambda(u)} (|L \phi|^2 + |L \phi_1|^2)\sqrt{g} \di \tau \di x \nnb \\
\lesssim &{} (I_{N+1}^2 + \delta M^4) \int^{u}_{-\infty} \frac{1}{\Lambda(u^\prime)} \di u^\prime \int_{C^t_{u^\prime}} \Lambdab (\ub) (|L \phi|^2 + |L \phi_1|^2)\sqrt{g} \di \ub \nnb \\
\lesssim &{}  \int^{u}_{-\infty} \frac{I_{N+1}^2 + \delta M^4}{\Lambda(u^\prime)} F^2_{N+1}(u^\prime, t) \di u^\prime,
\end{align}
and the same approach is valid for $\iint_{\D^{+}_{t, u}} \Lambdab (\ub) \left( |L^2 \phi|^2 + |L \Lb \phi|^2 \right) |\Lb \phi_k|^2 \sqrt{g}\di u \di \ub$ as well.

For the second line of \eqref{T-X},
\begin{align}\label{deform-L-2}
& \iint_{\D^{+}_{t, u}}  \Lambdab (\ub) \left(  (\Psip (u))^2 + |\Lb \phi|^2 +  |\Lb^2 \phi|^2 +  |\Lb L \phi|^2 \right) |L \phi_k|^2   \sqrt{g} \di u \di \ub \nnb \\
\lesssim &{} \int_{-\infty}^{u} \frac{I_{N+1}^2 + \delta M^4 }{\Lambda(u^{\prime}) }du^{\prime} \int_{C_{u^{\prime}}^{t}} \Lambdab (\ub) |L \phi_k|^2\sqrt{g} \di \ub,
\end{align}
where the improved $L^\infty$ estimates \eqref{improved-L-infty-Lb} are used, and
\begin{align}\label{deform-L-3}
& \iint_{\D^{+}_{t, u}} \Lmdub ( |\Psi^{\prime \prime} (u) L \phi| +  |\Lb^2 \phi L \phi| )  |L \phi_k|^2\sqrt{g} \di u \di \ub \nnb \\
 \lesssim &  \iint_{\D^{+}_{t, u}} \frac{\delta M^2}{\deyub\Lambda^{\frac{1}{2}}(u)} \Lambdab (\ub) |L \phi_k|^2 \sqrt{g}\di u \di \ub \nnb  \\
\lesssim & \int^{u}_{-\infty} \frac{\delta M^2}{\Lambda^{\frac{1}{2}}(u^{\prime})} \di u^\prime \int_{C^t_{u^\prime}} \Lambdab (\ub) |L \phi_k|^2 \lesssim \delta^3M^4.
\end{align}

\emph{The source term}: $$\Box_{g(\p\phi)}\phi_k \cdot \Lmdub \TL \phi_k = \Box_{g(\p\phi)}\phi_k \cdot \Lmdub \left( (1 + \text{l.o.t.} ) L \phi_k - g^{u u} \Lb \phi_k \right).$$

As before, we start with $\Box_{g(\p\phi)}\phi_k \Lambdab(\ub) L \phi_k$, which consists of the following types,
\als{
T_1 :=
&\sum_{\substack{l+i+j\leq k, \\ \{i,j\}\leq l<k}}  \Lambdab(\ub) \Lb \phi_i L \phi_j L  \Lb \phi_l L \phi_k, & & \sum_{\substack{l+i+j\leq k, \\ i \leq j, \, l< j\leq k}}  \Lambdab(\ub) \Lb \phi_i L \phi_j L \Lb \phi_l L \phi_k, \\ &\sum_{\substack{l+i+j\leq k, \\ l <k}} \Lambdab(\ub)  \Psi^{(i+1)}(u) L \phi_j L \Lb \phi_l L \phi_k, & & \sum_{\substack{l+i+j\leq k, \\ i \leq j, \, l< j\leq k}}  \Lambdab(\ub) L \phi_i L \phi_j \Lb^2 \phi_{l} L \phi_k,
}
and
\als{
T_2 := &\sum_{\substack{l+i+j\leq k, \\ j \leq l<k}}\Lambdab(\ub)  \Psi^{(i+1)}(u) \Lb \phi_j L^2 \phi_l L \phi_k, & & \sum_{\substack{l+i+j\leq k, \\ \{i,j\}\leq l<k}} \Lambdab(\ub) \Lb \phi_i \Lb \phi_j L^2 \phi_l L \phi_k, \\
T_3 :=
&\sum_{\substack{l+i+j\leq k, \\ j \leq i, \, l< i\leq k}}  \Lambdab(\ub) \Lb \phi_i L \phi_j  L \Lb \phi_l L \phi_k,  & & \sum_{\substack{l+i+j\leq k, \\ \{i,j\}\leq l<k}} \Lambdab(\ub) L \phi_i L \phi_j \Lb^2 \phi_{l} L \phi_k,\\
T_4 :=&\sum_{\substack{l+i+j\leq k, \\ l < j\leq k}}\Lambdab(\ub)  \Psi^{(i+1)}(u) \Lb \phi_j L^2 \phi_l L \phi_k, & & \sum_{\substack{l+i+j\leq k, \\ i \leq j, \, l< j\leq k}} \Lambdab(\ub) \Lb \phi_i \Lb \phi_j L^2 \phi_l L \phi_k,
}
and
\als{
T_5 :={}& \Psi^{(i+2)} (u) \Lambdab(\ub) L \phi_j L\phi_l L \phi_k, \quad i+j+l\leq k,\,l<k, \\
T_6 :={}& ( |\Psip (u) L^2 \phi| + |\Lb \phi L^2 \phi| + |L \phi \Lb L \phi| ) \Lambdab(\ub) |L\phi_k \Lb \phi_k|, \\
T_7 :={}& (|\Lb^2 \phi L \phi| + |\Lb \phi L \Lb\phi| + |\Psip(u) L \Lb \phi|) \Lambdab(\ub) |L\phi_k|^2.
}

For $T_1$, we perform $L^\infty$ bounds on the lower order derivatives, which take the forms of $\Lb \phi_i L \phi_j, \, \Lb \phi_i L \Lb \phi_l$, $\Psi^{(i+1)} L \phi_j$ (or $\Psi^{(i+1)} L \Lb \phi_l$), $L \phi_i \Lb^2 \phi_l$ in order and afford the smallness $\delta$ and the factor $\Lambda^{-\frac{1}{2}}(u)$, and $L^2$ estimates on the two higher order derivatives both of which contain $L$ derivative, such as $L \phi_j L \phi_k$. Consequently, there is
\als{
 \iint_{D_{t,u}^{+}} |T_1| \sqrt{g} \di u \di \ub
\lesssim&\int_{-\infty}^{u}\frac{\delta M + \delta M^2}{\Lambda^{\frac{1}{2}}(u^{\prime})} \di u^{\prime}\int_{C_{u^{\prime}}^{t}} \sum_{i \leq k} \Lambdab(\ub) |L\phi_i|^2 \sqrt{g} \di \ub\\
\lesssim & \delta^3 M^4.
}

For $T_2$, we should make use of the updated $L^\infty$ estimates \eqref{improved-L-infty-Lb} to arrive at the improvement
\als{
\|\Psi^{(i+1)} \Lb \phi_j\|_{L^\infty} \lesssim {}& (I_{N+1}+\delta^{\frac{1}{2}} M^2) \Lambda^{-1}(u), \quad j \leq N-1, \\
\|\Lb \phi_i \Lb \phi_j\|_{L^\infty} \lesssim{}& (I_{N+1}^2+\delta M^4) \Lambda^{-1}(u), \quad i, j \leq N-1.
}
As a result,
\begin{align}\label{esti-T2}
 \iint_{D_{t,u}^{+}} |T_2| \sqrt{g} \di u \di \ub
\lesssim &\int_{-\infty}^{u} \left( \frac{I_{N+1}^2+\delta M^4}{\Lambda(u^{\prime})} + \frac{I_{N+1}+\delta^{\frac{1}{2}} M^2}{\Lambda(u^{\prime})} \right) \di u^{\prime}\int_{C_{u^{\prime}}^{t}} \sum_{i \leq k} \Lambdab(\ub) |L\phi_i|^2  \sqrt{g}\di \ub \nnb \\
\lesssim&\int_{-\infty}^{u}\frac{ I_{N+1}^2+\delta^{\frac{1}{2}} M^2}{\Lambda(u^{\prime})} F^2_{k+1} (u^{\prime},t) \di u^{\prime}.
\end{align}

For $T_3$ that is similar to $\underline{N}_2$, we apply $L^\infty$ bounds to the lower order derivatives $$\|L \phi_j L \Lb \phi_l\|_{L^\infty} \lesssim \delta^2 M^2 \Lambdab^{-1}(\ub), \|L\phi_i L\phi_j\|_{L^\infty} \lesssim \delta^2 M^2 \Lambdab^{-1}(\ub).$$ Then we obtain
\als{
 \iint_{D_{t,u}^{+}} |T_3| \sqrt{g} \di u \di \ub
\lesssim & \iint_{D_{t}} \sum_{j \leq k }\delta^2 M^2|\Lb\phi_j||L\phi_k|\sqrt{g} \di u \di \ub\\
\lesssim & \int_{-\infty}^{+ \infty}\frac{\delta^3M^2}{\Lambdab(\ub^{\prime})} \di \ub^{\prime}\int_{\Cb_{\ub^{\prime}}^{t}}\Lambda(u)|\Lb\phi_j|^2\sqrt{g} \di u
\\
& + \int_{-\infty}^{+\infty}\frac{\delta M^2}{\Lambda(u^{\prime})} \di u^{\prime}\int_{C_{u^{\prime}}^{t}}\Lambdab(\ub)|L\phi_k|^2\sqrt{g} \di \ub \\
\lesssim {}& \delta^3 M^4.
}

For $T_4$, by the Cauchy-Schwarz inequality, we have
\begin{align}\label{esti-T4}
 & \iint_{D_{t,u}^{+}} |T_4| \sqrt{g} \di u \di \ub \lesssim \iint_{D_{t,u}^{+}} \sum_{l \leq j \leq k} |L^2\phi_l|^2|\Lb\phi_j|^2\Lambdab(\ub)\sqrt{g}  \di u \di \ub \nnb \\
& \qquad + \iint_{D_{t,u}^{+}}  |\Psi^{(i+1)}(u)|^2 |L\phi_k|^2\Lambdab(\ub)\sqrt{g}  \di u \di \ub + \iint_{D_{t,u}^{+}} \sum_{i\leq k/2}| \Lb\phi_i |^2 |L\phi_k|^2\Lambdab(\ub)\sqrt{g}  \di u \di \ub.
\end{align}
The first line on the right hand side of \eqref{esti-T4} can be bounded in the same way as \eqref{deform-L-1}, while for the second line, we refer to \eqref{esti-T2}. Therefore,
\als{
  \iint_{D_{t,u}^{+}} |T_4| \sqrt{g} \di u \di \ub \lesssim & \int_{-\infty}^{u}\frac{ I_{N+1}^2+\delta M^4 }{\Lambda(u^{\prime})} F^2_{k+1}(u^{\prime},t) \di u^{\prime}.
}

For $T_5$,
\als{
&\sum_{\substack{i+j+k\leq k,\\l<k}} \iint_{D_{t,u}^{+}} |\Psi^{(i+2)}(u)L\phi_j L\phi_l L\phi_k| \Lambdab(\ub)\sqrt{g}  \di u \di \ub\\
\lesssim &\iint_{D_{t,u}^{+}}\frac{\delta M}{\Lambda^{\frac{1}{2}}(u) \langle u\rangle^{\frac{1+\epsilon}{2}} \Lambdab^{\frac{1}{2}}(\ub)}\Lambdab(\ub) (|L\phi_k|^2+|L\phi_{\max\{j,l\}}|^2)\sqrt{g}  \di u \di \ub \\
\lesssim {} &\delta^3 M^3.
}

Finally,
\als{
T_6
\lesssim{}&   (|L \phi|^2 + |L^2 \phi|^2 ) \Lambdab(\ub) |\Lb \phi_k|^2 + (|\Lb \phi|^2 + |\Psip (u)|^2 + |\Lb L \phi|^2 ) \Lambdab(\ub) |L\phi_k|^2,
}
has been treated in \eqref{deform-L-1} and \eqref{deform-L-2}.
And $T_7$ can be estimated in an analogous way as \eqref{deform-L-3}.

The lower order term $|\Box_{g(\p\phi)}\phi_k \cdot \Lmdub g^{u u}\Lb \phi_k| \lesssim \delta^2M^2 |\Box_{g(\p\phi)}\phi_k \cdot  \Lb \phi_k|$, constitutes
\als{
N_1 :=
& \sum_{\substack{l+i+j\leq k, \\ \{i,j\}\leq l<k}}   \Lb \phi_i L \phi_j \Lb L \phi_l \Lb \phi_k, && \sum_{\substack{l+i+j\leq k, \\ j\leq l<k}}   \Psi^{(i+1)}(u) L \phi_j \Lb L \phi_l \Lb \phi_k, \\
&\sum_{\substack{l+i+j\leq k, \\ j \leq i, \, l< i\leq k}}   \Lb \phi_i L \phi_j L \Lb \phi_l \Lb \phi_k,  & & \sum_{\substack{l+i+j\leq k, \\ \{i,j\}\leq l<k}}  L \phi_i L \phi_j \Lb^2 \phi_{l} \Lb \phi_k,\\
&\sum_{\substack{l+i+j\leq k, \\ l< j\leq k}}  \Psi^{(i+1)}(u) \Lb \phi_j L^2 \phi_l \Lb \phi_k, & & \sum_{\substack{l+i+j\leq k, \\ i\leq j, \, l < j\leq k}}  \Lb \phi_i \Lb \phi_j L^2 \phi_l \Lb \phi_k, \\
N_2 :=& \sum_{\substack{l+i+j\leq k, \\ \{i,l\}\leq j\leq k}}   \Lb \phi_i L \phi_j L \Lb \phi_l \Lb \phi_k, & & \sum_{\substack{l+i+j\leq k, \\ l <j \leq k}}   \Psi^{(i+1)}(u) L \phi_j L \Lb \phi_l \Lb \phi_k,\\
&\sum_{\substack{l+i+j\leq k, \\ j \leq l<k}}  \Psi^{(i+1)}(u) \Lb \phi_j L^2 \phi_l \Lb \phi_k, & & \sum_{\substack{l+i+j\leq k, \\ \{i,j\}\leq l<k}}  \Lb \phi_i \Lb \phi_j L^2 \phi_l \Lb \phi_k, \\
&\sum_{\substack{l+i+j\leq k, \\ i \leq j,\, l< j\leq k}} L \phi_i L \phi_j \Lb^2 \phi_{l} \Lb \phi_k, & & \sum_{\substack{l+i+j\leq k, \\ l<k}} \Psi^{(i+2)} (u) L \phi_j L \phi_l \Lb \phi_k,\\
}
and
\als{
N_3: ={}&  ( |\Psip (u) L^2 \phi| + |\Lb \phi L^2\phi| + |L \phi L \Lb \phi| ) |\Lb \phi_k \Lb \phi_k|, \\
N_4 :={}&  (|\Lb^2 \phi L \phi| + |\Lb \phi L \Lb\phi| + |\Psip(u) L \Lb \phi|)  |L\phi_k \Lb \phi_k|.
}
Note that, there is no $\Lmdub$ factor, and the estimates are relatively easier.

For $N_1$ and $N_3$, we  perform $L^\infty$ bounds on the lower order derivatives, which take in order the forms of $\Lb \phi_i L \phi_j, \, \Psi^{(i+1)} (u) L \phi_j, \, L \phi_j L \Lb \phi_l, \, L \phi_i L \phi_j, \, \Psi^{(i+1)} (u) L^2 \phi_l, \, \Lb \phi_i  L^2 \phi_l$, $\Psip (u) L^2 \phi$, $\Lb \phi L^2\phi$, $L \phi L \Lb \phi$
  and are bounded by $\delta M^2 \Lambdab^{-\frac{1}{2}}(\ub)$, and perform $L^2$ estimates on the two higher order derivatives both of which present $\Lb$ derivative, such as $\Lb \phi_j \Lb \phi_k$. Then, we arrive at
\als{
 \iint_{D_{t,u}^{+}} (|N_1|+ |N_4| ) \sqrt{g} \di u \di \ub \lesssim & \int_{-\infty}^{+\infty}\frac{ \delta M^2}{ \Lambdab^{\frac{1}{2}}(\ub^{\prime})}\di \ub^{\prime}\int_{\Cb_{\ub^{\prime}}^{t}} \sum_{j \leq k}  |\Lb\phi_{j}|^2 \sqrt{g} \di \ub \\
 \lesssim {} &\delta M^4.
}

The structures of $N_2$ and $N_4$ are analogous to that of $\underline{T}_2$. We apply $L^\infty$ bounds to the lower order derivatives $\Lb \phi_i L \Lb \phi_l$, $\Psi^{(i+1)} (u) L \Lb \phi_l$, $\Psi^{(i+1)} (u) \Lb \phi_j$,  $\Lb \phi_i  \Lb \phi_j$, $L\phi_i \Lb^2 \phi_l$, $\Psi^{(i+2)} (u) L \phi_j$ (or $\Psi^{(i+2)} (u) L \phi_l$), $|\Lb^2 \phi L \phi|$, $|\Lb \phi L \Lb\phi|$, $|\Psip(u) L \Lb \phi|$, and they are bounded by $M^2$, then
\als{
& \iint_{D_{t,u}^{+}} (|N_2| + |N_4| ) \sqrt{g} \di u \di \ub \lesssim
  \sum_{j \leq k} \iint_{D_{t,u}^{+}} M^2  |L\phi_j||\Lb\phi_k|\sqrt{g} \di u \di \ub \\
\lesssim {} & \delta M^4, \quad \text{by}\, \eqref{esti-Tb2}.
}

As a consequence, there is,
$$\iint_{D_{t,u}^{+}} |\Box_{g(\p\phi)}\phi_k \cdot \Lmdub g^{u u} \Lb \phi_k|  \sqrt{g} \di u \di \ub \lesssim \delta^3 M^6.$$

Putting all these estimates together and summing up over $k \leq N$, we achieve
\begin{equation}\label{Energy-2}
\begin{split}
& E^2_{N+1}(u, t) + F^2_{N+1}(u, t) \\
 \lesssim{}& \delta^2 I^2_{N+1} + \delta^3 M^6 + \int^{u}_{-\infty} \frac{I_{N+1}^2 + \delta^{\frac{1}{2}} M^2}{\Lambda^{\frac{1}{2}}(u^\prime)} F^2_{N+1} (u^\prime, t) \di u^\prime,
 \end{split}
 \end{equation}
which further implies
\bes
F^2_{N+1}(u, t)  \lesssim \delta^2 I^2_{N+1} + \delta^3 M^6 + \int^{u}_{-\infty} \frac{I_{N+1}^2 + \delta^{\frac{1}{2}} M^2}{\Lambda^{\frac{1}{2}}(u^\prime)} F^2_{N+1} (u^\prime, t) \di u^\prime.
\ees
By the Gronwall's inequality, we have
\bes
F^2_{N+1}(u, t)  \lesssim \exp(I_{N+1}^2 + \delta^{\frac{1}{2}} M^2) (\delta^2 I^2_{N+1} + \delta^3 M^6).
\ees
Substituting the above bound into \eqref{Energy-2} yields
\be\label{Energy-final-2}
 E^2_{N+1}(u, t) + F^2_{N+1}(u, t)  \lesssim  \delta^2 I^2_{N+1} + \delta^3 M^6.
\ee

In the end, combining \eqref{Energy-1} and \eqref{Energy-final-2},  and taking the upper bound over $u \in \mathbb{R}, \, \ub \in \mathbb{R}$, we conclude Proposition \ref{pro} and close the energy argument.

\subsection{Non-small data}\label{sec-data}
As in \cite{Wang-Wei1}, we can verify that
\begin{lemma}\label{thm-data}
Suppose $\phi$ satisfies the equivalent Euler-Lagrangian equation \eqref{eq-expanding} with data described in the main theorem \ref{main-theorem1}, then $\phi$ has the initial energy \[ E^2_{N+1}(0) \lesssim \delta^2 I^2_{N+1}, \quad \Eb^2_{N+1} (0) \lesssim I^2_{N+1}. \]
\end{lemma}
\begin{proof}
For any bounded function $f \in C^\infty(\mathbb R)$, we call $f(x) \in O_\gamma$, if
\bes
\int_{\mathbb R}(1+|x|)^{2+2\gamma}|f^{(k)} (x)|^2 \di x \lesssim 1, \quad \text{for all integer} \,\, k \geq 0.
\ees
We also use $O_\gamma$ to denote a function in $O_\gamma$.

By the data assumption \eqref{data-ass-1}, and taking $$\p_{\ub} = \p_t + \p_x, \quad \p_u =  \frac{(\Psip (u))^2 }{2} (\p_t + \p_x ) +   \frac{1 }{2} (\p_t - \p_x ),$$ into accounts, we know that
\be\label{data-ass-1-1}
\p_{\ub} \phi |_{t=0} = \delta f(x), \quad  \p_{u} \phi |_{t=0}= \fb(x),
\ee
where $f(x), \, \fb (x) \in O_\gamma$.   This proves Lemma \ref{thm-data} with $N=0$.

  For the general case, we prove by induction. Suppose that there exist smooth functions $f_p(x), \, \fb_p(x) \in O_\gamma$, such that for $p \leq k$, there holds
  \be\label{data-induction}
  \p_{\ub}\phi_p|_{t=0}=\delta f_p(x),\quad \p_{u}\phi_p|_{t=0}=\fb_p(x).
  \ee
  Then we have
  \be\label{x}
  \p_x (\p_{\ub}\phi_{k}|_{t=0}) =\delta f^\prime_{k}(x),\quad \p_x (\p_{u}\phi_{k}|_{t=0} ) =\fb_{k}^\prime (x).
  \ee
  That is
  \begin{align}
\frac{(\Psip (u))^2 +1 }{2} \p_{\ub} \p_u \phi_{k}|_{t=0} - \p_u^2 \phi_{k}|_{t=0} ={}& O_\gamma, \label{x-u-data-k} \\
 \frac{(\Psip (u))^2 +1 }{2} \p_{\ub}^2 \phi_{k}|_{t=0} - \p_u \p_{\ub} \phi_{k}|_{t=0} ={}& \delta O_\gamma,  \label{x-ub-data-k}
\end{align}
noting that \[\p_x =  \frac{(\Psip (u))^2 +1 }{2} \p_{\ub} - \p_u. \]  Consider the high order equation \eqref{commute} for $\phi_k$, which is expanded as follows,
\be\label{H-equ}
2 g^{u\ub}\p_{\ub} \p_{u} \phi_{k}=-g^{uu}\p^2_{u}\phi_{k}-g^{\ub\ub}\p^2_{\ub}\phi_{k}+\p^{k}(R_0 (\p \phi))-
\sum_{ i+j\leq k,\, j<k }\p^{i}g^{\mu\nu}\p_{\mu}\p_{\nu}\phi_{j}.
\ee
Since $\p^{k}(R_0 (\p \phi))-
\sum_{ i+j\leq k,\, j<k }\p^{i}g^{\mu\nu}\p_{\mu}\p_{\nu}\phi_{j}$ contains the lower order derivatives and  satisfies the ``null condition'',  it has at least one factor $\p_{\ub}\phi_q$, with $q\leq k$, and lies in $\delta O_\gamma$ by induction. As a result, we  obtain from \eqref{H-equ} \[2g^{u\ub}\p_{u}\p_{\ub}\phi_{k}=-g^{uu}\p^2_{u}\phi_{k}-g^{\ub\ub}\p^2_{\ub}\phi_{k} + \delta O_\gamma. \] Substituting \eqref{x-u-data-k}-\eqref{x-ub-data-k} into the above equation, and in view of \eqref{data-ass-1-1},
\[
g^{u\ub}|_{t=0} + 1 \in O_\gamma, \quad g^{uu}|_{t=0} \in \delta^2 O_\gamma,\quad g^{\ub\ub}|_{t=0} \in O_\gamma,
\]
we derive
\be
\p_{\ub} \p_{u}\phi_{k}|_{t=0} \in\delta O_\gamma.
\ee
It then follows from \eqref{x-u-data-k} and \eqref{x-ub-data-k} that
\bes
\p^2_u\phi_{k}|_{t=0} \in O_\gamma,\quad \p^2_{\ub}\phi_{k}|_{t=0} \in \delta O_\gamma.
\ees
We conclude that \eqref{data-induction} holds true for $p = k+1$ and complete the proof.
\end{proof}

We end up with showing that the perturbation along the $\p_u$ direction is indeed non-small, if initially it is.
\begin{lemma}\label{lem-inital}
With the bootstrap assumptions \eqref{bt-Eb-Fb} and \eqref{bt-E-F}, we have \[ (\p_u \phi)^2 > (\psi (u))^2, \quad \text{in} \,\,\ub \in [\ub(0, a),\ub (0, b)], \] if initially there holds \[  (\p_u \phi)^2 (t, x) |_{t=0} >   (\psi (-x))^2, \quad x \in [a, b]. \] Here $\psi(x)$ is an arbitrarily smooth function, $\ub(0, a)$ ($\ub(0, b)$) denotes the value of $\ub$ when $t=0$ and $x=a$ ($x=b$), and $a, b \in \mathbb{R}$ are any finite numbers.
\end{lemma}
\begin{proof}
We rewrite the equivalent Euler-Lagrangian equation \eqref{eq-expanding} as
\als{
 & 2 \p_{\ub} \p_{u} \phi =  \bar S, \quad \text{where} \\
& \quad \bar S =  2 \Psip (u) \p_{\ub} \phi  \p_{\ub} \p_{u} \phi  - 2 g^{-1} \cg^2  \p^{u}_{\cg} \phi \p^{\ub}_{\cg} \phi \p_{\ub} \p_{u} \phi  - g^{-1} ( \cg  \p^{u}_{\cg} \phi)^2 \p^2_{u}\phi \\
& \quad \quad -2 \Psip (u) \p_{\ub} \phi  \p^2_{\ub}\phi  - g^{-1} (\cg \p^{\ub}_{\cg} \phi)^2 \p^2_{\ub}\phi - \cg R_0 (\p \phi),
 }
and observe that $\bar S$ has at least one factor of $\p_{\ub}$ (or $\p_{\ub}^2$) derivative and one factor with $\Lambda^{-\frac{1}{2}}(u)$ decay, such as $\Psip (u)$ or $\p_u\phi$.
Then energy bounds and hence the $L^\infty$ bounds tell that
 \als{
|\bar S| \lesssim {}& \frac{\delta }{\deyu \deyub}.
}
That is, we have \[|\p_{\ub} (\p_u \phi)^2| = |2 \bar S \cdot \p_u \phi| \lesssim \frac{\delta }{\Lmdu \deyub},\] and then
\[ |\p_{\ub} \left( (\p_u \phi)^2 -   (\psi (u))^2 \right) | = |2  \bar S \p_u \phi| \lesssim \frac{\delta }{ \deyub}, \]
which yields, after integrating along the curves of $\p_{\ub}$,
\als{
\big| \left( (\p_u \phi)^2 - (\psi(u))^2 \right)  - \left( (\p_u \phi)^2 - (\psi(u))^2 \right)|_{t=0} \big| \lesssim \delta.
}
Since $\delta$ is small enough, we complete the proof.
\end{proof}

\section{Conclusion}
The study of the relativistic string (membrane) equations is a hot research topic in mathematical physics. The highlight of our research in this paper lies in that we find a class of non-small perturbations of a non-trivial background solution (the travelling wave solutions) to the relativistic string equations still leads to global existence. Moreover, our large perturbations are allowed to spread in the whole space, not just concentrate in a thin null strip as what happens for the short pulse method. However, we only consider the strings in terms of graphic embedding in Minkowski spacetime here, there are a wealth of open problems that worth studying in our future research work, such as what follows.

(1) Due to its hyperbolicity, we do not expect the globally smooth solution for general smooth initial data. If the solution blows up in finite time, what is the mechanism for the formation of singularity and what type of singularities will appear?

(2) If we consider the relativistic string (membrane) in terms of graph or submanifolds  embedding with co-dimension larger than one in other Lorenzian spacetime such as Schwarzschild and Kerr ones, how can we obtain the globally smooth solutions under small initial data or some special large data?

\noindent{\Large {\bf Acknowledgements.}}  The authors are grateful to Pin Yu for helpful discussions. J. W. is supported by NSFC (Grant Nos. 12271450, 11871408). C. W. is supported by the NSFC (Grant Nos. 12071435, 11871212), Zhejiang Provincial Outstanding Youth Science Foundation (Grant. No. LR22A010004), and the Fundamental Research Fund of Zhejiang Sci-Tech University.

\appendix

\section{The equivalent Euler-Lagrangian equation}\label{appendix-A}
\subsection{Continued proofs of Proposition \ref{prop-EL-ini} and Proposition \ref{prop-EL}}\label{Appendix-A1}

We will complete the proof of Proposition \ref{prop-EL-ini} by calculating the detailed expression of $S_0 (\p^2 \phi, \p \phi)$.

\begin{proof}[Continued proof of Proposition \ref{prop-EL-ini}]
By the formula of $S_0 (\p^2 \phi, \p \phi)$ \eqref{E-L2} and in view of the fact $\cg=(1-\Psi^\prime(u) \p_{\ub} \phi)^2$,
 \begin{align}
 S_0 ={}& \frac{\cg}{\sqrt{g}} \p_{\ub} ((\sqrt{g})^{-1} \Psip (u) ) - \frac{\cg}{\sqrt{g}}  \p_{\ub} ((\sqrt{g})^{-1} (\Psip (u))^2 \p_{\ub} \phi)  \nnb \\
& - \frac{1}{\sqrt{g}}\p_u \left( ( - \cg +1 - \Psip (u) \p_{\ub} \phi ) \p_{\ub} \phi (\sqrt{g})^{-1} \right) \nnb \\
& - \frac{1}{\sqrt{g}}\p_{\ub} \left( (- \cg +1 + \Psip (u) \p_{\ub} \phi )  \p_{u} \phi (\sqrt{g})^{-1} \right) \nnb  \\
 &- \frac{1}{\sqrt{g}}\p_u (\cg -1) \left( \p_{\ub} \phi  (\sqrt{g})^{-1} \right) - \frac{1}{\sqrt{g}}\p_{\ub} (\cg -1) \left( \p_{u} \phi  (\sqrt{g})^{-1} \right). \label{S0-1}
\end{align}
Note that,
 \begin{align*}
& - \frac{1}{\sqrt{g}}\p_u \left( ( - \cg +1 - \Psip (u) \p_{\ub} \phi ) \p_{\ub} \phi (\sqrt{g})^{-1} \right)  \\
& - \frac{1}{\sqrt{g}}\p_{\ub} \left( (- \cg +1 + \Psip (u) \p_{\ub} \phi )  \p_{u} \phi (\sqrt{g})^{-1} \right)   \\
 = {} &  \frac{1}{\sqrt{g}}\p_u (\cg -1) \left( \p_{\ub} \phi  (\sqrt{g})^{-1} \right) + \frac{1}{\sqrt{g}}\p_{\ub} (\cg -1) \left( \p_{u} \phi  (\sqrt{g})^{-1} \right) \\
 & +  \frac{(\cg -1)}{\sqrt{g}} \p_u \left( \p_{\ub} \phi  (\sqrt{g})^{-1} \right) + \frac{(\cg -1)}{\sqrt{g}}   \p_{\ub}\left( \p_{u} \phi  (\sqrt{g})^{-1} \right) \\
 &+ \frac{1}{\sqrt{g}}\p_u \left(  \Psip (u) \p_{\ub} \phi \p_{\ub} \phi (\sqrt{g})^{-1} \right) - \frac{1}{\sqrt{g}}\p_{\ub} \left(  \Psip (u) \p_{\ub} \phi  \p_{u} \phi (\sqrt{g})^{-1} \right).
\end{align*}
Then we have
 \begin{align*}
 S_0 ={}& \frac{\cg}{\sqrt{g}} \p_{\ub} ((\sqrt{g})^{-1} \Psip (u) ) - \frac{\cg}{\sqrt{g}}  \p_{\ub} ((\sqrt{g})^{-1} (\Psip (u))^2 \p_{\ub} \phi)   \\
& +  \frac{(\cg -1)}{\sqrt{g}} \p_u \left( \p_{\ub} \phi  (\sqrt{g})^{-1} \right) + \frac{(\cg -1)}{\sqrt{g}}   \p_{\ub}\left( \p_{u} \phi  (\sqrt{g})^{-1} \right) \\
 &+ \frac{1}{\sqrt{g}}\p_u \left(  \Psip (u) \p_{\ub} \phi \p_{\ub} \phi (\sqrt{g})^{-1} \right) - \frac{1}{\sqrt{g}}\p_{\ub} \left(  \Psip (u) \p_{\ub} \phi  \p_{u} \phi (\sqrt{g})^{-1} \right),
\end{align*}
where in the first line
 \begin{align*}
 &  \frac{\cg}{\sqrt{g}}  \Psip (u) \p_{\ub} ((\sqrt{g})^{-1}) - \frac{\cg}{\sqrt{g}}  \p_{\ub} ((\sqrt{g})^{-1} (\Psip (u))^2 \p_{\ub} \phi)   \\
= {} & -  \frac{\cg}{2 \sqrt{g}}  \Psip (u) g^{-\frac{3}{2}} \p_{\ub} g - \frac{\cg}{g} (\Psip (u))^2 \p^2_{\ub} \phi - \frac{\cg}{\sqrt{g}} (\Psip (u))^2 \p_{\ub} \phi  \p_{\ub} (\sqrt{g})^{-1},
\end{align*}
and the last two terms
 \begin{align*}
 & \frac{1}{\sqrt{g}}\p_u \left(  \Psip (u) \p_{\ub} \phi \p_{\ub} \phi (\sqrt{g})^{-1} \right) - \frac{1}{\sqrt{g}}\p_{\ub} \left(  \Psip (u) \p_{\ub} \phi  \p_{u} \phi (\sqrt{g})^{-1} \right) \\
  ={}&  \frac{1}{\sqrt{g}} \p_u \left(  \Psip (u) \p_{\ub} \phi  (\sqrt{g})^{-1} \right) \p_{\ub} \phi - \frac{1}{\sqrt{g}}\p_{\ub} \left(  \Psip (u) \p_{\ub} \phi  (\sqrt{g})^{-1} \right)  \p_{u} \phi.
\end{align*}
Then $ S_0$ can be further simplified as
 \begin{align*}
 S_0 ={}& -  \frac{\cg}{2 g^2}  \Psip (u)  \p_{\ub} g - \frac{\cg}{g} (\Psip (u))^2 \p^2_{\ub} \phi  \\
 &- \frac{\cg}{\sqrt{g}} (\Psip (u))^2 \p_{\ub} \phi  \p_{\ub} (\sqrt{g})^{-1} +  \frac{(\cg -1)}{\sqrt{g}} 2 \p_u  \p_{\ub} \phi (\sqrt{g})^{-1} \\
&    +  \frac{(\cg -1)}{\sqrt{g}}  \p_{\ub} \phi  \p_u \left( (\sqrt{g})^{-1} \right) + \frac{(\cg -1)}{\sqrt{g}} \p_{u} \phi  \p_{\ub} \left(  (\sqrt{g})^{-1} \right) \\
 &+  \frac{1}{g} \p_u \left(  \Psip (u) \p_{\ub} \phi \right) \p_{\ub} \phi +  \frac{1}{\sqrt{g}} \left(  \Psip (u) \p_{\ub} \phi  \right) \p_{\ub} \phi \p_u (\sqrt{g})^{-1}  \\
 & - \frac{1}{g} \p_{\ub} \left(  \Psip (u) \p_{\ub} \phi \right)  \p_{u} \phi -  \frac{1}{\sqrt{g}}  \left(  \Psip (u) \p_{\ub} \phi   \right)  \p_{u} \phi \p_{\ub} (\sqrt{g})^{-1}.
\end{align*}
Here we observe a crucial cancellation in the first line,
 \begin{align*}
& -  \frac{\cg}{2 g^2}  \Psip (u)  \p_{\ub} g - \frac{\cg}{g} (\Psip (u))^2 \p^2_{\ub} \phi \\
={}& -  \frac{\cg}{2 g^2}  \Psip (u)  \left(-2 \Psip (u) \p^2_{\ub} \phi + 2 (\Psip (u))^2 \p_{\ub} \phi \p_{\ub}^2 \phi - 2 \p_{\ub} ( \p_u \phi \p_{\ub} \phi) \right) - \frac{\cg}{g} (\Psip (u))^2 \p^2_{\ub} \phi \\
={}& \left( \frac{\cg}{g^2} - \frac{\cg}{g} \right) (\Psip (u))^2 \p^2_{\ub} \phi - \frac{\cg}{g^2}  (\Psip (u))^3 \p_{\ub} \phi \p_{\ub}^2 \phi + \frac{\cg}{g^2}  \Psip (u)  \p_{\ub} \p_u \phi \p_{\ub} \phi  + \frac{\cg}{g^2}  \Psip (u)  \p_u \phi \p^2_{\ub} \phi,
\end{align*}
and moreover, the last term above $+ \frac{\cg}{g^2}  \Psip (u)  \p_u \phi \p^2_{\ub} \phi$ also partially cancels with the term $- \frac{1}{g} \p_{\ub} \left(  \Psip (u) \p_{\ub} \phi \right)  \p_{u} \phi$ (in the fifth line of the preceding formula of $S_0$).
 It turns out there are no bad terms such as $(\Psip (u))^2 \p^2_{\ub} \phi$, $ \Psip (u)  \p_u \phi \p^2_{\ub} \phi$,
 \begin{align}
 S_0 ={}& \frac{\cg}{g^2}  \left( 1-g \right) (\Psip (u))^2 \p^2_{\ub} \phi + \frac{1}{g^2}  \left( \cg -g \right)  \Psip (u)  \p_u \phi \p^2_{\ub} \phi \nonumber\\
 &- \frac{\cg}{g^2}  (\Psip (u))^3 \p_{\ub} \phi \p_{\ub}^2 \phi + \frac{\cg}{g^2}  \Psip (u) \p_{\ub} \phi \p_{\ub} \p_u \phi  \nonumber\\
&- \frac{\cg}{\sqrt{g}} (\Psip (u))^2 \p_{\ub} \phi  \p_{\ub} (\sqrt{g})^{-1} +  \frac{(\cg -1)}{\sqrt{g}} 2 \p_u  \p_{\ub} \phi (\sqrt{g})^{-1}\nonumber \\
&    +  \frac{(\cg -1)}{\sqrt{g}}  \p_{\ub} \phi  \p_u \left( (\sqrt{g})^{-1} \right) + \frac{(\cg -1)}{\sqrt{g}} \p_{u} \phi  \p_{\ub} \left(  (\sqrt{g})^{-1} \right) \nonumber\\
 &+  \frac{1}{g} \left(  \Psi^{\prime \prime} (u) \p_{\ub} \phi \right) \p_{\ub} \phi +  \frac{1}{g}  \Psip (u)  \p_{\ub} \phi \p_u  \p_{\ub} \phi   \nonumber\\
 &+  \frac{1}{\sqrt{g}} \left(  \Psip (u) \p_{\ub} \phi  \right) \p_{\ub} \phi \p_u (\sqrt{g})^{-1} -  \frac{1}{\sqrt{g}}  \left(  \Psip (u) \p_{\ub} \phi   \right)  \p_{u} \phi \p_{\ub} (\sqrt{g})^{-1}.\label{S0}
 \end{align}

Then the formula \eqref{m-e} follows.
\end{proof}

The proof of Proposition \ref{prop-EL-ini} involves plenty of technical computations, which eventually helps to achieve the equivalently simplified Euler-Lagrangian equation in Proposition \ref{prop-EL}.

\begin{proof}[Proof of Proposition \ref{prop-EL}]
We start with the expansion
\als{
&\p_{\mu} \left( \cg^{\mu\nu}\p_{\nu} \phi \cg (\sqrt{g})^{-1}\right)
=  g^{\mu\nu}\p_{\mu}\p_{\nu}\phi \cg (\sqrt{g})^{-1} + f_1+f_2+f_3, \\
&\qquad f_1=   \cg^2 g^{-1} \p_{\cg}^{\mu}\phi\p_{\cg}^{\nu}\phi\p_{\mu}\p_{\nu}\phi (\sqrt{g})^{-1}, \\
&\qquad f_2= \cg^{\ub\ub} \cg \p_{\ub}\phi\p_{\ub}(\sqrt{g})^{-1}
+\cg^{u\ub} \cg \p_{u}(\sqrt{g})^{-1}\p_{\ub}\phi
+\cg^{\ub u} \cg \p_{\ub}(\sqrt{g})^{-1}\p_{u}\phi, \\
& \qquad f_3= \p_{\ub} (\cg^{\ub \ub} \cg) \p_{\ub} \phi(\sqrt{g})^{-1} + \p_{\ub} (\cg^{\ub u} \cg) \p_{u}\phi (\sqrt{g})^{-1} +\p_{u} (\cg^{u \ub} \cg) \p_{\ub}\phi(\sqrt{g})^{-1}.
}
Note that  $\cg \p^u_{\cg} \phi = - \p_{\ub} \phi + \Psip (u) (\p_{\ub} \phi)^2$, $\cg \p^{\ub}_{\cg} \phi =  - \p_u \phi - \Psip (u) \p_u \phi \p_{\ub} \phi$, and
\als{
& \cg^2 \p_{\cg}^{\mu}\phi\p_{\cg}^{\nu}\phi\p_{\mu}\p_{\nu}\phi = \cg^2 (\p_{\cg}^{u}\phi)^2 \p^2_u\phi + \cg^2 (\p_{\cg}^{\ub}\phi)^2 \p^2_{\ub}\phi + 2 \cg^2 \p_{\cg}^{u}\phi \p^{\ub}_{\cg} \phi \p_u \p_{\ub} \phi,
}
and $\cg^{u\ub} \cg = -1+ \Psip (u) \p_{\ub} \phi$, $\cg^{\ub \ub} \cg = -2 \Psip (u) \p_{u} \phi$,  $\partial(\sqrt{g})^{-1}=-\frac{1}{2}g^{-\frac{3}{2}}\partial g$, then
\als{
 f_1&+f_2= g^{-\frac{3}{2}}  (- \p_{\ub} \phi + \Psip (u) (\p_{\ub} \phi)^2 )^2 \p^2_u\phi  + g^{-\frac{3}{2}}   (- \p_u \phi - \Psip (u) \p_{\ub} \phi \p_u \phi )^2 \p^2_{\ub} \phi  \\
&+ 2 g^{-\frac{3}{2}}  (- \p_{\ub} \phi + \Psip (u) (\p_{\ub} \phi)^2 ) (- \p_u \phi - \Psip(u) \p_u \phi \p_{\ub} \phi) \p_u \p_{\ub} \phi \\
& +  g^{-\frac{3}{2}} \Psip (u) \p_{u} \phi \p_{\ub}\phi \p_{\ub} g
-\frac{1}{2} g^{-\frac{3}{2}} (\Psip (u) \p_{\ub} \phi -1)  \p_{\ub}\phi \p_{u}g
-\frac{1}{2} g^{-\frac{3}{2}} (\Psip (u) \p_{\ub} \phi -1)  \p_{u}\phi \p_{\ub}g, \\
f_3 ={}& - 2\p_{\ub} (\Psip (u) \p_u \phi) \p_{\ub}\phi(\sqrt{g})^{-1} + \p_{\ub} (\Psip (u) \p_{\ub} \phi ) \p_{u} \phi(\sqrt{g})^{-1} + \p_{u} (\Psip (u) \p_{\ub} \phi ) \p_{\ub}\phi (\sqrt{g})^{-1}.
}
We further expand
\als{
-\frac{1}{2} \partial g=&-\frac{1}{2} \partial(-2\p_{u}\phi\p_{\ub}\phi-2\Psip(u) \p_{\ub}\phi + (\Psip(u) \p_{\ub}\phi)^2 )\\
=&  \partial(\p_{u} \phi \p_{\ub}\phi+\Psip(u)\p_{\ub}\phi) -  \Psip(u)\p_{\ub}\phi \p (\Psip(u)\p_{\ub}\phi).
}
Then we can see a main cancellation (in the first line below) that
\als{
f_1 & +f_2+f_3=  g^{-\frac{3}{2}} (1- 1)  ( (\p_{\ub} \phi)^2   \p^2_u \phi + (\p_{u} \phi)^2   \p^2_{\ub} \phi + 2 \p_{\ub} \phi \p_u \phi   \p_u \p_{\ub} \phi ) \\
&+ g^{-\frac{3}{2}}  (- 2\p_{\ub} \phi \Psip (u) (\p_{\ub} \phi)^2 + (\Psip (u))^2 (\p_{\ub} \phi)^4 ) \p^2_u\phi  \\
&+ g^{-\frac{3}{2}}  (2 \p_u \phi  \Psip (u) \p_u \phi \p_{\ub} \phi  + (\Psip (u) \p_{\ub} \phi \p_u \phi)^2)  \p^2_{\ub} \phi  \\
&- 2 g^{-\frac{3}{2}}  \Psip (u) (\p_{\ub} \phi)^2 ( \p_u \phi + \Psip (u) \p_u \phi \p_{\ub} \phi) \p_u \p_{\ub} \phi + 2 g^{-\frac{3}{2}}   \p_{\ub} \phi  \Psip (u) \p_{\ub} \phi \p_u \phi  \p_u \p_{\ub} \phi \\
& - 2  g^{-\frac{3}{2}} \Psip (u) \p_{u} \phi \p_{\ub}\phi  \left( \p_{\ub}(\p_{u} \phi \p_{\ub}\phi+\Psip(u)\p_{\ub}\phi) -  \Psip(u)\p_{\ub}\phi  \p_{\ub} (\Psip(u)\p_{\ub}\phi) \right) \\
&+ g^{-\frac{3}{2}} (\Psip (u) \p_{\ub} \phi )  \p_{\ub}\phi \left( \p_{u}(\p_{u} \phi \p_{\ub}\phi+\Psip(u)\p_{\ub}\phi) -  \Psip(u)\p_{\ub}\phi  \p_{u} (\Psip(u)\p_{\ub}\phi) \right) \\
&+g^{-\frac{3}{2}} (\Psip (u) \p_{\ub} \phi )  \p_{u}\phi  \left( \p_{\ub}(\p_{u} \phi \p_{\ub}\phi+\Psip(u)\p_{\ub}\phi) -  \Psip(u)\p_{\ub}\phi  \p_{\ub} (\Psip(u)\p_{\ub}\phi) \right) \\
& - g^{-\frac{3}{2}}   \p_{\ub}\phi \left( \p_u(\Psip(u)\p_{\ub} \phi) -  \Psip(u)\p_{\ub}\phi \p_u (\Psip(u)\p_{\ub}\phi) \right) \\
& - g^{-\frac{3}{2}}  \p_{u}\phi \left( \p_{\ub} (\Psip(u)\p_{\ub}\phi) -  \Psip(u)\p_{\ub}\phi \p_{\ub} (\Psip(u)\p_{\ub}\phi) \right) \\
&- 2\p_{\ub} (\Psip (u) \p_u \phi) \p_{\ub}\phi(\sqrt{g})^{-1} + \p_{\ub} (\Psip (u) \p_{\ub} \phi ) \p_{u} \phi(\sqrt{g})^{-1} + \p_{u} (\Psip (u) \p_{\ub} \phi ) \p_{\ub}\phi (\sqrt{g})^{-1}.
}
And also note the important cancellations in the last three lines above, therefore
\als{
 f_1&+f_2+f_3
=  g^{-\frac{3}{2}}  (- 2 \Psip (u) (\p_{\ub} \phi)^3 + (\Psip (u))^2 (\p_{\ub} \phi)^4 ) \p^2_u\phi  \\
&+ g^{-\frac{3}{2}}  [2 \Psip (u) (\p_u \phi)^2 \p_{\ub} \phi  + (\Psip (u) \p_{\ub} \phi \p_u \phi)^2]  \p^2_{\ub} \phi - 2 g^{-\frac{3}{2}}  (\Psip (u))^2  \p_u \phi  (\p_{\ub} \phi)^3 \p_u \p_{\ub} \phi \\
& - g^{-\frac{3}{2}} \Psip (u) \p_{u} \phi \p_{\ub}\phi  \left( \p_{\ub}(\p_{u} \phi \p_{\ub}\phi  ) -  \Psip(u)\p_{\ub}\phi  \p_{\ub} (\Psip(u)\p_{\ub}\phi) \right) \\
&+ g^{-\frac{3}{2}}  \Psip(u) (\p_{\ub}\phi)^2 \left( \p_{u}(\p_{u} \phi \p_{\ub}\phi+ 2 \Psip(u)\p_{\ub}\phi) -  \Psip(u)\p_{\ub}\phi  \p_{u} (\Psip(u)\p_{\ub}\phi) \right) \\
&- 2 \Psip (u) \p_{\ub} \p_u \phi \p_{\ub}\phi (\sqrt{g})^{-1} + g^{-\frac{3}{2}} (g-1) [\p_{\ub} \phi \p_u (\Psip(u)\p_{\ub}\phi) +  \Psip (u) \p_u \phi \p_{\ub}^2 \phi].
}
A further calculation leads to
\als{
 f_1&+f_2+f_3
 =  g^{-\frac{3}{2}}  \Psip (u) (\p_{\ub} \phi)^3 \left( \Psip (u) \p_{\ub} \phi -1 \right) \p^2_u \phi \\
 &- 2 g^{-\frac{3}{2}}  \Psip (u) \p_u \phi (\p_{\ub}\phi)^2 (  \Psip (u) \p_{\ub} \phi + 1) \p_u \p_{\ub}\phi  \\
 & - g^{-\frac{3}{2}}  \Psip (u)  \p_u \phi \p_{\ub} \phi \left(  \p_u \phi + 2 \Psip (u)  \right)  \p^2_{\ub} \phi \\
 & + g^{-\frac{3}{2}} ( \Psip (u))^2 \p_{u} \phi (\p_{\ub}\phi)^2 \left(  \p_{u} \phi  + 2 \Psip (u) \right)  \p^2_{\ub} \phi \\
 &  - 2  \Psip (u)  \p_{\ub}\phi \p_{\ub} \p_u \phi (\sqrt{g})^{-1}  -2  g^{-\frac{3}{2}} \Psi^{\prime \prime} \p_u \phi (\p_{\ub} \phi)^3.
}
Then we obtain
\als{
g^{\mu\nu} & (\sqrt{g})^{-1} \cg \p_{\mu}\p_{\nu}\phi+ g^{-\frac{3}{2}}  \Psip (u) (\p_{\ub} \phi)^3 \left( \Psip (u) \p_{\ub} \phi -1 \right) \p^2_u \phi \\
& - 2 g^{-\frac{3}{2}}  \Psip (u) \p_u \phi (\p_{\ub}\phi)^2 (  \Psip (u) \p_{\ub} \phi + 1) \p_u \p_{\ub}\phi  \\
 & - g^{-\frac{3}{2}}  \Psip (u)  \p_u \phi \p_{\ub} \phi \left(  \p_u \phi + 2 \Psip (u)  \right)  \p^2_{\ub} \phi \\
 & + g^{-\frac{3}{2}} ( \Psip (u))^2 \p_{u} \phi (\p_{\ub}\phi)^2 \left(  \p_{u} \phi  + 2 \Psip (u) \right)  \p^2_{\ub} \phi \\
 & - 2  \Psip (u)  \p_{\ub}\phi \p_{\ub} \p_u \phi (\sqrt{g})^{-1}  -2  g^{-\frac{3}{2}} \Psi^{\prime \prime} (u) \p_u \phi (\p_{\ub} \phi)^3 =  \sqrt{g} S_0 (\p^2 \phi, \p \phi),
}
where, with further expansions on $S_0 (\p^2 \phi, \p \phi)$ \eqref{S0}, we derive
 \begin{align*}
  \sqrt{g}  & S_0 (\p^2 \phi, \p \phi) =  \sqrt{g} \left( \frac{\cg}{g^2} -  \frac{3}{g} \right)   \Psip (u)    \p_{\ub} \phi  \p_{\ub} \p_u \phi \\
  &+  \sqrt{g} \left(\frac{\cg}{g^2} - \frac{3}{g^2} \right)  ( \Psip (u) )^2 \p_u \phi \p_{\ub} \phi  \p^2_{\ub} \phi - \frac{\sqrt{g}}{g^2}   \Psip (u)  ( \p_u \phi)^2 \p_{\ub} \phi  \p^2_{\ub} \phi    \\
 &+ \sqrt{g} \left(  \frac{2}{g} - \frac{\cg}{ g^2} \right) ( \Psip (u) )^2 (\p_{\ub} \phi)^2 \p_u \p_{\ub} \phi -  \Psip (u)   (\p_{\ub} \phi)^2  \p_u \left( (\sqrt{g})^{-1} \right)\\
& +  ( \Psip (u)  )^2 (\p_{\ub} \phi)^3  \p_u \left( (\sqrt{g})^{-1} \right) + ( \Psip (u)  )^2 (\p_{\ub} \phi)^2 \p_u \phi  \p_{\ub} \left( (\sqrt{g})^{-1} \right) \\
&+ 3 g^{-\frac{3}{2}} ( \Psip (u)  )^3 (\p_{\ub} \phi)^2 \p_u \phi \p_{\ub}^2 \phi - 3 g^{-\frac{3}{2}}  \Psip (u) (u) \p_{u} \phi (\p_{\ub} \phi)^2 \p_{\ub} \p_u \phi + g^{-\frac{1}{2}}\Psi^{\prime \prime} (u) (\p_{\ub} \phi)^2.
\end{align*}
Notice the cancellations in terms involving the leading terms $(\Psip (u))^2 \p_u \phi \p_{\ub} \phi  \p^2_{\ub} \phi$, $\Psip (u)  \p_{\ub} \phi  \p_{\ub} \p_u \phi$ and $g^{-\frac{3}{2}} \Psip (u)  ( \p_u \phi)^2 \p_{\ub} \phi  \p^2_{\ub} \phi$, then we get
\als{
&g^{\mu\nu}(\sqrt{g})^{-1} \cg \p_{\mu}\p_{\nu}\phi \\
={}& g^{-\frac{3}{2}} \left(\cg -1 \right) (\Psip (u))^2 \p_u \phi \p_{\ub} \phi  \p^2_{\ub} \phi  + g^{-\frac{3}{2}} \left(\cg -g \right)  \Psip (u)  \p_{\ub} \phi  \p_{\ub} \p_u \phi  \\
&- g^{-\frac{3}{2}} \Psip (u) (\p_{\ub} \phi)^3 \left( \Psip (u) \p_{\ub} \phi -1 \right) \p^2_u \phi \\
& + 2 g^{-\frac{3}{2}} \Psip(u) \p_u \phi (\p_{\ub}\phi)^2 \left( \Psip (u) \p_{\ub} \phi + 1 \right) \p_u \p_{\ub}\phi  \\
 &- g^{-\frac{3}{2}}  (\Psip (u))^2 \p_{u} \phi (\p_{\ub}\phi)^2 \left( \p_{u} \phi + 2 \Psip (u) \right)  \p^2_{\ub} \phi
\\
&+ g^{-\frac{3}{2}} \left(2g - \cg \right) (\Psip (u))^2 (\p_{\ub} \phi)^2 \p_u \p_{\ub} \phi - \Psip (u) (\p_{\ub} \phi)^2  \p_u \left( (\sqrt{g})^{-1} \right)\\
& +  (\Psip (u))^2 (\p_{\ub} \phi)^3  \p_u \left( (\sqrt{g})^{-1} \right) + (\Psip (u))^2 (\p_{\ub} \phi)^2 \p_u \phi  \p_{\ub} \left( (\sqrt{g})^{-1} \right) \\
&+ 3 g^{-\frac{3}{2}} (\Psip (u))^3 (\p_{\ub} \phi)^2 \p_u \phi \p_{\ub}^2 \phi - 3 g^{-\frac{3}{2}} \Psip (u) \p_{u} \phi (\p_{\ub} \phi)^2 \p_{\ub} \p_u \phi \\
& + 2  g^{-\frac{3}{2}} \Psi^{\prime \prime} (u) \p_u \phi (\p_{\ub} \phi)^3 + g^{-\frac{1}{2}}\Psi^{\prime \prime} (u) (\p_{\ub} \phi)^2.
}
We further do some expansions to achieve
\als{
 g^{\mu\nu} &(\sqrt{g})^{-1} \cg \p_{\mu}\p_{\nu}\phi = - g^{-\frac{3}{2}} \Psip (u) (\p_{\ub} \phi)^3 \left( \Psip (u) \p_{\ub} \phi -1 \right) \p^2_u \phi \\
 &+ g^{-\frac{3}{2}} \Psip(u) \p_u \phi (\p_{\ub}\phi)^2 (2 \Psip (u) \p_{\ub} \phi +1) \p_u \p_{\ub}\phi  \\
 &- g^{-\frac{3}{2}}  (\Psip (u))^2 \p_{u} \phi  (\p_{\ub}\phi)^2 \left(  \p_{u} \phi  + \Psip (u)   - (\Psip (u))^2 \p_{\ub}\phi \right)  \p^2_{\ub} \phi   \\
 &+ g^{-\frac{3}{2}} \left( 2g - \cg  \right) (\Psip (u))^2 (\p_{\ub} \phi)^2 \p_u \p_{\ub} \phi
\\
&- g^{-\frac{3}{2}} \Psip (u) (\p_{\ub} \phi)^2 \left( \p_u (\p_{u} \phi \p_{\ub}\phi+\Psip(u)\p_{\ub}\phi) -  \Psip(u)\p_{\ub}\phi \p_u (\Psip(u)\p_{\ub}\phi \right)  \\
& + g^{-\frac{3}{2}}  (\Psip (u))^2 (\p_{\ub} \phi)^3 \left( \p_u (\p_{u} \phi \p_{\ub}\phi+\Psip(u)\p_{\ub}\phi) -  \Psip(u)\p_{\ub}\phi \p_u (\Psip(u)\p_{\ub}\phi) \right) \\
&+  g^{-\frac{3}{2}} (\Psip (u))^2 (\p_{\ub} \phi)^2 \p_u \phi  \left(  \p_{\ub} (\p_{u} \phi \p_{\ub}\phi+\Psip(u)\p_{\ub}\phi) -  \Psip(u)\p_{\ub}\phi \p_{\ub} (\Psip(u)\p_{\ub}\phi) \right) \\
& + 2  g^{-\frac{3}{2}} \Psi^{\prime \prime}  (u) \p_u \phi (\p_{\ub} \phi)^3 + g^{-\frac{1}{2}}\Psi^{\prime \prime} (u) (\p_{\ub} \phi)^2 \\
={} 
& - g^{-\frac{3}{2}} \Psip (u)  \Psi^{\prime \prime} (u) (\p_{\ub} \phi)^3 + 2 g^{-\frac{3}{2}}  (\Psip (u))^2 \Psi^{\prime \prime}(u)  (\p_{\ub} \phi)^4  \\
& - g^{-\frac{3}{2}}  (\Psip (u))^3  \Psi^{\prime \prime}(u)  (\p_{\ub} \phi)^5 + 2  g^{-\frac{3}{2}} \Psi^{\prime \prime} (u) \p_u \phi (\p_{\ub} \phi)^3 + g^{-\frac{1}{2}}\Psi^{\prime \prime} (u) (\p_{\ub} \phi)^2.
}
We conclude this proposition.
\end{proof}

\subsection{An alternative formula of the  Euler-Lagrangian equation}\label{Appendix-A2}
In this subsection, we will process the calculations in a slightly different way, which leads to an equivalent but much more complicated formula (containing more quasilinear terms) of the  Euler-Lagrangian equation.

The counterpart of Propositions \ref{prop-EL-ini} is provided as follows.
\begin{proposition}\label{prop-EL-ini1}
The Euler-Lagrangian equation \eqref{eq-EL-0} leads to
\begin{align}
\Box_{g(\p\phi)} \phi = {}& \tilde{S}_0 (\p^2 \phi, \p \phi), \nnb \\
  \tilde{S}_0  (\p^2 \phi, \p \phi) ={}& \frac{1}{g^2}  \left( 1-g \right) (\Psip (u))^2 \p^2_{\ub} \phi + \frac{1}{g^2}  \left( 1 -g \right)  \Psip (u)  \p_u \phi \p^2_{\ub} \phi \nonumber\\
 &- \frac{1}{g^2}  (\Psip (u))^3 \p_{\ub} \phi \p_{\ub}^2 \phi + \frac{1}{g^2}  \Psip (u)   \p_{\ub} \phi  \p_{\ub} \p_u \phi +  \frac{1}{g}  \Psip (u)  \p_{\ub} \phi \p_u  \p_{\ub} \phi  \nonumber\\
 &- \frac{1}{\sqrt{g}} (\Psip (u))^2 \p_{\ub} \phi  \p_{\ub} (\sqrt{g})^{-1} +  \frac{1}{g} \Psi^{\prime \prime} (u) (\p_{\ub} \phi )^2   \nonumber \\
 &+  \frac{1}{\sqrt{g}}  \Psip (u) (\p_{\ub} \phi)^2 \p_u (\sqrt{g})^{-1} -  \frac{1}{\sqrt{g}}  \Psip (u) \p_{\ub} \phi   \p_{u} \phi \p_{\ub} (\sqrt{g})^{-1}. \label{m-e1}
\end{align}
\end{proposition}
\begin{remark}\label{rem-minus}
The two variants of the Euler-Lagrangian equations in the propositions \ref{prop-EL-ini} and  \ref{prop-EL-ini1} are equivalent.
In fact, we deduce from \eqref{m-e} and \eqref{m-e1} that
\als{
&\quad \tilde{S}_0 (\p^2 \phi, \p \phi) - S_0(\p^2 \phi, \p \phi) \\
&=\frac{1-\cg}{g^2}(\Psip(u))^2\p_{\ub}^2\phi+\frac{1}{g^2}(1-\cg)\Psip(u)\p_u\phi\p_{\ub}^2\phi\\
&-\frac{1-\cg}{g^2}(\Psip(u))^3\p_{\ub}\phi\p^2_{\ub}\phi
+\frac{1-\cg}{g^2}\Psip(u)\p_{\ub}\phi\p_{u}\p_{\ub}\phi\\
&-\frac{1-\cg}{\sqrt{g}}(\Psip(u))^2\p_{\ub}\phi\p_{\ub}(\sqrt{g}^{-1})-\frac{1-\cg}{g^2}\p_{\ub}((\Psip(u))^2\p_{\ub}\phi)\\
&+\frac{1-\cg}{\sqrt{g}}\left(2\sqrt{g}^{-1}\p_{u}\p_{\ub}\phi+\p_{\ub}\phi\p_{u}(\sqrt{g}^{-1})+
\p_{u}\phi\p_{\ub}(\sqrt{g})^{-1}\right)\\
&=\frac{1-\cg}{2g^2}\Psip(u)\p_{\ub}\left(-2\p_{u}\phi\p_{\ub}\phi-2\Psip(u)\p_{\ub}\phi+(\Psip(u))^2(\p_{\ub}\phi)^2\right)\\
&-\frac{1-\cg}{\sqrt{g}}\p_{\ub}\left((\Psip(u))^2\p_{\ub}\phi\sqrt{g}^{-1}\right)+\frac{1-\cg}{\sqrt{g}}\left(
\p_{u}(\p_{\ub}\phi\sqrt{g}^{-1})+\p_{\ub}(\p_{u}\phi\sqrt{g}^{-1})\right)\\
&=\frac{\cg-1}{\sqrt{g}}\left(
-\p_{u}(\p_{\ub}\phi\sqrt{g}^{-1})-\p_{\ub}(\p_{u}\phi\sqrt{g}^{-1})
+\p_{\ub}\left((\Psip(u))^2\p_{\ub}\phi\sqrt{g}^{-1}\right)-\p_{\ub}(\Psip(u)\sqrt{g}^{-1})\right)\\
&=0,
}
where in the last identity, we have used the Euler-Lagrangian equation \eqref{eq-EL-0}.
\end{remark}

\begin{proof}[Proof of Proposition \ref{prop-EL-ini1}]
In this proof, we use the facts
$\cg^{u\ub} \cg = -1+ \Psip (u) \p_{\ub} \phi$, $\cg^{\ub \ub} \cg = -2 \Psip (u) \p_{u} \phi$ to replace the formulation of \eqref{E-C}\footnote{This manipulation makes $\tilde S_0 (\p^2 \phi, \p \phi)$ \eqref{E-L2-1}  different from $S_0 (\p^2 \phi, \p \phi)$ \eqref{S0-1}.},
\begin{align*}
\p_\mu (\cg^{\mu \nu} \p_\nu \phi \cg (\sqrt{g})^{-1} ) = {}& \p_u \left(\cg^{u u} \p_u \phi \cg (\sqrt{g})^{-1} + \cg^{u \ub} \p_{\ub} \phi \cg (\sqrt{g})^{-1} \right) \nnb \\
&+ \p_{\ub} \left( \cg^{\ub u} \p_u \phi \cg (\sqrt{g})^{-1} + \cg^{\ub \ub} \p_{\ub} \phi \cg (\sqrt{g})^{-1} \right) \nnb \\
={}& \p_u \left(-  \p_{\ub} \phi (\sqrt{g})^{-1} \right) + \p_{\ub} \left(-  \p_{u} \phi  (\sqrt{g})^{-1} \right) \nnb \\
& + \p_u \left(  \Psip (u) \p_{\ub} \phi \p_{\ub} \phi (\sqrt{g})^{-1} \right) \nnb \\
& + \p_{\ub} \left( \Psip (u) \p_{\ub} \phi  \p_{u} \phi (\sqrt{g})^{-1} \right) \nnb \\
& -2 \p_{\ub} \left(\Psip (u) \p_{u} \phi \p_{\ub} \phi (\sqrt{g})^{-1} \right).
\end{align*}
Combining this formula with the Euler-Lagrangian equation \eqref{eq-EL-0} and the fact that $$\p_\mu (\cg^{\mu \nu} \p_\nu \phi \cg (\sqrt{g})^{-1} )= \p_\mu ( g^{\mu \nu} \p_\nu \phi \sqrt{g}  ),$$  leads to the following formula
\begin{align}
\Box_g \phi  = {}& \tilde{S}_0 (\p^2 \phi, \p \phi), \nnb \\
 \tilde{S}_0 (\p^2 \phi, \p \phi) ={}& \frac{1}{\sqrt{g}} \Psip (u)  \p_{\ub} (\sqrt{g})^{-1} - \frac{1}{\sqrt{g}}  (\Psip (u))^2  \p_{\ub} ((\sqrt{g})^{-1}  \p_{\ub} \phi)  \nnb \\
& + \frac{1}{\sqrt{g}}\p_u \left(  \Psip (u) (\p_{\ub} \phi)^2 (\sqrt{g})^{-1} \right)  - \frac{1}{\sqrt{g}} \p_{\ub} \left( \Psip (u) \p_{\ub} \phi \p_{u} \phi (\sqrt{g})^{-1} \right).  \label{E-L2-1}
\end{align}

In the formula of $\tilde S_0 (\p^2 \phi, \p \phi)$ \eqref{E-L2-1}, the first two terms
 \begin{align*}
 &   \frac{1}{\sqrt{g}} \Psip (u)  \p_{\ub} (\sqrt{g})^{-1} - \frac{1}{\sqrt{g}}  (\Psip (u))^2  \p_{\ub} ((\sqrt{g})^{-1}  \p_{\ub} \phi)   \\
= {} & -  \frac{1}{2 \sqrt{g}}  \Psip (u) g^{-\frac{3}{2}} \p_{\ub} g - \frac{1}{g} (\Psip (u))^2 \p^2_{\ub} \phi - \frac{1}{\sqrt{g}} (\Psip (u))^2 \p_{\ub} \phi  \p_{\ub} (\sqrt{g})^{-1},
\end{align*}
and the last two terms
 \begin{align*}
 & \frac{1}{\sqrt{g}}\p_u \left(  \Psip (u) \p_{\ub} \phi \p_{\ub} \phi (\sqrt{g})^{-1} \right) - \frac{1}{\sqrt{g}}\p_{\ub} \left(  \Psip (u) \p_{\ub} \phi  \p_{u} \phi (\sqrt{g})^{-1} \right) \\
  ={}&  \frac{1}{\sqrt{g}} \p_u \left(  \Psip (u) \p_{\ub} \phi  (\sqrt{g})^{-1} \right) \p_{\ub} \phi - \frac{1}{\sqrt{g}}\p_{\ub} \left(  \Psip (u) \p_{\ub} \phi  (\sqrt{g})^{-1} \right)  \p_{u} \phi.
\end{align*}
Then $ \tilde{S}_0$ can be further simplified as
 \begin{align*}
 \tilde{S}_0 ={}& -  \frac{1}{2 g^2}  \Psip (u)  \p_{\ub} g - \frac{1}{g} (\Psip (u))^2 \p^2_{\ub} \phi - \frac{1}{\sqrt{g}} (\Psip (u))^2 \p_{\ub} \phi  \p_{\ub} (\sqrt{g})^{-1}   \\
 &+  \frac{1}{g} \p_u \left(  \Psip (u) \p_{\ub} \phi \right) \p_{\ub} \phi +  \frac{1}{\sqrt{g}} \left(  \Psip (u) \p_{\ub} \phi  \right) \p_{\ub} \phi \p_u (\sqrt{g})^{-1}  \\
 & - \frac{1}{g} \p_{\ub} \left(  \Psip (u) \p_{\ub} \phi \right)  \p_{u} \phi -  \frac{1}{\sqrt{g}}  \left(  \Psip (u) \p_{\ub} \phi   \right)  \p_{u} \phi \p_{\ub} (\sqrt{g})^{-1}.
\end{align*}
As before, there is a crucial cancellation in the first line,
 \begin{align*}
& -  \frac{1}{2 g^2}  \Psip (u)  \p_{\ub} g - \frac{1}{g} (\Psip (u))^2 \p^2_{\ub} \phi \\
={}& -  \frac{1}{2 g^2}  \Psip (u)  \left(-2 \Psip (u) \p^2_{\ub} \phi + 2 (\Psip (u))^2 \p_{\ub} \phi \p_{\ub}^2 \phi - 2 \p_{\ub} ( \p_u \phi \p_{\ub} \phi) \right) - \frac{1}{g} (\Psip (u))^2 \p^2_{\ub} \phi \\
={}& \frac{1-g}{g^2}  (\Psip (u))^2 \p^2_{\ub} \phi - \frac{1}{g^2}  (\Psip (u))^3 \p_{\ub} \phi \p_{\ub}^2 \phi + \frac{1}{g^2}  \Psip (u)  \p_{\ub} \p_u \phi \p_{\ub} \phi  + \frac{1}{g^2}  \Psip (u)  \p_u \phi \p^2_{\ub} \phi,
\end{align*}
and moreover, the last term above $+ \frac{1}{g^2}  \Psip (u)  \p_u \phi \p^2_{\ub} \phi$ also partially  cancels with the term $- \frac{1}{g} \p_{\ub} \left(  \Psip (u) \p_{\ub} \phi \right)  \p_{u} \phi$ (in the last line of the preceding formula of $\tilde S_0$).
Consequently, there are no bad terms such as $(\Psip (u))^2 \p^2_{\ub} \phi$, $ \Psip (u)  \p_u \phi \p^2_{\ub} \phi$,
 \begin{align}
 \tilde{S}_0 ={}& \frac{1}{g^2}  \left( 1-g \right) (\Psip (u))^2 \p^2_{\ub} \phi + \frac{1}{g^2}  \left( 1 -g \right)  \Psip (u)  \p_u \phi \p^2_{\ub} \phi \nonumber\\
 &- \frac{1}{g^2}  (\Psip (u))^3 \p_{\ub} \phi \p_{\ub}^2 \phi + \frac{1}{g^2}  \Psip (u)  \p_{\ub} \phi  \p_{\ub} \p_u \phi +  \frac{1}{g}  \Psip (u)  \p_{\ub} \phi \p_u  \p_{\ub} \phi \nonumber\\
 &+  \frac{1}{g} \left(  \Psi^{\prime \prime} (u) \p_{\ub} \phi \right) \p_{\ub} \phi - \frac{1}{\sqrt{g}} (\Psip (u))^2 \p_{\ub} \phi  \p_{\ub} (\sqrt{g})^{-1}  \nonumber \\
 &+  \frac{1}{\sqrt{g}} \left(  \Psip (u) \p_{\ub} \phi  \right) \p_{\ub} \phi \p_u (\sqrt{g})^{-1} -  \frac{1}{\sqrt{g}}  \left(  \Psip (u) \p_{\ub} \phi   \right)  \p_{u} \phi \p_{\ub} (\sqrt{g})^{-1}.  \nonumber
 \end{align}

We finish the proof.
\end{proof}

Although Proposition \ref{prop-EL-ini1} amounts to Proposition \ref{prop-EL-ini}.
The reason that we prefer the equivalent Euler-Lagrangian equation \eqref{m-e} rather than \eqref{m-e1} lies in the followings.
Based on the  \eqref{m-e1}, when following the proof of Proposition \ref{prop-EL}, we derive an alternative formula \eqref{eq-expanding1} of the Euler-Lagrangian equation which involves much more quasilinear terms.
\begin{proposition}
In the $\{u, \ub\}$ coordinate system, the variant of Euler-Lagrangian equation \eqref{m-e1} leads to
\begin{equation}\label{eq-expanding1}
g^{\mu\nu} (\p \phi) \p_{\mu}\p_{\nu}\phi = \tilde R_0 (\p \phi),
\end{equation}
where $\tilde R_0 (\p \phi)$ denotes the semi-linear term
\als{
\tilde R_0 (\p \phi) ={}
& 2 (\cg g)^{-1} \left( 2 \Psip (u)  + \p_u \phi \right)   \p_u \phi \Psip (u)  \p_{\ub} \phi  \p^2_{\ub} \phi  \\
&- (\cg g)^{-1} \left( 2 \Psip (u)  + \p_u \phi \right)   \p_u \phi ( \Psip (u))^2   (\p_{\ub}\phi)^2  \p^2_{\ub} \phi\\
 & + (\cg g)^{-1} \left( 1+3g \right)  \Psip (u)  \p_{\ub} \phi  \p_{\ub} \p_u \phi  + 2 (\cg g)^{-1}  \Psip (u) \p_u \phi (\p_{\ub}\phi)^2 \p_u \p_{\ub}\phi  \\
 & +  (\cg g)^{-1} (2 \p_u \phi  - \Psip (u) )  (\Psip (u))^2 (\p_{\ub}\phi)^3 \p_u \p_{\ub}\phi  \\
 &+ 2 (\cg g)^{-1} \Psip (u) (\p_{\ub} \phi)^3 \p_u^2 \phi - (\cg g)^{-1}  (\Psip (u))^2 (\p_{\ub} \phi)^4   \p^2_u \phi \\
&+ ( \cg g)^{-1} (1- \Psip (u)  \p_{\ub} \phi  ) \Psi^{\prime \prime} (u) (\p_{\ub} \phi)^2.
}
\end{proposition}
\begin{remark}
As what has been shown in Remark \ref{rem-minus}, \eqref{eq-expanding1} is indeed equivalent to \eqref{eq-expanding}. That means, in the $\{u, \ub\}$ coordinate system, \eqref{eq-expanding1} can be further simplified as \eqref{eq-expanding}, if we substitute the Euler-Lagrangian \eqref{eq-EL-0} into \eqref{eq-expanding1} again.

On the other hand, observe that in $\tilde R_0 (\p \phi)$, $\p^2_{\ub} \phi$ always has the factor $ \left( 2 \Psip (u)  + \p_u \phi \right)$ in its coefficients. This fact makes our method work well for the formula \eqref{eq-expanding1} (one should make use of the extra positive term $\int_{\Cb_{\ub}^t}\Lambda(u) (2 \Psip (u) + \Lb \phi)^2  |\Lb \phi|^2|L\phi_k|^2 \sqrt{g} \di u$ in $\Fb^2_{(k+1)} (\ub, t)$). Despite that, the  extra quasilinear terms involve much more work, and we will not bother the readers for that.
\end{remark}

\section{Deformation tensors for multipliers}
Before computing the deformation tensor, let us note two identities,
\begin{align}
T^u_u [\psi] = {} &g^{u \mu}\partial_{\mu}\psi\partial_{u}\psi - \frac{1}{2} g^{\mu\nu}\partial_{\mu}\psi\partial_{\nu}\psi  = \mathcal{A}(\psi),  \\
T^{\ub}_{\ub} [\psi] = {} &g^{\ub \mu}\partial_{\mu}\psi\partial_{\ub}\psi - \frac{1}{2} g^{\mu\nu}\partial_{\mu}\psi\partial_{\nu}\psi  = - \mathcal{A}(\psi),
\end{align}
where
\be\label{def-A}
\mathcal{A}(\psi)  := \frac{1}{2}  \left( g^{u u} (\partial_{u}\psi)^2 - g^{\ub \ub} (\p_{\ub} \psi)^2 \right).
\ee
These two identities will be used repeatedly in the proof of the following lemma without comment.

\begin{lemma}\label{lem-deform}
We have
\als{
 \T^{\alpha}_{\beta} [\psi]  \p_\alpha \TL^{\beta} ={} & - \frac{1}{2} \Lambdab^\prime (\ub)  g^{u \ub} \left( g^{u u} (\p_u \psi)^2 + g^{\ub \ub} (\p_{\ub} \psi)^2 \right)  - \Lambdab^\prime (\ub)  g^{\ub \ub} \p_{\ub} \psi g^{u u} \p_u \psi \\
 & + (\cg g)^{-1} \Lmdub \left( \p_u \left( \cg^2 \p^u_{\cg} \phi \p^u_{\cg} \phi \right) - g \Psip (u) \p_{\ub}^2 \phi - \p_{\ub} \left( \cg^2 \p^{\ub}_{\cg} \phi \p^u_{\cg} \phi \right) \right) \mathcal{A}(\psi)  \\
&+  \Lmdub \cg^{-1}  \left( g^{u u}\partial_{u}\psi\partial_{\ub}\psi + g^{u \ub}\partial_{\ub}\psi\partial_{\ub}\psi \right) \p_{u} ( \Psip (u) \p_{\ub} \phi) \\
& +   \Lmdub (\cg g)^{-1} \left( g^{u \mu}\partial_{\mu}\psi\partial_{\ub}\psi  \right) \p_u \left( \cg^2 \p^{\ub}_{\cg} \phi \p^u_{\cg} \phi \right)  \\
&+  \Lmdub (\cg g)^{-1} \left( g^{\ub \mu}\partial_{\mu}\psi\partial_{u}\psi  \right) \p_{\ub} \left( \cg^2 \p^{u}_{\cg} \phi \p^u_{\cg} \phi \right) \\
& +  \Lmdub \left( \p_u (\cg g)^{-1} \left( \cg^2 \p^u_{\cg} \phi \p^u_{\cg} \phi \right)  - \p_{\ub} (\cg g)^{-1} \left( \cg^2 \p^{\ub}_{\cg} \phi \p^u_{\cg} \phi \right) \right) \mathcal{A}(\psi)  \\
& +   \Lmdub  \left( g^{u \mu}\partial_{\mu}\psi\partial_{\ub}\psi  \right) \p_u(\cg g)^{-1} \left( \cg^2 \p^{\ub}_{\cg} \phi \p^u_{\cg} \phi \right)  \\
&+  \Lmdub \left( g^{\ub \mu}\partial_{\mu}\psi\partial_{u}\psi  \right) \p_{\ub}(\cg g)^{-1} \left( \cg^2 \p^{u}_{\cg} \phi \p^u_{\cg} \phi \right).
}
 And
\als{
\T^{\alpha}_{\beta} [\psi]  \p_\alpha \TLb^\beta = {}& - \frac{1}{2} \Lambda^\prime (u) g^{u \ub} \left( g^{u u} (\p_u \psi)^2 + g^{\ub \ub} (\p_{\ub} \psi)^2 \right)  -  \Lambda^\prime (u)  g^{\ub \ub} \p_{\ub} \psi g^{u u} \p_u \psi \\
& + \cg^{-1}  \Lambda (u)  \left( \p_{u} ( \Psip (u) \p_{\ub} \phi) - 2 \p_{\ub} ( \Psip (u) \p_{u} \phi)  \right) \mathcal{A}(\psi) \\
& +4 \cg^{-2}  \Lambda (u) (\Psip (u))^2 \left( \Psip (u)  \p_u \phi \p_{\ub} \phi \p_{\ub} \p_{u} \phi -  \p_u \phi \p^2_{\ub} \phi) \right) \mathcal{A}(\psi) \\
& + (\cg g)^{-1}  \Lambda (u) \left(  \p_u \left( \cg^2 \p^u_{\cg} \phi \p^{\ub}_{\cg} \phi \right) - \p_{\ub} \left( \cg^2 \p^{\ub}_{\cg} \phi \p^{\ub}_{\cg} \phi \right)  \right) \mathcal{A}(\psi) \\
&+  \Lambda (u) \cg^{-1} \left( g^{\ub u} (\partial_{u}\psi )^2 + g^{\ub \ub}\partial_{\ub}\psi\partial_{u}\psi \right) \Psip (u) \p^2_{\ub} \phi  \\
& + 2 \Lambda (u) \cg^{-1}  (g^{u \mu}\partial_{\mu}\psi\partial_{\ub}\psi ) \p_{u} ( \Psip (u) \p_{u} \phi)  \\
 & - 4 \Lambda (u) \cg^{-2}  (g^{u \mu}\partial_{\mu}\psi\partial_{\ub}\psi ) \p_{u} ( \Psip (u) \p_{u} \phi)  (\Psip (u))^2 \p_u \phi \p_{\ub} \phi    \\
&+ 4  \Lambda (u) \cg^{-2}  (g^{u \mu}\partial_{\mu}\psi\partial_{\ub}\psi )  \p_{u} ( \Psip (u) \p_{\ub} \phi)  \Psip (u) \p_u \phi  \\
& + \Lambda (u) ( \cg g)^{-1} \left( g^{u \mu}\partial_{\mu}\psi\partial_{\ub}\psi  \right) \p_u \left( \cg^2 \p^{\ub}_{\cg} \phi \p^{\ub}_{\cg} \phi \right)  \\
&+ \Lambda (u) ( \cg g)^{-1}  \left( g^{\ub \mu}\partial_{\mu}\psi\partial_{u}\psi  \right) \p_{\ub} \left(\cg^2 \p^{u}_{\cg} \phi \p^{\ub}_{\cg} \phi \right) \\
& +   \Lambda (u) \left(  \p_u (\cg g)^{-1}  \cg^2 \p^u_{\cg} \phi \p^{\ub}_{\cg} \phi  - \p_{\ub} (\cg g)^{-1}  \cg^2 \p^{\ub}_{\cg} \phi \p^{\ub}_{\cg} \phi  \right) \mathcal{A}(\psi) \\
& + \Lambda (u)  \left( g^{u \mu}\partial_{\mu}\psi\partial_{\ub}\psi  \right) \p_u( \cg g)^{-1} \left( \cg^2 \p^{\ub}_{\cg} \phi \p^{\ub}_{\cg} \phi \right)  \\
&+ \Lambda (u) ( \cg g)^{-1}  \left( g^{\ub \mu}\partial_{\mu}\psi\partial_{u}\psi  \right) \p_{\ub}( \cg g)^{-1} \left(\cg^2 \p^{u}_{\cg} \phi \p^{\ub}_{\cg} \phi \right).
}

And
\begin{align*}
& \frac{1}{\sqrt{g}} \TL (\sqrt{g}g^{\gamma\rho}) \p_{\gamma}\phi_k \p_{\rho}\phi_k \\
={}& - \Lmdub \cg^{-1} g^{u \alpha} \p_\alpha (\cg \cg^{\gamma \rho}) \p_\gamma \phi_k \p_\rho \phi_k + \Lmdub  \cg^{-1} g^{u \alpha} \p_\alpha \cg \cg^{\gamma \rho} \p_\gamma \phi_k \p_\rho \phi_k   \\
 & +  \Lmdub   (\cg g)^{-1} g^{u \alpha} \p_\alpha ( \p_{\ub} \phi)^2  (\p_{u} \phi_k)^2 +\Lmdub  (\cg g)^{-1} g^{u \alpha} \p_\alpha ( \p_{u} \phi)^2 (\p_{\ub} \phi_k)^2 \\
 & +  \Lmdub  (\cg g)^{-1} g^{u \alpha} \p_\alpha ( ( \p_{\ub} \phi)^2 (-2 \Psip (u) \p_{\ub} \phi +(\Psip (u) \p_{\ub} \phi)^2) )  (\p_{u} \phi_k)^2 \\
 &+\Lmdub  (\cg g)^{-1} g^{u \alpha} \p_\alpha  ( ( \p_{u} \phi)^2 (2 \Psip (u) \p_{\ub} \phi+( \Psip (u) \p_{\ub} \phi)^2) ) (\p_{\ub} \phi_k)^2 \\
 & + \Lmdub  (\cg g)^{-1} g^{u \alpha} \p_\alpha g \Psip (u) \p_u \phi (\p_{\ub} \phi_k)^2 + 2  \Lmdub  (\cg g)^{-1} g^{u \alpha}\p_\alpha ( \p_u \phi \p_{\ub} \phi)  \Psip (u) \p_{\ub} \phi  \p_{u} \phi_k \p_{\ub} \phi_k \\
 &- 2 \Lmdub  (\cg g)^{-1} g^{u \alpha}\p_\alpha ( \p_u \phi \p_{\ub} \phi (\Psip (u) \p_{\ub} \phi)^2 )  \Psip (u) \p_{\ub} \phi  \p_{u} \phi_k \p_{\ub} \phi_k\\
&- \Lmdub \cg (2g)^{-2}  g^{u \alpha} \p_\alpha g \p^\gamma_{\cg} \phi \p^\rho_{\cg} \phi \p_{\gamma}\phi_k \p_{\rho}\phi_k - \Lmdub g^{-1}  g^{u \alpha} \p_\alpha \cg   \p^\gamma_{\cg} \phi \p^\rho_{\cg} \phi \p_{\gamma}\phi_k \p_{\rho}\phi_k.
\end{align*}

And
\begin{align*}
& \frac{1}{\sqrt{g}} \TLb (\sqrt{g}g^{\gamma\rho}) \p_{\gamma}\phi_k \p_{\rho}\phi_k \\
={}& - \Lmdu \cg^{-1} g^{\ub \alpha} \p_\alpha (\cg \cg^{\gamma \rho}) \p_\gamma \phi_k \p_\rho \phi_k + \Lmdu  \cg^{-1} g^{\ub \alpha} \p_\alpha \cg \cg^{\gamma \rho} \p_\gamma \phi_k \p_\rho \phi_k   \\
 & +  \Lmdu  (\cg g)^{-1} g^{\ub \alpha} \p_\alpha ( \p_{\ub} \phi)^2 (\p_{u} \phi_k)^2 +\Lmdu  (\cg g)^{-1} g^{\ub \alpha} \p_\alpha ( \p_{u} \phi)^2  (\p_{\ub} \phi_k)^2 \\
 & +  \Lmdu  (\cg g)^{-1} g^{\ub \alpha} \p_\alpha ( ( \p_{\ub} \phi)^2 (-2\Psip \p_{\ub} \phi + (\Psip (u) \p_{\ub} \phi)^2) )  (\p_{u} \phi_k)^2 \\
 &+\Lmdu (\cg g)^{-1} g^{\ub \alpha} \p_\alpha  ( ( \p_{u} \phi)^2 (2\Psip \p_{\ub} \phi +( \Psip (u) \p_{\ub} \phi)^2) ) (\p_{\ub} \phi_k)^2 \\
 & + \Lmdu  (\cg g)^{-1} g^{\ub \alpha} \p_\alpha g \Psip (u) \p_u \phi (\p_{\ub} \phi_k)^2 + 2  \Lmdu (\cg g)^{-1} g^{\ub \alpha}\p_\alpha ( \p_u \phi \p_{\ub} \phi)  \Psip (u) \p_{\ub} \phi  \p_{u} \phi_k \p_{\ub} \phi_k \\
 & - 2\Lmdu  (\cg g)^{-1} g^{\ub \alpha}\p_\alpha ( \p_u \phi \p_{\ub} \phi (\Psip (u) \p_{\ub} \phi)^2 )  \Psip (u) \p_{\ub} \phi  \p_{u} \phi_k \p_{\ub} \phi_k\\
&- \Lmdu \cg (2g)^{-2} g^{\ub \alpha} \p_\alpha g \p^\gamma_{\cg} \phi \p^\rho_{\cg} \phi \p_{\gamma}\phi_k \p_{\rho}\phi_k - \Lmdu g^{-1}  g^{\ub \alpha} \p_\alpha \cg   \p^\gamma_{\cg} \phi \p^\rho_{\cg} \phi \p_{\gamma}\phi_k \p_{\rho}\phi_k.
\end{align*}

\end{lemma}

\begin{proof}[Proof of Lemma \ref{lem-deform}]
This lemma follows by straightforward calculations.

For $\TL$, there holds that
\begin{align*}
\T^{\alpha}_{\beta} [\psi] \cdot \p_{\alpha} \TL^{\beta} &= \T^{\alpha}_{\beta} [\psi] \cdot \left( \p_{\alpha} \Lmdub (- Du)^{\beta} + \Lmdub (- \p_\alpha (Du)^\beta)  \right) \\
& = -\Lambdab^\prime (\ub)  g^{u \beta} \T^{\ub}_{\beta} [\psi]  - \Lmdub  \p_\alpha g^{u \beta} \T^{\alpha}_{\beta} [\psi].
\end{align*}
We know that,
\als{
g^{u \beta} \T^{\ub}_{\beta} [\psi] = {}& g^{\ub \mu} \p_\mu \psi g^{u \nu} \p_\nu \psi - \frac{1}{2} g^{u \ub} g^{\mu \nu} \p_\mu \psi \p_\nu \psi \\
={}& \left( g^{\ub u} \p_u \psi  + g^{\ub \ub} \p_{\ub} \psi \right)  \left( g^{u u} \p_u \psi + g^{u \ub} \p_{\ub} \psi  \right) \\
& - \frac{1}{2} g^{u \ub} \left( 2 g^{u \ub} \p_u \psi \p_{\ub} \psi + g^{u u} (\p_u \psi)^2 + g^{\ub \ub} (\p_{\ub} \psi)^2 \right) \\
={}& g^{\ub \ub} \p_{\ub} \psi g^{u u} \p_u \psi + \frac{1}{2} g^{u \ub} \left( g^{u u} (\p_u \psi)^2 + g^{\ub \ub} (\p_{\ub} \psi)^2 \right).
}
On the other hand, $\cg^{u\ub} \cg = -1+ \Psip (u) \p_{\ub} \phi$, $\cg^{\ub \ub} \cg = -2 \Psip (u) \p_{u} \phi$,
\als{
\T^{\alpha}_{\beta} [\psi]  \p_\alpha g^{u \beta} ={}& \T^{\alpha}_{\ub} [\psi] \p_\alpha \cg^{u \ub} - \T^{\alpha}_{\beta} [\psi] \p_\alpha \left( \cg g^{-1} \p^\beta_{\cg} \phi \p^u_{\cg} \phi \right) \\
={}& \T^{\alpha}_{\ub} [\psi] \left( \p_\alpha ( \cg \cg^{u \ub}) \cg^{-1} + \cg \cg^{u \ub}  \p_\alpha \cg^{-1} \right) \\
&- \T^{\alpha}_{\beta} [\psi] \p_\alpha \left( \cg^2 \p^\beta_{\cg} \phi \p^u_{\cg} \phi \right)  (\cg g)^{-1}  - \T^{\alpha}_{\beta} [\psi]  \left( \cg^2 \p^\beta_{\cg} \phi \p^u_{\cg} \phi \right) \p_\alpha (\cg g)^{-1}.
}
Note that,
\be\label{simplify-1}
\p_\alpha ( \cg \cg^{u \ub}) \cg^{-1} + \cg \cg^{u \ub}  \p_\alpha \cg^{-1}  = - \cg^{-1}  \p_\alpha ( \Psip (u) \p_{\ub} \phi),
\ee
and hence
\als{
& \T^{\alpha}_{\ub} [\psi] \left( \p_\alpha ( \cg \cg^{u \ub}) \cg^{-1} + \cg \cg^{u \ub}  \p_\alpha \cg^{-1} \right) = - \cg^{-1} \T^{\alpha}_{\ub} [\psi] \p_\alpha ( \Psip (u) \p_{\ub} \phi) \\
={}& \cg^{-1}\mathcal{A}(\psi) (u) \Psip (u) \p_{\ub}^2 \phi - \cg^{-1} \left( g^{u u}\partial_{u}\psi\partial_{\ub}\psi   + g^{u \ub}\partial_{\ub}\psi\partial_{\ub}\psi  \right) \p_{u} ( \Psip (u) \p_{\ub} \phi).
}
And
\als{
&  \T^{\alpha}_{\beta} [\psi] \p_\alpha \left( \cg^2 \p^\beta_{\cg} \phi \p^u_{\cg} \phi \right) \\
={}& \mathcal{A}(\psi) \p_u \left( \cg^2 \p^u_{\cg} \phi \p^u_{\cg} \phi \right)  - \mathcal{A}(\psi) \p_{\ub} \left( \cg^2 \p^{\ub}_{\cg} \phi \p^u_{\cg} \phi \right) \\
& +  \left( g^{u \mu}\partial_{\mu}\psi\partial_{\ub}\psi  \right) \p_u \left( \cg^2 \p^{\ub}_{\cg} \phi \p^u_{\cg} \phi \right)  + \left( g^{\ub \mu}\partial_{\mu}\psi\partial_{u}\psi  \right) \p_{\ub} \left( \cg^2 \p^{u}_{\cg} \phi \p^u_{\cg} \phi \right).
}
Compared to $- \T^{\alpha}_{\beta} [\psi] \p_\alpha \left( \cg^2 \p^\beta_{\cg} \phi \p^u_{\cg} \phi \right)  (\cg g)^{-1}$, $- \T^{\alpha}_{\beta} [\psi]  \left( \cg^2 \p^\beta_{\cg} \phi \p^u_{\cg} \phi \right) \p_\alpha (\cg g)^{-1}$ is of lower order,
\als{
&\T^{\alpha}_{\beta} [\psi]  \left( \cg^2 \p^\beta_{\cg} \phi \p^u_{\cg} \phi \right) \p_\alpha (\cg g)^{-1} \\
={}& \mathcal{A}(\psi) \p_u (\cg g)^{-1} \left( \cg^2 \p^u_{\cg} \phi \p^u_{\cg} \phi \right)  - \mathcal{A}(\psi) \p_{\ub}(\cg g)^{-1} \left( \cg^2 \p^{\ub}_{\cg} \phi \p^u_{\cg} \phi \right) \\
& +  \left( g^{u \mu}\partial_{\mu}\psi\partial_{\ub}\psi  \right) \p_u(\cg g)^{-1} \left( \cg^2 \p^{\ub}_{\cg} \phi \p^u_{\cg} \phi \right)  + \left( g^{\ub \mu}\partial_{\mu}\psi\partial_{u}\psi  \right) \p_{\ub}(\cg g)^{-1} \left( \cg^2 \p^{u}_{\cg} \phi \p^u_{\cg} \phi \right).
}

For $\TLb$, the computing is conducted in a similar way.
\begin{align*}
\T^{\alpha}_{\beta} [\psi] \cdot \p_{\alpha} \TLb^{\beta} 
= -  \Lambda^\prime (u) g^{\ub \beta} \T^{u}_{\beta} [\psi] - \Lmdu \p_\alpha g^{\ub \beta} \T^{\alpha}_{\beta} [\psi],
\end{align*}
and \[g^{\ub \beta} \T^{u}_{\beta} [\psi] = g^{u \beta} \T^{\ub}_{\beta} = g^{\ub \ub} \p_{\ub} \psi g^{u u} \p_u \psi + \frac{1}{2} g^{u \ub} \left( g^{u u} (\p_u \psi)^2 + g^{\ub \ub} (\p_{\ub} \psi)^2 \right),\]
and
\als{
\T^{\alpha}_{\beta} [\psi]  \p_\alpha g^{\ub \beta} 
=
&B_1 + B_2- \T^{\alpha}_{\beta} [\psi] \p_\alpha \left( \cg^2 \p^\beta_{\cg} \phi \p^{\ub}_{\cg} \phi \right)  (\cg g)^{-1}  - \T^{\alpha}_{\beta} [\psi]  \left( \cg^2 \p^\beta_{\cg} \phi \p^{\ub}_{\cg} \phi \right) \p_\alpha (\cg g)^{-1},
}
with
\als{
 B_1 ={}& \T^{\alpha}_{u} [\psi] \left( \p_\alpha ( \cg \cg^{\ub u}) \cg^{-1} + \cg \cg^{\ub u}  \p_\alpha \cg^{-1} \right), \\
  B_2 ={}& \T^{\alpha}_{\ub} [\psi] \left( \p_\alpha ( \cg \cg^{\ub \ub}) \cg^{-1} +   \cg \cg^{\ub \ub}  \p_\alpha \cg^{-1} \right).
}
Note that, by \eqref{simplify-1},
\als{
B_1  ={}&  - \cg^{-1} \T^{\alpha}_{u} [\psi] \p_\alpha ( \Psip (u) \p_{\ub} \phi),   \\
={}&- \cg^{-1} \mathcal{A}(\psi)  \p_{u} ( \Psip (u) \p_{\ub} \phi) - \cg^{-1} \left( g^{\ub u}\partial_{u}\psi\partial_{u}\psi + g^{\ub \ub}\partial_{\ub}\psi\partial_{u}\psi \right)  \Psip (u) \p^2_{\ub} \phi.
}
And as for $B_2$, we first calculate,
\als{
&  \p_\alpha ( \cg \cg^{\ub \ub}) \cg^{-1} + \cg \cg^{\ub \ub}  \p_\alpha \cg^{-1} \\
={}& -2 \p_\alpha ( \Psip (u) \p_{u} \phi ) \cg^{-1} + 2 \cg^{-2} \Psip (u) \p_u \phi \p_\alpha (\Psip (u) \p_{\ub} \phi -1)^2,
}
and then
\als{
B_2 ={}& \T^{\alpha}_{\ub} [\psi] \left( -2 \p_\alpha ( \Psip (u) \p_{u} \phi ) \cg^{-1}  + 4 \cg^{-2} (\Psip (u))^2 \p_u \phi \p_{\ub} \phi  \p_\alpha (\Psip (u) \p_{\ub} \phi )  \right) \\
 & + \T^{\alpha}_{\ub} [\psi] \left(  - 4 \cg^{-2} \Psip (u) \p_u \phi \p_\alpha (\Psip (u) \p_{\ub} \phi ) \right) \\
= {}& 2 \cg^{-1} \mathcal{A}(\psi) \p_{\ub} ( \Psip (u) \p_{u} \phi) - 2 \cg^{-1}  (g^{u \mu}\partial_{\mu}\psi\partial_{\ub}\psi ) \p_{u} ( \Psip (u) \p_{u} \phi)  \\
 {}&- 4 \cg^{-2} (\Psip (u))^2 \p_u \phi \p_{\ub} \phi \mathcal{A}(\psi) \p_{\ub} ( \Psip (u) \p_{u} \phi) + 4 \cg^{-2} (\Psip (u))^2 \p_u \phi \p_{\ub} \phi   (g^{u \mu}\partial_{\mu}\psi\partial_{\ub}\psi ) \p_{u} ( \Psip (u) \p_{u} \phi) \\
& + 4 \cg^{-2} \mathcal{A}(\psi)  \Psip (u) \p_u \phi  \p_{\ub} ( \Psip (u) \p_{\ub} \phi) - 4 \cg^{-2}  (g^{u \mu}\partial_{\mu}\psi\partial_{\ub}\psi )  \Psip (u) \p_u \phi  \p_{u} ( \Psip (u) \p_{\ub} \phi).
}
And
\als{
 &\T^{\alpha}_{\beta} [\psi] \p_\alpha \left( \cg^2 \p^\beta_{\cg} \phi \p^{\ub}_{\cg} \phi \right) \\
={}& \mathcal{A}(\psi) \p_u \left( \cg^2 \p^u_{\cg} \phi \p^{\ub}_{\cg} \phi \right) -\mathcal{A}(\psi) \p_{\ub} \left( \cg^2 \p^{\ub}_{\cg} \phi \p^{\ub}_{\cg} \phi \right) \\
& +  \left( g^{u \mu}\partial_{\mu}\psi\partial_{\ub}\psi  \right) \p_u \left( \cg^2 \p^{\ub}_{\cg} \phi \p^{\ub}_{\cg} \phi \right)  + \left( g^{\ub \mu}\partial_{\mu}\psi\partial_{u}\psi  \right) \p_{\ub} \left(\cg^2 \p^{u}_{\cg} \phi \p^{\ub}_{\cg} \phi \right).
}
Besides, compared to $ - \T^{\alpha}_{\beta} [\psi] \p_\alpha \left( \cg^2 \p^\beta_{\cg} \phi \p^{\ub}_{\cg} \phi \right)  (\cg g)^{-1}$, $ - \T^{\alpha}_{\beta} [\psi]  \left( \cg^2 \p^\beta_{\cg} \phi \p^{\ub}_{\cg} \phi \right) \p_\alpha (\cg g)^{-1}$ is of lower order,
\als{
 &\T^{\alpha}_{\beta} [\psi] \p_\alpha (\cg g)^{-1} \left( \cg^2 \p^\beta_{\cg} \phi \p^{\ub}_{\cg} \phi \right) \\
={}& \mathcal{A}(\psi) \p_u(\cg g)^{-1} \left( \cg^2 \p^u_{\cg} \phi \p^{\ub}_{\cg} \phi \right) -\mathcal{A}(\psi) \p_{\ub}(\cg g)^{-1} \left( \cg^2 \p^{\ub}_{\cg} \phi \p^{\ub}_{\cg} \phi \right) \\
& +  \left( g^{u \mu}\partial_{\mu}\psi\partial_{\ub}\psi  \right) \p_u(\cg g)^{-1} \left( \cg^2 \p^{\ub}_{\cg} \phi \p^{\ub}_{\cg} \phi \right)  + \left( g^{\ub \mu}\partial_{\mu}\psi\partial_{u}\psi  \right) \p_{\ub}(\cg g)^{-1} \left(\cg^2 \p^{u}_{\cg} \phi \p^{\ub}_{\cg} \phi \right).
}

We next calculate
\begin{align*}
& \frac{1}{\sqrt{g}}  \p_\alpha (\sqrt{g}g^{\gamma\rho}) \p_{\gamma}\phi_k \p_{\rho}\phi_k =  \frac{\p_\alpha g}{2 g}  g^{\gamma\rho} \p_{\gamma}\phi_k \p_{\rho}\phi_k +  \p_\alpha  g^{\gamma\rho} \p_{\gamma}\phi_k \p_{\rho}\phi_k \\
= &{}  \frac{\p_\alpha g}{2g} g^{\gamma\rho} \p_{\gamma}\phi_k \p_{\rho}\phi_k -  \frac{2 }{ \cg g}  \p_\alpha (\cg \p^\gamma_{\cg} \phi) \cg \p^\rho_{\cg} \phi \p_{\gamma}\phi_k \p_{\rho}\phi_k + \p_\alpha \cg^{\gamma \rho} \p_\gamma \phi_k \p_\rho \phi_k \\
&+ \frac{\cg}{g^2}   \p_\alpha g \p^\gamma_{\cg} \phi \p^\rho_{\cg} \phi \p_{\gamma}\phi_k \p_{\rho}\phi_k  + \frac{1}{g}   \p_\alpha \cg   \p^\gamma_{\cg} \phi \p^\rho_{\cg} \phi \p_{\gamma}\phi_k \p_{\rho}\phi_k  \\
= &{}  \frac{\p_\alpha g}{2g} \cg^{\gamma\rho} \p_{\gamma}\phi_k \p_{\rho}\phi_k -  \frac{2}{\cg g}  \p_\alpha (\cg \p^\gamma_{\cg} \phi) \cg \p^\rho_{\cg} \phi \p_{\gamma}\phi_k \p_{\rho}\phi_k +  \p_\alpha \cg^{\gamma \rho} \p_\gamma \phi_k \p_\rho \phi_k  \\
&+ \frac{\cg}{2g^2}   \p_\alpha g \p^\gamma_{\cg} \phi \p^\rho_{\cg} \phi \p_{\gamma}\phi_k \p_{\rho}\phi_k + \frac{1}{g}   \p_\alpha \cg   \p^\gamma_{\cg} \phi \p^\rho_{\cg} \phi \p_{\gamma}\phi_k \p_{\rho}\phi_k .
\end{align*}
Expanding the following terms
\als{
 & \qquad \frac{\p_\alpha g}{2g} \cg^{\gamma\rho} \p_{\gamma}\phi_k \p_{\rho}\phi_k -  \frac{2}{\cg g}  \p_\alpha (\cg \p^\gamma_{\cg} \phi) \cg \p^\rho_{\cg} \phi \p_{\gamma}\phi_k \p_{\rho}\phi_k = -  \frac{1}{\cg g}  \p_\alpha g \Psip (u) \p_u \phi (\p_{\ub} \phi_k)^2 \\
 &+  \frac{2}{\cg g} \p_\alpha ( \p_u \phi \p_{\ub} \phi (\Psip (u) \p_{\ub} \phi)^2 )  \Psip (u) \p_{\ub} \phi  \p_{u} \phi_k \p_{\ub} \phi_k - \frac{2}{\cg g} \p_\alpha ( \p_u \phi \p_{\ub} \phi)  \Psip (u) \p_{\ub} \phi  \p_{u} \phi_k \p_{\ub} \phi_k\\
 & - \frac{1}{\cg g}  \p_\alpha ( ( \p_{\ub} \phi)^2 (1-\Psip (u) \p_{\ub} \phi)^2 )  (\p_{u} \phi_k)^2  - \frac{1}{\cg g}  \p_\alpha  ( ( \p_{u} \phi)^2 (1+ \Psip (u) \p_{\ub} \phi)^2 ) (\p_{\ub} \phi_k)^2,
}
and
\[  \p_\alpha \cg^{\gamma \rho} \p_\gamma \phi_k \p_\rho \phi_k = \cg^{-1}  \p_\alpha (\cg \cg^{\gamma \rho}) \p_\gamma \phi_k \p_\rho \phi_k -  \cg^{-1}  \p_\alpha \cg \cg^{\gamma \rho} \p_\gamma \phi_k \p_\rho \phi_k,\]
it then follows that
\begin{align*}
& \frac{1}{\sqrt{g}}  \p_\alpha (\sqrt{g}g^{\gamma\rho}) \p_{\gamma}\phi_k \p_{\rho}\phi_k \\
={}& \cg^{-1}  \p_\alpha (\cg \cg^{\gamma \rho}) \p_\gamma \phi_k \p_\rho \phi_k -  \cg^{-1}  \p_\alpha \cg \cg^{\gamma \rho} \p_\gamma \phi_k \p_\rho \phi_k -  \frac{1}{\cg g}  \p_\alpha g \Psip (u) \p_u \phi (\p_{\ub} \phi_k)^2  \\
 & - \frac{1}{\cg g}  \p_\alpha ( ( \p_{\ub} \phi)^2 (1-\Psip (u) \p_{\ub} \phi)^2 )  (\p_{u} \phi_k)^2  - \frac{1}{\cg g}  \p_\alpha  ( ( \p_{u} \phi)^2 (1+ \Psip (u) \p_{\ub} \phi)^2 ) (\p_{\ub} \phi_k)^2 \\
 & \frac{-2}{\cg g} \p_\alpha ( \p_u \phi \p_{\ub} \phi)  \Psip (u) \p_{\ub} \phi  \p_{u} \phi_k \p_{\ub} \phi_k +  \frac{2}{\cg g} \p_\alpha ( \p_u \phi \p_{\ub} \phi (\Psip (u) \p_{\ub} \phi)^2 )  \Psip (u) \p_{\ub} \phi  \p_{u} \phi_k \p_{\ub} \phi_k\\
&+ \frac{\cg}{2g^2}   \p_\alpha g \p^\gamma_{\cg} \phi \p^\rho_{\cg} \phi \p_{\gamma}\phi_k \p_{\rho}\phi_k + \frac{1}{g}   \p_\alpha \cg   \p^\gamma_{\cg} \phi \p^\rho_{\cg} \phi \p_{\gamma}\phi_k \p_{\rho}\phi_k.
\end{align*}
Substituting this formula into \[- \Lmdub \frac{1}{\sqrt{g}} g^{u \alpha} \p_\alpha (\sqrt{g}g^{\gamma\rho}) \p_{\gamma}\phi_k \p_{\rho}\phi_k, \quad - \Lmdu \frac{1}{\sqrt{g}} g^{\ub \alpha} \p_\alpha (\sqrt{g}g^{\gamma\rho}) \p_{\gamma}\phi_k \p_{\rho}\phi_k, \]
we complete the proof.

\end{proof}

\section{Divergence Theorem}\label{subsection:2.5}
\begin{theorem}[Divergence theorem]
Suppose $V$ is a subset of $R^{n}$, which is compact and has a piecewise smooth boundary $S$ (also indicated with $\partial V=S$). If $F$ is a continuously differentiable vector field defined on a neighborhood of $V$, then we have
\bes
\int_{V}(\nabla\cdot F)dV=\oint_{S}(F\cdot n)dS,
\ees
where $n$ is the outward pointing unit normal field of the boundary $\partial V$.
\end{theorem}

\end{document}